\documentclass[11pt]{article}

\usepackage{geometry}
\usepackage{amsmath}
\usepackage{amsthm}
\usepackage{amsfonts}
\usepackage{amssymb}
\usepackage{cases}

\usepackage{hyperref}

\numberwithin{equation}{section}
\allowdisplaybreaks

\newtheorem{thm}{Theorem}
\newtheorem{lm}{Lemma}
\newtheorem{prop}{Proposition}
\newtheorem{remark}{Remark}

\newcommand{\subeqref}[2]{$ \eqref{#1}_{#2} $}
\newcommand{\norm}[2]{\Arrowvert #1 \Arrowvert_{#2}}
\newcommand{\Lnorm}[1]{L^{#1}}
\newcommand{\abs}[2]{\bigl| #1 \bigr|^{#2}}

\newcommand{\dt}[0]{\partial_t}
\newcommand{\dr}[0]{\partial_r}

\title{On the Self-similar Solutions to Full Compressible Navier-Stokes Equations without Heat Conductivity}

\author{Xin Liu, Yuan Yuan}

\begin{document}

\maketitle

\begin{abstract}
	In this work, we establish a class of globally defined, large solutions to the free boundary problem of compressible Navier-Stokes equations with constant shear viscosity and vanishing bulk viscosity. We establish such solutions with initial data perturbed around any self-similar solution when $ \gamma > 7/6 $. In the case when $ 7/6 < \gamma < 7/3 $, as long as the self-similar solution has bounded entropy, a solution with bounded entropy can be constructed. It should be pointed out that the solutions we obtain in this fashion do not in general keep being a small perturbation of the self-similar solution due to the second law of thermodynamics, i.e., the growth of entropy. If in addition, in the case when $ 11/9 < \gamma < 5/3 $, we can construct a solution as a global-in-time small perturbation of the self-similar solution and the entropy is uniformly bounded in time.
\end{abstract}

\setcounter{tocdepth}{1}
\tableofcontents

\section{Introduction}

\subsection{Equations, background and motivations}

In this work, we are considering the following compressible Navier-Stokes system without heat conductivity for non-isentropic flows in spherically symmetric motions:
\begin{equation}\label{CNS}
	\begin{cases}
		\dt (r^2 \rho) + \dr ( r^2 \rho u) = 0 & r \in (0,R(t)),\\
		\dt (r^2 \rho u) + \dr (r^2 \rho u^2) + r^2 \dr p \\
		~~~~  = (2\mu+\lambda) r^2 \dr \biggl( \dfrac{\dr(r^2u)}{r^2} \biggr) & r \in (0,R(t)),\\
		\dfrac{1}{\gamma-1} \bigl( \dt (r^2 p) + \dr (r^2 p u) \bigr) + p \dr (r^2 u) \\
		~~~~  = 2\mu r^2 \biggl( (\dr u)^2 + 2 (\dfrac{u}{r})^2 \biggr) 
		+ \lambda r^2 \biggl( \dr u + 2 \dfrac{u}{r} \biggr)^2& r \in (0,R(t)),
	\end{cases}
\end{equation}
where $ \rho, u, p $ are the scaler density, the radial velocity and the pressure potential, respectively, and $ R(t) $ is the radius of the evolving domain. $ \mu,  \lambda $ are the viscosity coefficients satisfying the Lam\'e relation
$$ \mu \geq 0, ~ 2\mu + 3 \lambda \geq 0, $$
representing the non-negativity of the shear and bulk viscosities, respectively. Also, $ \gamma > 1 $ is the thermodynamic coefficient of the fluid. In particular, for the ideal gas, the pressure potential $ p $ and the specific inner energy $ e $ can be expressed as, in terms of the temperature $ \theta $ and the density $ \rho $, $$ p = K\rho\theta, ~ e = c_\nu \theta, $$
for some positive constants $ K $ and $ c_\nu $, referred to as the thermodynamic and specific inner energy coefficients, respectively. Then $ \gamma $ is given by, in this case, 
$$ \gamma = 1 + \dfrac{K}{c_\nu} > 1. $$ 
Moreover, denote the entropy $ s $ by
\begin{equation}\label{df:entropy}
s : = c_\nu \log \dfrac{p}{\rho^{\gamma}} + \bar s
\end{equation}
for some arbitrary constant $ \bar s \in \mathbb R $. It is easy to verify, the Gibbs-Helmholtz equation,
\begin{equation*}
ds = \theta^{-1} de - p \rho^{-2} \theta^{-1} d\rho,
\end{equation*}
is satisfied. 

In this work, we consider the flows with only positive shear viscosity while the bulk viscosity vanishes. That it, 
\begin{equation}\label{weak-viscosity}
\mu > 0, ~ 2\mu + 3\lambda = 0.
\end{equation}
In this case, system \eqref{CNS} can be written as,
\begin{equation*}\tag{\ref{CNS}'}\label{CNS'}
\begin{cases}
\dt (r^2 \rho) + \dr ( r^2 \rho u) = 0 & r \in (0,R(t)),\\
\dt (r^2 \rho u) + \dr (r^2 \rho u^2) + r^2 \dr p 
= \dfrac{4}{3}\mu r^2 \dr \biggl( \dfrac{\dr(r^2u)}{r^2} \biggr) \\
~~~~ ~~~~ = \dfrac{4}{3} \mu r^2 \biggl( \dr u - \dfrac{u}{r} \biggr)_r + 4\mu r^2 \biggl( \dfrac{u}{r} \biggr)_r & r \in (0,R(t)),\\
\dfrac{1}{\gamma-1} \bigl( \dt (r^2 p) + \dr (r^2 p u) \bigr) + p \dr (r^2 u) \\
~~~~ ~~~~ = \dfrac{4}{3}\mu r^2 \biggl( \dr u - \dfrac{u}{r} \biggr)^2 & r \in (0,R(t)),
\end{cases}
\end{equation*}
Also, system \eqref{CNS} is complemented with the following boundary conditions
\begin{equation}\label{bc}
\begin{gathered}
p - \dfrac{4}{3} \mu \biggl(\dr u - \dfrac{u}{r} \biggr) \big|_{r = R(t)} = 0,\\
u|_{t= 0} = 0,\\
R'(t) = u(R(t),t),
\end{gathered}
\end{equation}
representing the free surface stress tension, the zero central velocity, and the evolution of the free boundary, respectively.

There have been a huge number of literatures concerning the compressible Navier-Stokes equations (CNS). It surely will be too ambitious to mention all the important works. Instead, we will mention some of those closely related to our problem. For instance, the Cauchy and first initial boundary value problems have been studied widely. In the absence of vacuum (i.e. $ \rho \geq \underline\rho > 0 $), the local well-posedness of classical solutions has been investigated by Serrin \cite{Serrin1959}, Itaya \cite{Itaya1971}, Tani \cite{Tani1977,Tani1986}. The poineering works of Matsumura and Nishida \cite{Matsumura1980,Matsumura1983} showed the global stability of equilibria for the heat conductive flows with respect to small perturbations. On the other hand, if the heat conductivity is taken away from the equations, the entropy wave is no longer  dissipative, even though the acoustic wave dissipates thanks to the viscosity. In such a scenario, Liu and Zeng in \cite{LiuZeng1999} studied the large time behavior of such a system of composite type with initial data closed to a constant state. 

When the density profile contains vacuum states (i.e. $ \rho \geq 0$), Cho, Choe, Kim \cite{Cho2004,Cho2006a,Cho2006c} showed the local well-posedness of CNS for isentropic and heat conductive flows. With small initial energy, Huang, Li, Xin \cite{HuangLiXin2012,Huang2018} established the global well-posedness for the isentropic and heat conductive flows. However, These solutions fail to track the entropy in the vacuum area. In fact, as pointed out by Xin and Yan \cite{zpxin1998,XinYan2013}, the classical solutions to non-heat-conductive CNS with bounded entropy will blow up in finite time due to the appearance of vacuum states. Also, as pointed out in \cite{Liu1998a} by Liu, Xin, Yang, the vacuum states for CNS may not produce physically desirable solutions. More recently, Li, Wang, Xin \cite{Li2017b} showed that with vacuum states, the classical solutions to CNS do not exist with finite entropy. 

Motivated by the studies mentioned above, in order to establish a solution with bounded entropy and vacuum states, we are working on the free boundary problem of CNS. Before moving on to our works, it is worth mentioning some previous works, if not all of them, in the following. When the density connects to the vacuum area on the moving boundary with a jump, the local well-posedness theory and the global stability of equilibria can be tracked back to Solonnikov, Tani, Zadrzy\'nska, and Zaj\c{a}czkowski, \cite{Solonnikov1992,Zajaczkowski1993,Zadrzynska1994,Zajaczkowski1994,Zajaczkowski1995,Zajaczkowski1998,Zajaczkowski1999,Zadrzynska2001}. On the other hand, when the density profile connects continuously to vacuum across the moving boundary, the degeneracy of the density causes singularities when establishing derivative estimates. This has been pointed out by Liu \cite{Liu1996} for inviscid flows with physical vacuum. With weighted energy estimates, Jang, Masmoudi, Coutand, Lindblad, Shkoller established the local well-posedness in \cite{Jang2010a,Jang2009b,Coutand2012a,Coutand2011a,Coutand2010} for such flows. See also \cite{Gu2012,Gu2015,LuoXinZeng2014,Hadzic2016a,Hadzic2016}. In the case of viscous flows, the extra viscosities bring more regularity estimates to the velocity field. Consequently, the degeneracy is comparably causing fewer troubles. See \cite{Luo2000a,ZengHH2015,XL2018,Jang2010} for the isentropic flows. We emphasize that the physical vacuum profiles mentioned here commonly exist in a lot of physical models, such as, the gaseous star problem, the Euler damping equations, etc. We refer these results to \cite{LuoXinZeng2015,LuoXinZeng2016,ZengHH2015a,ZengHH2017}.

However, the studies of free boundary problems mentioned above mainly concern the isentropic flows. We start the study of free boundary problem for non-isentropic flows by studying the equilibria of the radiation gaseous stars in \cite{XL2017}, in which we establish the corresponding degeneracy of density and temperature near the vacuum boundary. Also, we establish the local well-posedness with such degenerate profiles in \cite{XL2017,LY2018}. 

This work is part of our project on studying the large time dynamics of flows with bounded entropy in the setting of free boundary problems. In particular, as a starting point, we are investigating the flows without heat conductivity in spherically symmetric motions. From \eqref{df:entropy}, in order to establish a solution to \eqref{CNS} with bounded entropy and density connecting to vacuum continuously, it is necessary for the pressure potential $ p $ to vanish on the moving boundary. In particular, by imposing the vanished bulk viscosity in \eqref{weak-viscosity}, the terms on the right of \subeqref{CNS}{3} vanish together with the pressure (see, i.e., \subeqref{CNS'}{3} and \eqref{bc}). In particular, the choice of the viscosity coefficients in \eqref{weak-viscosity} is indicated by the kinetic theory (see, i.e., \cite[pp.3, (1.11)]{Lions1996}). Another benefit given by \eqref{weak-viscosity} is that one can construct self-similar solutions to \eqref{CNS'} as shown later, similar to those of inviscid flows in \cite{Sideris2012}, and the self-similar solutions have non-growing, hence possibly bounded, entropy. However, this comes with a price. As one can verify easily, the viscosity tensor only guarantee the dissipation of $ \dr u - \dfrac{u}{r} $ in system \eqref{CNS'}, which is not enough, in general, to obtain sufficient regularity for $ u $. We notice, from the work of Had\v{z}i\'c and Jang in \cite{Hadzic2016a} (also \cite{Hadzic2016}), by linearizing \subeqref{CNS'}{2} near a self-similar solution in the Lagrangian coordinates, one can obtain a viscous wave equation with extra damping, which gives us the missing estimate of $ u $. We have made use of such a structure in \cite{XL2018a} to deal with radiation gaseous stars with isentropic or heat conductive models. In this work, we will continue exploring this structure, coupling with non-decreasing entropy. We emphasize here, the non-decreasing entropy in \eqref{CNS'} is the consequence of the hyperbolic transport equation \subeqref{CNS'}{3} with the source from the viscous friction. This corresponds to the second law of thermodynamics, and brings us the main difficulty in establishing the entropy-bounded solution of \eqref{CNS'}. In particular, to carefully track the growing of the entropy and its coupling with the momentum is the main obstacle in this work. 

\subsection{Self-similar solutions}

Our goal is to investigate the self-similar solutions to system \eqref{CNS'}. In order to do so, consider the ansatz:
\begin{equation}\label{ansatz}
\begin{cases}
{r} = {r}(x,t) = \alpha(t) x, \\
{u}({r},t) = {u}(x,t) = \alpha'(t) x, \\
{\rho}({r},t) = {\rho}(x,t) = \alpha^{-3}(t) \overline{\rho}(x),\\
p(r,t) = p(x,t) = \beta(t) \overline{p}(x),
\end{cases}
\end{equation}
where $ \alpha, \beta $ are positive functions of $ t \in [0,\infty) $, and $ \overline{\rho}, \overline{p} $ are non-negative functions of $ x \in [0,1] $. Then after substituting \eqref{ansatz} into \eqref{CNS'} and \eqref{bc}, $ \alpha, \beta, \overline{\rho}, \overline{p} $ satisfy the equations, 
\begin{equation*}
\begin{gathered}
\overline{p}\big|_{x=1} = 0,\\
\alpha'' x \overline{\rho} + \alpha^2 \beta (\overline{p})_x = 0,\\
\alpha\beta' + 3 \gamma \alpha'\beta=0.
\end{gathered}
\end{equation*}
Consequently, without loss of generality, one can take 
$ \beta = \alpha^{-3\gamma} $.
Therefore, the self-similar solutions in the form of \eqref{ansatz} are determined by the following system:
\begin{equation}\label{eq:self-similar-sol}
\begin{cases}
\alpha^{3\gamma-2} \alpha'' = \delta, \\
\delta x \overline{\rho} + (\overline{p})_x = 0,\\
\overline{p}\big|_{x=1}=0, ~ \beta = \alpha^{-3\gamma},
\end{cases}
\end{equation}
for some constant $ \delta \in \mathbb{R}^+ $, since we are interested in solutions with $ \overline\rho, \overline{p} > 0 $ in $ x \in (0,1) $. For some technical reason, we assume further that
\begin{equation}\label{density:boundary}
	\overline\rho(x) \simeq (1-x)^\varrho, ~ \text{for some} ~ \varrho \geq 0. 
\end{equation} 
In particular, given any fixed function $ \overline{\rho} \in C[0,1] $ and any fixed positive constant $ \delta \in (0,\infty) $, together with initial data $ (\alpha, \alpha') \big|_{t=0} = ( \alpha_0, \alpha_1) \in (0,\infty)\times (-\infty,\infty)  $, system 
\eqref{eq:self-similar-sol} is globally well-posed and there are constants $ c_1, c_2 $ depending on $ \alpha_0, \alpha_1, \delta $, such that
\begin{equation}\label{self-similar-asymptote}
\sup_{t > 0 }\dfrac{\alpha(t)}{c_1 + c_2 t} = 1.
\end{equation}
Indeed, from \subeqref{eq:self-similar-sol}{1}, one has
\begin{equation*}
\dfrac{1}{2} \bigl( (\alpha')^2 \bigr)' = \dfrac{\delta}{3-3\gamma} \bigl( \alpha^{3-3\gamma} \bigr)'.
\end{equation*}
Therefore, $$ (\alpha')^2 + \dfrac{2\delta}{3\gamma - 3} \alpha^{3-3\gamma} = \alpha_1^2 + \dfrac{2\delta}{3\gamma - 3} \alpha_0^{3-3\gamma} \in (0,\infty), $$
which implies that $ \alpha $ is globally defined and strictly positive. Together with \subeqref{eq:self-similar-sol}{1}, $ \alpha'' $ is strictly positive and for $ T^* $ large enough and $ t \in (T^*,\infty) $, $ \alpha'(t) > 0 $. \eqref{self-similar-asymptote} follows after evaluating the limit as $ t \rightarrow \infty $ using L'H\^{o}pital's rule. It is worth noticing that, the assumption \eqref{density:boundary} allows the density profile to connect to vacuum either with a jump or continuously, and the physical vacuum profile (see \cite{Liu2016}) is allowed. 

Recalled, in the isentropic case (i.e.,taking $ p = \rho^\gamma $), similar arguments as above still hold and the density profile $ \overline{\rho} $ is determined by
$$ \delta x \overline \rho + (\overline{\rho}^\gamma)_x = 0,~ \overline{\rho}\big|_{x=1} = 0, $$
with given total mass $ \int_0^1 \overline{\rho} \,dx = M \in (0,\infty) $. Consequently, in the isentropic case, the density function is fully determined by $ \delta $ and total mass $ M $, which are two numbers. However, in our setting (i.e. non-isentropic), the density profile is fully free provided that it is an integrable non-negative function in $[0,1] $. Instead, the pressure $ \overline{p} $ (equivalently, the entropy) is determined by $ \delta $ and $ \overline{\rho} $, which are a number and a function.

Moreover, the entropy for the self-similar solutions of the form \eqref{ansatz} are given by
\begin{equation}\label{ansatz-entropy-0}
s(x,t) = c_\nu \log \dfrac{\overline p}{\overline{\rho}^\gamma} + \bar s,
\end{equation}
which is independent of $ t \in (0,\infty) $ for any fixed $ x \in (0,1) $. Therefore, if $ s $ is bounded initially, it remains so in the coming future. Without loss of generality, one can take 
\begin{equation}\label{ansatz-entropy}
\bar s = - c_\nu \log \dfrac{\overline{p}}{\overline{\rho}^\gamma},
\end{equation}
provided that it is finite, 
and so $ s \equiv 0 $.
In particular, the self-similar solutions in the isentropic case satisfy \eqref{ansatz} and \eqref{ansatz-entropy} with $ s $ being a constant in the space-time domain.

\subsection{Perturbation, Lagrangian formulation and main theorem}

Compared to the isentropic flows and the inviscid flows, the growth of entropy due to the turbulence viscosities, i.e. the last term on the right of \subeqref{CNS'}{3}, contributes the main difficulty in studying the dynamics of system \eqref{CNS}. In fact, such a contribution to the entropy is nothing but the second low of thermodynamics in such a system. Even though the vanishing of bulk viscosity in \eqref{weak-viscosity} benefits us in a way  that any self-similar solution given by \eqref{ansatz} and \eqref{eq:self-similar-sol} has non-growing entropy, i.e. \eqref{ansatz-entropy-0}, a perturbation, in general, will yield a non-trivial shear viscosity which will make a contribution to the growth of entropy. Consequently, 
there is no guarantee that the perturbation, no matter how small it is initially, of this self-similar solution, will decay, or keep small, globally in time. Indeed, \subeqref{CNS'}{3} is a hyperbolic equation of $ p $ with sources, and no obvious damping, nor dissipation. 
Our goal in this work is to study the global-in-time stability of the self-similar solutions given by \eqref{ansatz} and \eqref{eq:self-similar-sol}, above. 

On the other hand, as discussed, the self-similar solutions of the form \eqref{ansatz} are fully determined by system \eqref{eq:self-similar-sol}. In particular, with any one of the non-negative functions $  \overline{\rho}, \overline{p} $ and a fixed constant $ \delta \in (0,\infty) $, system \eqref{eq:self-similar-sol} can be reduced to an ODE of $ \alpha $, i.e. \subeqref{eq:self-similar-sol}{1}. Moreover, \eqref{ansatz-entropy} implies that any one of $ \overline{\rho}, \overline{p} $ can be replaced by $ \bar s $. For this reason, we say the self-similar solutions of the form \eqref{ansatz} are labeled by $ \lbrace \delta, \overline{\rho} \rbrace $, where $ \overline \rho $ can be replaced by $ \overline{p} $ or $ \bar s $. Then our stability problem can be phrased as follows:
\begin{center}
	is the manifold of self-similar solutions of the form \eqref{ansatz} stable ? 
\end{center}
Suppose the answer is yes. Then we should have the following equivalent description:
\begin{center}
	a small perturbation of a self-similar solution of the form \eqref{ansatz} will result in a relabeled self-similar solution as the asymptote as time grows up.
\end{center}
We demonstrate this statement when $ \frac{11}{9} < \gamma < \frac{5}{3} $ in the sense that the perturbation variables (defined in \eqref{pert-variable-01}, below) have limits as $ t \rightarrow \infty $. Indeed, in our rescaled time-coordinate $ \tau $ (i.e., \eqref{rescaled-time}, below), $ \eta_\tau, x\eta_{x\tau} $ admit a decay rate of $ \alpha^{-\sigma_1} \simeq e^{-\sigma_1 \alpha_1 \tau} $ (see, \eqref{thmest:pointwise-1}, below), and hence $ \eta, x\eta_x $ admit limits as $ \tau \rightarrow \infty $. On the other hand, equation \eqref{eq:perturbation} implies that  $ (1 + q)(1+\eta)^{2\gamma}(1+\eta+x\eta_x)^\gamma  $ is non-decreasing. Together with the uniform boundedness of $ q $ in  \eqref{thmest:pointwise-q-1}, below, this implies that the limit of $ q $ as $ \tau \rightarrow \infty $ exists. Unfortunately, we have not verified that the limit is actually a solution to \eqref{CNS'}.

Next, we define the perturbation variables with respect to a self-similar solution of the form \eqref{ansatz}. Let $ \eta = \eta(x,t), q = q(x,t) $ be functions satisfying
\begin{equation}\label{pert-variable-01}
r(x,t) = (1+\eta(x,t)) \alpha(t) x, ~
p(r,t) = (1+q(x,t)) \alpha(t)^{-3\gamma} \overline{p}(x),
\end{equation}
where $ r $ is given by the conservation of mass, i.e.,
\begin{equation}\label{pert-variable-02}
\int_0^r s^2 \rho(s, t)  \,ds = \int_0^{\alpha x} s^2 \alpha^{-3} \overline{\rho}(\alpha^{-1} s)  \,ds = \int_0^x y^2 \overline{\rho}(y) \,dy.
\end{equation}
Then after taking spatial and temporal derivatives of \eqref{pert-variable-02}, together with \eqref{pert-variable-01}, one can derive
\begin{equation}\label{pert-variable-03}
\begin{cases}
r(x,t) = (1+\eta(x,t)) \alpha(t) x, \\
\rho(r,t) = \rho(x,t) = \dfrac{x^2 \overline{\rho}(x)}{r^2 r_x} = (1+\eta)^{-2}(1+\eta + x\eta_x)^{-1} \alpha^{-3}(t) \overline \rho(x), \\
\dt r = u(r,t) = u((1+\eta(x,t)) \alpha(t) x, t),\\
p(r,t) = p(x,t) = (1+q(x,t)) \alpha(t)^{-3\gamma} \overline{p}(x).
\end{cases}
\end{equation}
After substituting \eqref{pert-variable-03} to \eqref{CNS'} and applying the change of coordinates $ (r,t) \leadsto (x,t) $, in terms of $ (\eta, q) $, system \eqref{CNS'} and boundary conditions \eqref{bc} are written as:
\begin{equation}\label{eq:perturbation-pre}
\begin{cases}
\dfrac{\overline{\rho}}{(1+\eta)^2} x \biggl( \alpha \dt^2 \eta + 2 \alpha_t \dt \eta \biggr) + \biggl( \dfrac{1}{1+\eta} - 1 - q \biggr) \alpha^{2-3\gamma} \delta \overline \rho x \\
~~~~ + \alpha^{2-3\gamma} q_x \overline p = \dfrac{4}{3} \mu \alpha^2 \biggl( \dfrac{\eta_t + x \eta_{xt}}{1+\eta + x \eta_x} - \dfrac{\eta_t}{1+\eta} \biggr)_x + 4\mu \alpha^2 \biggl( \dfrac{\eta_t}{1+\eta} \biggr)_x, \\
\overline p \dt q + \gamma \overline{p} (1+q) \biggl( \dfrac{\eta_t+x\eta_{xt}}{1+\eta+x\eta_x} + 2 \dfrac{\eta_t}{1+\eta} \biggr)\\
~~~~ = \dfrac{4}{3} \mu (\gamma-1) \alpha^{3\gamma} \biggl(\dfrac{\eta_t + x \eta_{xt}}{1+\eta+x\eta_x} - \dfrac{\eta_t}{1+\eta} \biggr)^2, 
\end{cases}
\end{equation}
$ x \in (0,1), t \in (0,\infty) $, 
with boundary condition
\begin{equation}\label{bc:perturbation-pre}
\biggl(\dfrac{\eta_t+x\eta_{xt}}{1+\eta+x\eta_x} - \dfrac{\eta_t}{1+\eta}\biggr) \big|_{x=1} = 0,
\end{equation}
where $ \delta, \alpha, \overline{\rho}, \overline{p} $ satisfy \eqref{eq:self-similar-sol}. 
In addition, we introduce the following change of variable in time,
\begin{equation}\label{rescaled-time}
\tau = \tau(t) := \int_0^t \dfrac{1}{\alpha(\sigma)} \,d\sigma.
\end{equation}
We use the same notations for the functions $ \alpha, \eta, q $ in the new coordinates $ (x,\tau) $ as in the coordinates $ (x,t) $. 

Then, system \eqref{eq:perturbation-pre} and boundary condition \eqref{bc:perturbation-pre} are written as, in the new coordinates:
\begin{equation}\label{eq:perturbation}
\begin{cases}
\dfrac{x\overline{\rho}}{(1+\eta)^2} \biggl( \alpha \partial_\tau^2 \eta + \alpha_\tau \partial_\tau \eta \biggr) + \biggl( \dfrac{1}{1+\eta} - 1 -q \biggr) \alpha^{4-3\gamma}\delta x \overline{\rho}\\
~~~~ + \alpha^{4-3\gamma} \overline{p} q_x  = \dfrac{4}{3} \mu \alpha^3 \biggl( \dfrac{\eta_\tau + x\eta_{x\tau}}{1+\eta+x\eta_x} - \dfrac{\eta_\tau}{1+\eta} \biggr)_x + 4\mu \alpha^3 \biggl( \dfrac{\eta_\tau}{1+\eta} \biggr)_x,\\
\overline{p} q_\tau + \gamma \overline{p} (1+q) \biggl( \dfrac{\eta_\tau + x\eta_{x\tau}}{1+\eta + x\eta_x} +2  \dfrac{\eta_\tau}{1+\eta} \biggr) \\
~~~~ =  \dfrac{4}{3} \mu (\gamma-1) \alpha^{3\gamma-1} \biggl(\dfrac{\eta_\tau + x\eta_{x\tau}}{1+\eta+x\eta_x} - \dfrac{\eta_\tau}{1+\eta}\biggr)^2,
\end{cases}
\end{equation}
and
\begin{equation}\label{bc:perturbation}
\biggl(\dfrac{\eta_\tau + x\eta_{x\tau}}{1+\eta + x \eta_x} - \dfrac{\eta_\tau}{1+\eta}\biggr) \big|_{x=1} = 0,
\end{equation}
where, recalled, $ \delta x \overline{\rho} = - (\overline{p})_x $.

Also, from \subeqref{eq:self-similar-sol}{1}, $ \alpha(\tau) $ satisfies
\begin{equation}\label{eq:expanding}
\alpha^{3\gamma - 4} \partial_\tau^2 \alpha - \alpha^{3\gamma - 5} (\partial_\tau \alpha)^2 = \delta,
\end{equation}
or equivalently, noticing $ \alpha_\tau \big|_{\tau = 0} = (\alpha \alpha_t) \big|_{t=0} = \alpha_0 \alpha_1 $, after multiplying the above equation with $ 2\alpha^{2-3\gamma} \alpha_\tau $ and integrating the resultant in $ \tau $, 
\begin{equation*}
\dfrac{\alpha_\tau^2}{\alpha^2} + \dfrac{2\delta}{3\gamma-3} \alpha^{3-3\gamma} = \alpha_1^2 + \dfrac{2\delta}{3\gamma-3} \alpha_0^{3-3\gamma}.
\end{equation*}
Consequently, as $ \tau \rightarrow \infty $, 
$$
\dfrac{\alpha_\tau}{\alpha} \rightarrow \biggl(\alpha_1^2 + \dfrac{2\delta}{3\gamma-3} \alpha_0^{3-3\gamma} \biggr)^{\frac{1}{2}}.
$$
In particular, if $ \alpha_1 > 0 $ (i.e., the flow is expanding starting from the beginning and $ \alpha $ is monotonically increasing), we have
\begin{equation*}
\alpha_1 \leq  \dfrac{\alpha_\tau}{\alpha} \leq \biggl(\alpha_1^2 + \dfrac{2\delta}{3\gamma-3} \alpha_0^{3-3\gamma} \biggr)^{\frac{1}{2}},
\end{equation*}
and
\begin{equation}\label{expandingrate}
\alpha_0 e^{\beta_1 \tau} \leq \alpha \leq \alpha_0 e^{\beta_2 \tau}, ~ \beta_1 \alpha \leq \alpha_\tau \leq \beta_2 \alpha,
\end{equation}
where $ \beta_1 = \alpha_1, \beta_2 = \biggl(\alpha_1^2 + \dfrac{2\delta}{3\gamma-3} \alpha_0^{3-3\gamma} \biggr)^{\frac{1}{2}} $. Therefore, by substituting \eqref{expandingrate} to \eqref{eq:expanding}, one has
\begin{equation}\label{expandingrate-02}
	\alpha_{\tau\tau} \leq \beta_2^2 \alpha + \delta \alpha^{4-3\gamma} \lesssim \beta_2^2 \alpha . 
\end{equation}
Notice, for fixed $ \alpha_0 $,
\begin{equation*}
	\beta_1 \simeq \beta_2 \simeq \alpha_1. 
\end{equation*}

Now we state the main theorem of this work.

\begin{thm}[Main theorem, informal statement]\label{thm:informal}
	Consider $ \gamma > \dfrac{7}{6} $, $ \alpha_0 > 0 $ and $ \alpha_1 > 0 $ large enough. The self-similar solutions given by \eqref{ansatz} and \eqref{eq:self-similar-sol}, with $$ \alpha(0) = \alpha_0, \dfrac{d}{dt}\alpha(0) = \alpha_1, $$ 
	are stable in the following sense: 
	with any small enough perturbation of the self-similar solutions initially, there exists a global strong solution to \eqref{CNS'}. If, in addition, $ \gamma \in (\dfrac{11}{9}, \dfrac{5}{3}) $, there exists a global strong solution to \eqref{CNS'} with uniformly bounded entropy as a small perturbation of an entropy-bounded, self-similar solution. 
\end{thm}
This main theorem will be stated in terms of the perturbation variables in Theorem \ref{thm:full}.

\subsection{Comments and methodology}

One of the motivation of this study is to investigate the evolution of entropy in system \eqref{CNS'}. In particular, with our self-similar solution in the form of \eqref{ansatz}, as long as the entropy given by \eqref{ansatz-entropy-0} is bounded initially, the self-similar solution admits globally bounded entropy. Given such self-similar solution, we are studying the solution to \eqref{CNS'} with perturbations as presented in \eqref{pert-variable-03}. One can easily check the entropy for such a solution is given by, as in \eqref{df:entropy}, 
\begin{equation*}
	s = c_\nu \log \bigl\lbrack (1+q) ( 1+\eta)^{2\gamma} (1+\eta+x\eta_x)^\gamma \bigr\rbrack + c_\nu \log \dfrac{\overline p}{\overline \rho^\gamma} + \overline s. 
\end{equation*}
Consequently, in order to obtain a solution with bounded entropy, we will need, besides the boundedness of entropy of the self-similar solution (i.e.,  $ \log \dfrac{\overline p}{\overline \rho} < \infty $), the boundedness of the quantity
\begin{equation*}
	(1+q) ( 1+\eta)^{2\gamma} (1+\eta+x\eta_x)^\gamma.
\end{equation*}
This will require the integrability of 
$$ \alpha^{3\gamma - 1} \dfrac{\bigl(\dfrac{\eta_\tau + x\eta_{x\tau}}{1+\eta + x \eta_x} - \dfrac{\eta_\tau}{1+\eta}\bigr)^2}{\overline p}, $$
in $ \tau $. Indeed, \subeqref{eq:perturbation}{2} can be written as
\begin{equation}\label{eq:entropy}
\begin{gathered}
\dfrac{d}{d\tau} \biggl\lbrace (1+q) \bigl\lbrack ( 1 + \eta +x \eta_{x}) ( 1 + \eta )^2 \bigr\rbrack^\gamma \biggr\rbrace = \dfrac{4}{3} \mu (\lambda - 1)\bigl\lbrack ( 1 + \eta +x \eta_{x}) ( 1 + \eta )^2 \bigr\rbrack^\gamma \\
\times
\alpha^{3\gamma - 1} \dfrac{\bigl(\dfrac{\eta_\tau + x\eta_{x\tau}}{1+\eta + x \eta_x} - \dfrac{\eta_\tau}{1+\eta}\bigr)^2}{\overline p}.
\end{gathered}
\end{equation}
In the following, we denote
\begin{equation}\label{def:frak-B}
\mathfrak B := \dfrac{\eta_\tau + x\eta_{x\tau}}{1+\eta + x \eta_x} - \dfrac{\eta_\tau}{1+\eta}.
\end{equation}
Then after integrating \subeqref{eq:perturbation}{1} in $ (x, 1) $ for $ x \in (0,1) $, it holds
\begin{equation}\label{ana:001}
\begin{aligned}
& \dfrac{4}{3} \mu \alpha^3 \mathfrak B = - \int_x^1 \dfrac{x\overline{\rho}}{(1+\eta)^2} \biggl(\alpha \partial_\tau^2 \eta + \alpha_\tau \partial_\tau\eta \biggr)\,dx\\
& ~~~~ -\alpha^{4-3\gamma} \int_x^1 \biggl( \dfrac{1}{1+\eta} - 1 -q \biggr)  \delta \overline{\rho} x +  q_x \overline{p} \,dx + 4\mu \alpha^3 \int_x^1 \biggl( \dfrac{\eta_\tau}{1+\eta} \biggr)_x \,dx\\
& ~~ \lesssim \biggl(\int_x^1 \overline{\rho} \,dx \biggr)^{1/2} \biggl( \alpha^2 \int_x^1 x^2 \overline{\rho} \abs{\partial_\tau^2 \eta}{2} \,dx + \alpha_\tau^2 \int_x^1 x^2 \overline{\rho}\abs{\partial_\tau \eta}{2} \,dx \biggr)^{1/2} \\
& ~~~~ + \delta\alpha^{4-3\gamma}\biggl(\int_x^1 \overline{\rho} \,dx \biggr)^{1/2} \biggl( \int_x^1 x^2 \overline{\rho} \abs{\eta}{2}\,dx    \biggr)^{1/2} + \alpha^{4-3\gamma} \overline p(x) q(x) \\
& ~~~~ + \alpha^3 \int_x^1 \abs{ \biggl( \dfrac{\eta_\tau}{1+\eta} \biggr)_x }{} \,dx,
\end{aligned}
\end{equation}
supposed that $ \norm{\eta}{\Lnorm{\infty}}, \norm{x\eta_x}{\Lnorm{\infty}} $ are small enough.
Notice, on the one hand, 
\begin{equation*}
\biggl(\dfrac{\eta_\tau}{1+\eta}\biggr)_x = \dfrac{(1+\eta+x\eta_x)}{x(1+\eta)} \mathfrak B,
\end{equation*}
which vanishes on the boundary $ x=1 $. On the other hand, a direct calculation shows that
\begin{equation*}
\int_x^1 \overline{\rho} \,dx \lesssim \overline{p}(x),
\end{equation*}
for $ x \in (1/2,1) $ due to \eqref{eq:self-similar-sol}. 
Therefore, the inequality above actually gives us some estimate on the quantity
\begin{equation*}
\dfrac{\mathfrak B^2}{\overline{p}},
\end{equation*}
near the boundary $ x= 1 $. Combining inequalities \eqref{ana:001} and  \eqref{eq:entropy} forms an inequality which is like
\begin{equation*}
	\dfrac{d}{dt} q \leq  \alpha^{1-3\gamma} q^2 + A,
\end{equation*}
for some quantity $ A $. 
One can conclude from this inequality, via continuity arguments, the uniform boundedness of $ q $ provided that it is small enough initially, at least near the boundary. 
 However, the term we shorten as $ A $ in this inequality is obtained as, roughly speaking, the product of the square of terms on the right-hand side of \eqref{ana:001} and the growing factor $ \alpha^{3\gamma- 1} $, which seems to be growing quite fast for general $ \gamma $. Consequently, it seems too ambitious to get a perturbed solution with bounded entropy, even the entropy is bounded in the self-similar solution. Nevertheless, with some constraints on $ \gamma $, this can be done as shown in Proposition \ref{prop:boundedness-entropy}.

In general, we observe that the global existence of solutions to \eqref{eq:perturbation} only requires the regularity of the quantity
\begin{equation*}
	\overline p q.
\end{equation*}
In addition, inspired by our previous works \cite{XL2018a} (also by \cite{Hadzic2016a,Hadzic2016}), we impose the boundedness of $ \eta, x\eta_x $ and the decay of $ \eta_\tau, x\eta_{x\tau} $ as a priori assumptions. However, instead of imposing boundedness of $ \overline p q $ directly, we only impose the decay of the quantity $ \mathfrak B $. It turns out that, even $ \mathfrak B \simeq  \eta_\tau \pm x\eta_{x\tau} $ by definition, $ \mathfrak B $ has a better decay than $  \eta_\tau $ and $ x\eta_{x\tau}$ (see \eqref{def:a-priori-assm}). Unfortunately, this decay fails to guarantee the uniform boundedness of $ \overline p q $ (or $ q $) in general. Instead, we track the growth of $ \overline p q $, which is sufficient to establish the a priori energy estimates in sections \ref{sec:energy-estimate} and \ref{sec:interior}. With these estimates, in order to close the a priori assumption as in \eqref{def:a-priori-assm}, below, we shall establish some point-wise estimates. This step is inspired by our previous study in \cite{XL2018} and \cite{XL2017}. In fact, we employ some point-wise estimates in section \ref{sec:point-wise} to get the estimates of $ \eta_\tau, x\eta_{x\tau}, \mathfrak B $ in terms of the energy functionals. We point out that 
the manipulations in \eqref{ene:103} and \eqref{ene:105} are crucial in obtaining the point-wise estimates of $ \eta_\tau, x\eta_{x\tau} $ and $ \mathfrak B $, and hence closing the a priori estimates. 

In the end, we establish the regularity of our solutions in section \ref{sec:regularity}. In this part, with the well-established point-wise estimates, the estimates are relatively straight-forward. We only mention that the $ L^2 $ estimates of $ x\eta_{xx\tau} $ and $ \eta_{x\tau} $ are obtained through the relative entropy, which is introduced in \cite{LuoXinZeng2016}. 

By denoting the intervals of $ \gamma $ for the global existence of solutions, the global uniform boundedness of  $ q $ away from the center $ x = 0 $, and the global uniform boundedness of  $ q $ in the whole domain, as $ I_0, I_1, I_2 $, respectively, we have $ I_2 \subset I_1 \subset I_0 $, as described in Theorem \ref{thm:full}. Remarkably, when $ \gamma \in I_1 $, the $ L^\infty $ bound of $ q $ over the whole domain has a growing estimate as indicated by \eqref{thmfull:estimates-1}, even though we fail to obtain the uniform boundedness. 

Inspired by the analysis above, we consider the unknown $ ( \eta, \zeta ) $ with
\begin{equation}\label{def:weighted-q}
	\zeta := \overline p q \bigl \lbrack (1+\eta+x\eta_x)(1+\eta)^2\bigr\rbrack^\gamma,
\end{equation}
satisfying the system
\begin{equation}\label{eq:perturbation-0}
	\begin{cases}
		\dfrac{x\overline{\rho}}{(1+\eta)^2} \biggl( \alpha \partial_\tau^2 \eta + \alpha_\tau \partial_\tau \eta \biggr) - \alpha^{4-3\gamma} \dfrac{\delta x \overline{\rho} \eta}{1+\eta} \\
		~~~~ ~~~~ + \alpha^{4-3\gamma} \biggl( \dfrac{\zeta}{\bigl \lbrack (1+\eta+x\eta_x)(1+\eta)^2\bigr\rbrack^\gamma} \biggr)_x \\
		~~~~  = \dfrac{4}{3} \mu \alpha^3 \mathfrak B_x + 4\mu \alpha^3 \biggl( \dfrac{\eta_\tau}{1+\eta} \biggr)_x, \\
		\partial_\tau \zeta + \gamma \overline p (1+\eta+x\eta_x)^{\gamma-1}(1+\eta)^{2\gamma} (\eta_\tau + x\eta_{x\tau}) \\
		~~~~ ~~~~ + 2 \gamma \overline p (1+\eta+x\eta_x)^{\gamma}(1+\eta)^{2\gamma-1} \eta_\tau \\
		~~~~ = \dfrac{4}{3} \mu (\gamma - 1) \alpha^{3\gamma-1} \bigl \lbrack (1+\eta+x\eta_x)(1+\eta)^2\bigr\rbrack^\gamma \mathfrak B^2,
	\end{cases}
\end{equation}
where
\begin{equation*}
	\tag{\ref{def:frak-B}}
	\mathfrak B= \dfrac{\eta_\tau + x\eta_{x\tau}}{1+\eta+x\eta_x} - \dfrac{\eta_\tau}{1+\eta},
\end{equation*}
satisfying the boundary condition
\begin{equation*}{\tag{\ref{bc:perturbation}}}
	\mathfrak B\big|_{x=1} = 0
\end{equation*}
We introduce the initial data $ (\eta_0, \eta_1, \zeta_0) $ for \eqref{eq:perturbation-0},
\begin{equation}\label{initial:perturbation}
	\begin{gathered}
	(\eta, \eta_\tau, \zeta)\Big|_{t=0} = (\eta_0, \eta_1, \zeta_0) \in H^2(0,1) \times H^2(0,1) \times  H^1(0,1), \\
	~~ \text{with} ~ \zeta_0 = \overline p q_0 \bigl\lbrack (1+\eta_0 + x\eta_{0,x} )(1+\eta_0)^2  \bigr\rbrack^\gamma,
	 \end{gathered}
\end{equation}
for some $ q_0 $ representing the initial data of $ q $, 
and $ (\eta_{2}, \zeta_1) $ given by,
\begin{equation}\label{initial:temporal-derivative}
	\begin{aligned}
		& \dfrac{x\overline{\rho}}{(1+\eta_0)^2} \biggl( \alpha_0 \eta_2 + \alpha_0 \alpha_1 \eta_1 \biggr) - \alpha_0^{4-3\gamma} \dfrac{\delta x \overline{\rho} \eta_0}{1+\eta_0} \\
		& ~~~~ ~~~~ + \alpha_0^{4-3\gamma} \biggl( \dfrac{\zeta_0}{\bigl \lbrack (1+\eta_0+x\eta_{0,x})(1+\eta_0)^2\bigr\rbrack^\gamma} \biggr)_x \\
		& ~~~~  = \dfrac{4}{3} \mu \alpha_0^3 \biggl( \dfrac{\eta_1 + x\eta_{1,x}}{1+\eta_0 + x\eta_{0,x}} + 2 \dfrac{\eta_1}{1+\eta_0} \biggr)_x, \\
		&  \zeta_1 + \gamma \overline p (1+\eta_0+x\eta_{0,x})^{\gamma-1}(1+\eta_0)^{2\gamma} (\eta_1 + x\eta_{1,x}) \\
		& ~~~~ ~~~~ + 2 \gamma \overline p (1+\eta_0+x\eta_{0,x})^{\gamma}(1+\eta_0)^{2\gamma-1} \eta_1 \\
		& ~~~~ = \dfrac{4}{3} \mu (\gamma - 1) \alpha_0^{3\gamma-1} \bigl \lbrack (1+\eta_0+x\eta_{0,x})(1+\eta_0)^2\bigr\rbrack^\gamma \\
		& ~~~~ ~~~~ \times \biggl( \dfrac{\eta_1 + x\eta_{1,x}}{1+\eta_0 + x\eta_{0,x}} - \dfrac{\eta_1}{1+\eta_0} \biggr)^2.
	\end{aligned}
\end{equation}
In particular, bear in mind that since this work is studying the perturbation $ (\eta, q) $ in \eqref{pert-variable-03} of the self-similar solution given by \eqref{ansatz} and \eqref{eq:self-similar-sol}, we assume that 
\begin{equation}\label{initial:uniform-bound}
	\max\lbrace \norm{\eta_0}{\Lnorm{\infty}},  \norm{x \eta_{0,x}}{\Lnorm{\infty}}, \norm{\eta_1}{\Lnorm{\infty}}, \norm{x \eta_{1,x}}{\Lnorm{\infty}}, \norm{q_0}{\Lnorm{\infty}} \rbrace < \omega,
\end{equation}
for some constant $ \omega \in (0,1) $.

Notice, as a consequence of \eqref{bc:perturbation}, the fact $ \overline p\big|_{x=1} = 0 $, and the transport equation \subeqref{eq:perturbation-0}{2}, we have
$$ \zeta\big|_{x = 1} = \zeta_0\big|_{x=1} = 0, $$
for a regular solution to \eqref{eq:perturbation-0}. In particular, the regularity of the solution in this work is sufficient to guarantee the vanishing of $ \zeta $ on the boundary $ \lbrace x =1 \rbrace $. 

Now we introduce the energy and dissipation functionals that we are going to use in this work. Let $ T \in (0,\infty) $ be an arbitrary positive constant. Also, for any fixed $ r_1, \sigma_1 $ 
satisfying
\begin{equation}\label{constraint-000}
	2 - r_1 < 2 \sigma_1 < r_1 < \min \lbrace 6 \gamma - 6, 2 \rbrace,
\end{equation}
denote
\begin{equation}\label{indices}
\begin{gathered}
	l_1 := 6\gamma - 6 -r_1,  ~ r_2 := r_1 + 2\sigma_1 - 2 \leq r_1, \\
	l_2 := 6\gamma - 6 - r_2 = 6\gamma - 4 -r_1 - 2\sigma_1 , \\
	l_3 := l_1 + 2 = 6\gamma - 4 - r_1, ~ r_3 := r_1, \\
	r_4 := 2 - r_2 = 4 - r_1 - 2 \sigma_1 , \\
	l_4 :=  6\gamma - 6 + r_4 = 6\gamma - 2 - r_1 - 2 \sigma_ 1, \\
	\mathfrak a := r_1 + \sigma_1 + 1, ~ \mathfrak b := r_1. 
\end{gathered}
\end{equation}
Then, inside the whole spatial domain, we define:
\begin{equation}\label{def:functional-01}
	\begin{aligned}
		& \mathcal E_0 (T) : = \sup_{0 \leq \tau \leq T} \biggl\lbrace \alpha^{r_1}(\tau) \norm{x^2 \overline\rho^{1/2} \eta_\tau (\tau)}{\Lnorm{2}}^2 
			+ \alpha^{-l_1}(\tau) \norm{x \zeta(\tau)}{\Lnorm{2}}^2 \\
			& ~~~~+ \alpha^{r_2}(\tau) \norm{x^2 \overline\rho^{1/2} \eta_{\tau\tau} (\tau)}{\Lnorm{2}}^2 
			+ \alpha^{-l_2}(\tau) \norm{x \zeta_\tau(\tau)}{\Lnorm{2}}^2 \\
		& ~~~~ + \norm{x^2 \overline\rho^{1/2} \eta (\tau)}{\Lnorm{2}}^2  \biggr\rbrace , \\
		& \mathcal D_0(T): = \int_0^T \alpha^{r_1-1}(\tau) \alpha_\tau(\tau) \norm{x^2 \overline\rho^{1/2} \eta_\tau (\tau)}{\Lnorm{2}}^2 \,d\tau \\
			& ~~~~ + \int_0^T \alpha^{r_1+2}(\tau) \norm{x\bigl\lbrack (1+\eta) x\eta_{x\tau}(\tau) - x\eta_x \eta_\tau(\tau) \bigr\rbrack}{\Lnorm{2}}^2 \,d\tau\\
			& ~~~~ + \int_0^T \alpha^{r_2-1}(\tau) \alpha_\tau(\tau) \norm{x^2 \overline\rho^{1/2} \eta_{\tau\tau} (\tau)}{\Lnorm{2}}^2\,d\tau \\
			& ~~~~ + \int_0^T \alpha^{r_2+2}(\tau) \norm{x\bigl\lbrack (1+\eta) x\eta_{x\tau\tau}(\tau) - x\eta_x \eta_{\tau\tau}(\tau) \bigr\rbrack}{\Lnorm{2}}^2  \,d\tau \\
			& ~~~~ + \int_0^T \alpha^{-l_1 - 1}(\tau) \alpha_\tau(\tau) \norm{x \zeta(\tau)}{\Lnorm{2}}^2\,d\tau \\
			& ~~~~ + \int_0^T \alpha^{-l_2 - 1}(\tau) \alpha_\tau (\tau) \norm{x \zeta_\tau(\tau)}{\Lnorm{2}}^2\,d\tau.
	\end{aligned}
\end{equation}
To define the interior energy and dissipation functionals, let the interior cut-off function be
\begin{equation}\label{def:cut-off}
	\chi(x) := \begin{cases}
		1 & 0 \leq x\leq 1/2, \\
		0 & 3/4 \leq x \leq 1,
	\end{cases}
\end{equation}
with $ - 8 \leq \chi' \leq 0 $. Then the interior energy and dissipation functionals are given by
\begin{equation}\label{def:functional-02}
	\begin{aligned}
		& \mathcal E_{1}(T) := \sup_{0\leq \tau \leq T} \biggl\lbrace 
		\alpha^{-l_3}(\tau) \norm{\chi^{1/2} \zeta(\tau)}{\Lnorm{2}}^2 + \alpha^{-r_4}(\tau) \norm{\chi^{1/2} x \overline \rho^{1/2} \eta_{\tau\tau}(\tau)}{\Lnorm{2}}^2 \\
		& ~~~~ + \alpha^{-l_4}(\tau) \norm{\chi^{1/2} \zeta_\tau(\tau)}{\Lnorm{2}}^2
		\biggr\rbrace, \\
		& \mathcal D_{1}(T) := \int_0^T \alpha^{-l_3 - 1}(\tau) \alpha_\tau(\tau) \norm{\chi^{1/2} \zeta(\tau)}{\Lnorm{2}}^2 \,d\tau \\
		& ~~~~ + \int_0^T \alpha^{r_3}(\tau) \bigl( \norm{\chi^{1/2} \eta_\tau(\tau)}{\Lnorm{2}}^2 + \norm{\chi^{1/2} x\eta_{x\tau}(\tau)}{\Lnorm{2}}^2 \bigr) \,d\tau \\
		& ~~~~ + \int_0^T \alpha^{-r_4-1}(\tau) \alpha_\tau(\tau) \norm{\chi^{1/2} x \overline \rho^{1/2} \eta_{\tau\tau}(\tau)}{\Lnorm{2}}^2 \,d\tau  \\
		& ~~~~ + \int_0^T \alpha^{2-r_4}(\tau) \bigl( \norm{\chi^{1/2} \eta_{\tau\tau}(\tau)}{\Lnorm{2}}^2 + \norm{\chi^{1/2} x\eta_{x\tau\tau}(\tau)}{\Lnorm{2}}^2 \bigr) \,d\tau \\
		& ~~~~ + \int_0^T \alpha^{-l_4-1}(\tau) \alpha_\tau(\tau) \norm{\chi^{1/2} \zeta_\tau(\tau)}{\Lnorm{2}}^2 \,d\tau.
	\end{aligned}
\end{equation}

Correspondingly, we denote the initial energy as
\begin{equation}\label{def:initial-energy}
	\begin{aligned}
		& \mathcal E_{in}:= \norm{x^2 \overline \rho^{1/2} \eta_1}{\Lnorm{2}}^2 + \norm{x\zeta_0}{\Lnorm{2}}^2 + \norm{x^2 \overline\rho^{1/2}\eta_2}{\Lnorm{2}}^2 \\
		& ~~~~ + \norm{x\zeta_1}{\Lnorm{2}}^2 + \norm{x^2 \overline\rho^{1/2}\eta_0}{\Lnorm{2}}^2 + \norm{\chi^{1/2} \zeta_0}{\Lnorm{2}}^2 + \norm{\chi^{1/2} x\overline\rho^{1/2} \eta_{2}}{\Lnorm{2}} \\
		& ~~~~ + \norm{\chi^{1/2} \zeta_1}{\Lnorm{2}}^2. 
	\end{aligned}
\end{equation}
In terms of the energy functionals defined in \eqref{def:functional-01} and \eqref{def:functional-02}, we can rephrase our main theorem, i.e., Theorem \ref{thm:informal} as follows:
\begin{thm}\label{thm:full}
Let $ T > 0 $ be any positive constant, and 
\begin{equation}\label{thmfull:gamma-intervals}
		I_0 := (\dfrac{7}{6}, \infty), ~ I_1 := ( \dfrac{7}{6}, \dfrac{7}{3} ), ~ 
		I_2 := ( \dfrac{11}{9}, \dfrac{5}{3} ).
\end{equation}
{\par\noindent\bf Case 1:}
If $ \gamma \in I_0 $, provided that $ \omega $ in \eqref{initial:uniform-bound} and $ \mathcal E_{in} $ in \eqref{def:initial-energy} are small enough, there exists a global strong solution $ (\eta, \zeta) $ to system \eqref{eq:perturbation-0} with \eqref{bc:perturbation} and \eqref{initial:perturbation}, which satisfies
\begin{equation}\label{thmfull:regularity}
	\begin{gathered}
		\eta, \eta_\tau, \zeta \in L^\infty\bigl(0,T;H^1(0,1) \bigr), ~ x\eta, x\eta_\tau \in L^\infty\bigl(0,T;H^2(0,1) \bigr), \\
		x \overline\rho^{1/2} \eta_\tau, x \overline\rho^{1/2} \eta_{\tau\tau}, \zeta_\tau \in L^\infty\bigl( 0,T; L^2(0,1)\bigr),\\
		 x \eta_{x\tau}, x\eta_{x\tau\tau} \in L^2\bigl(0,T;L^2(0,1) \bigr),
	\end{gathered}
\end{equation}
and 
	\begin{equation}\label{thmest:pointwise-1}
	 \max \lbrace \norm{\eta}{\Lnorm{\infty}}, \norm{x\eta_x}{\Lnorm{\infty}}, \norm{\alpha^{\sigma_1} \eta_\tau}{\Lnorm{\infty}}, \norm{\alpha^{\sigma_1} x\eta_{x\tau}}{\Lnorm{\infty}}, \norm{\alpha^\sigma \mathfrak B}{\Lnorm{\infty}} \rbrace \leq \omega.
	\end{equation}
	Here $ \sigma = 1 $ and $ \sigma_1 $ satisfies \eqref{constraint-000}. 
Also, the following estimates hold
\begin{equation}\label{thmfull:estimates}
	\begin{gathered}
		 \mathcal E_0 (T) + \mathcal E_1(T) + \mathcal D_0(T) + \mathcal D_1(T) \leq C_{r_1,\sigma_1,\alpha_1} \mathcal E_{in} \\
	 ~~~~ + C_{r_1,\sigma_1,\alpha_1} \bigl( \norm{\eta_0}{\Lnorm{2}}^2 + \norm{x\eta_{0,x}}{\Lnorm{2}}^2 \bigr), \\
	\sup_{0\leq \tau \leq T} \bigl\lbrace \norm{x\lbrack (1+\eta(\tau))x\eta_{x\tau}(\tau) - x\eta_{x}(\tau) \eta_\tau(\tau)\rbrack}{\Lnorm{2}}^2 + \norm{\chi^{1/2} \eta_\tau(\tau)}{\Lnorm{2}}^2 \\
	~~~~ + \norm{\chi^{1/2}x \eta_{x\tau}(\tau)}{\Lnorm{2}}^2 
	 + \norm{\chi^{1/2} \eta(\tau)}{\Lnorm{2}}^2 + \norm{\chi^{1/2}x \eta_{x}(\tau)}{\Lnorm{2}}^2 \\
	~~~~ + \norm{\eta_{x}(\tau)}{\Lnorm{2}}^2 
	 +  \norm{x\eta_{xx}(\tau)}{\Lnorm{2}}^2 +  \norm{\eta_{x\tau}(\tau)}{\Lnorm{2}}^2 + \norm{x\eta_{xx\tau}(\tau)}{\Lnorm{2}}^2 \bigr\rbrace \\
	~~~~ \leq C_{r_1,\sigma_1,\alpha_1} \mathcal E_{in}  + C_{r_1,\sigma_1,\alpha_1} \bigl(\norm{\eta_0}{\Lnorm{\infty}}^2 + \norm{x\eta_{0,x}}{\Lnorm{\infty}}^2 \\
	~~~~ + \norm{\zeta_{0,x}}{\Lnorm{2}}^2 + \norm{\eta_{0,x}}{\Lnorm{2}}^2 + \norm{x\eta_{0,x}}{\Lnorm{2}}^2 \bigr), \\
	\sup_{0\leq \tau \leq T}\norm{\zeta_x(\tau)}{\Lnorm{2}}^2 \leq C_{r_1, r_2, \alpha_1} \bigl\lbrace \alpha^{3\gamma-1} \mathcal E_{in} + \norm{\eta_0}{\Lnorm{\infty}}^2 + \norm{x\eta_{0,x}}{\Lnorm{\infty}}^2 \\
	~~~~ + \norm{\zeta_{0,x}}{\Lnorm{2}}^2 + \norm{\eta_{0,x}}{\Lnorm{2}}^2 + \norm{x\eta_{0,x}}{\Lnorm{2}}^2 \bigr\rbrace. 
	\end{gathered}
\end{equation}
{\par\noindent\bf Case 2:}
If, in addition, $ \gamma \in I_1 \subset I_0 $ and
$$
\max\lbrace 3\gamma - 3 - r_1, 2-r_1 \rbrace < 2 \sigma_1 < r_1 < \min\lbrace 6\gamma - 6, 2 \rbrace,
$$
we have, $ \exists \varepsilon_0 \in (0,1) $, 
\begin{equation}\label{thmfull:estimates-1}
	\begin{aligned}
	& \sup_{0\leq\tau\leq T} \norm{\dfrac{\zeta(\tau)}{\overline p}}{\Lnorm{\infty}(\varepsilon_0,1)} + \alpha^{3-3\gamma}\norm{\dfrac{\zeta}{\overline p}}{\Lnorm{\infty}}  \leq C_{r_1,\sigma_1,\alpha_1} \biggl\lbrace  \norm{q_0}{\Lnorm{\infty}} + \norm{\eta_0}{\Lnorm{\infty}} \\
	& ~~~~ + \norm{x\eta_{0,x}}{\Lnorm{\infty}} + \norm{q_0}{\Lnorm{\infty}}^2 
	 + \norm{\eta_0}{\Lnorm{\infty}}^2 
	 + \norm{x\eta_{0,x}}{\Lnorm{\infty}}^2 
	+  \mathcal E_{in}^{1/2}+\mathcal E_{in} \biggr\rbrace,
	\end{aligned}
\end{equation}
provided that $ \norm{q_0}{\Lnorm{\infty}} $ is small enough. 
{\par\noindent\bf Case 3:}
Moreover, if $ \gamma \in I_2 \subset I_1 \subset I_0 $ and
$$
\max\lbrace 3\gamma - 1 - r_1, 2-r_1 \rbrace < 2 \sigma_1 < r_1 < \min\lbrace 6\gamma - 6, 2 \rbrace,
$$
we have,
\begin{equation}\label{thmest:pointwise-q-1}
	\begin{aligned}
	& \sup_{0\leq\tau\leq T} \norm{\dfrac{\zeta(\tau)}{\overline p}}{\Lnorm{\infty}} \leq C_{r_1,\sigma_1,\alpha_1} \biggl\lbrace  \norm{q_0}{\Lnorm{\infty}} + \norm{\eta_0}{\Lnorm{\infty}} 
	 + \norm{x\eta_{0,x}}{\Lnorm{\infty}} \\
	 & ~~~~ + \norm{q_0}{\Lnorm{\infty}}^2 
	 + \norm{\eta_0}{\Lnorm{\infty}}^2 
	 + \norm{x\eta_{0,x}}{\Lnorm{\infty}}^2 
	+  \mathcal E_{in}^{1/2}+\mathcal E_{in} \biggr\rbrace,
	\end{aligned}
\end{equation}
provided that $ \norm{q_0}{\Lnorm{\infty}} $ is small enough. Recall the definition in \eqref{def:weighted-q}, \eqref{thmest:pointwise-q-1} implies that $$ \norm{q}{L^\infty(0,T;L^\infty(0,1))} < C_{r_1, \sigma_1, \alpha_1, \mathcal E_{in}, \omega}, $$ independent of $ T $.
\end{thm}

\begin{proof}[Proof of Theorem \ref{thm:full}]
The estimates \eqref{thmest:pointwise-1} and \eqref{thmfull:estimates} follow from Proposition \ref{prop:total-functional}, Proposition \ref{prop:point-wise-0}, Proposition \ref{prop:regularity}, with an application of continuity arguments. The regularity in \eqref{thmfull:regularity} follows directly from the estimates \eqref{thmest:pointwise-1}, \eqref{thmfull:estimates}, the Sobolev embedding inequality, and Hardy's inequality. \eqref{thmfull:estimates-1} and \eqref{thmest:pointwise-q-1} follow from Proposition \ref{prop:boundedness-entropy}.	
\end{proof}

\begin{remark} The local well-posedness of system \eqref{eq:perturbation-0} can be obtained following a similar finite difference method as in the appendix of \cite{LuoXinZeng2016}. This is due to the fact that 
the additional equation for the pressure \subeqref{eq:perturbation-0}{2} does not have a degenerate weight on $ \zeta $. 
\end{remark}


Now we introduce some notations and inequalities which we will frequently use in this work. Without further mentioned, we use
\begin{equation*}
	\int \,dx = \int_0^1 \,dx, ~ \int\, d\tau = \int_0^T \,d\tau,
\end{equation*}
for any $ T \in (0,\infty) $. $ \norm{\cdot}{\Lnorm{\iota}} $ denotes the $ L^\iota $ norm in the $ x $-variable, $ \iota \in [1,\infty] $. 
We use the notion $ A \lesssim B $ to represent $ A \leq CB $ for some positive generic constant $ C $, which might be different from line to line. We say $ A \simeq B $ if $ A \lesssim B $ and $ B \lesssim A $. 

The following form of Hardy's inequality will be employed in this work.

\begin{lm}[Hardy's inequality, \cite{Jang2014}]
	Let $ k $ be a given real number, and let $ g $ be a function satisfying $ \int_0^1 s^k (g^2 + g'^2)\,ds < \infty $.
	\begin{enumerate}
		\item If $ k > 1$, then we have
		\begin{equation*}
			\int_0^1 s^{k-2}g^2 \,ds \leq C \int_0^1 s^k (g^2+g'^2 )\,ds.
		\end{equation*}
		\item If $ k < 1 $, then $ g $ has a trace at $ x=0 $ and moreover
		\begin{equation*}
			\int_0^1 s^{k-2}(g-g(0))^2\,ds\leq C\int_0^1 s^k g'^2 \,ds.
		\end{equation*}
	\end{enumerate}
\end{lm}

In this work, inspired by \cite{LuoXinZeng2016}, let us define the relative entropy functional as 
\begin{equation}\label{def:relative-entropy-functional}
	\mathfrak H(h) : = \log(1+h)^2(1+h+xh_x),
\end{equation}
where $ h: (0,1)\times(0,T) \mapsto \mathbb R $ is any smooth function. 
We will have the following estimates on the relative entropy.
\begin{lm}[Relative entropy, \cite{XL2018a}]\label{lm:estimates-of-relative-entropy}
	For $ h $ satisfying, 
	\begin{equation}\label{lm:s003} 
		\max \lbrace \norm{h}{\Lnorm{\infty}}, \norm{xh_x}{\Lnorm{\infty}}, \norm{h_\tau}{\Lnorm{\infty}},  \norm{xh_{x\tau}}{\Lnorm{\infty}} 
		\rbrace < \varepsilon,
	\end{equation} 
	with some $ 0 < \varepsilon  < 1 $ small enough, the following estimates of the function $ \mathfrak H (h) $ hold:
	\begin{gather}
		\int ( h_x^2 + x^2 h_{xx}^2 ) \,dx  \leq C \int \bigl( \mathfrak H(h) \bigr)_x^2 \,dx ,  \label{lm:s001}  \\
		\int (h_{x\tau}^2 + x^2 h_{xx\tau}^2) \,dx \leq C \int \bigl(\mathfrak H(h) \bigr)_{x\tau}^2 \,dx + \varepsilon C \int (h_x^2 + x^2 h_{xx}^2 )\,dx, \label{lm:s002}
	\end{gather}
	for some positive constant $ C $.
\end{lm}

\section{Energy estimates}\label{sec:energy-estimate}

In the following, 
we make the a prior assumption: for some positive constant $ \omega \in (0,1) $, which is small enough, such that, for any $ \tau \in (0,T) $, 
\begin{equation}\label{def:a-priori-assm}
\max \lbrace \norm{\eta}{\Lnorm{\infty}}, \norm{x\eta_x}{\Lnorm{\infty}}, \norm{\alpha^{\sigma_1} \eta_\tau}{\Lnorm{\infty}}, \norm{\alpha^{\sigma_1} x\eta_{x\tau}}{\Lnorm{\infty}}, \norm{\alpha^\sigma \mathfrak B}{\Lnorm{\infty}} \rbrace \leq \omega,
\end{equation}
for some positive constant $ \sigma_1 \in (0,\sigma) $, where $ \sigma = 1 $. This a priori assumption will dramatically simplified the presentation of our proof. We want to point-out that in section \ref{sec:point-wise}, such a priori assumption is closed in the sense that the $ L^\infty $ norm listed in \eqref{def:a-priori-assm} can be bounded by the energy functionals. Thus, with a continuity argument, one can conclude the rigidity of the estimates and the a priori assumption. 

In this section, we are going to show the following:
\begin{prop}[Energy estimates]\label{prop:energy-estimates}
Consider a smooth enough solution $ (\eta, \zeta) $ to system \eqref{eq:perturbation-0}, satisfying \eqref{bc:perturbation} and \eqref{initial:perturbation}. Suppose that \eqref{def:a-priori-assm} is satisfied with $ \omega \in (0,1) $ small enough, and that the expanding rate $ \beta_1 = \alpha_1 $ in \eqref{expandingrate} is large enough, we have
\begin{equation}\label{propest:energy}
	\mathcal E_0 (T) + \mathcal D_0(T) \leq C_{r_1,\sigma_1,\beta_1,\beta_2} \mathcal E_{in},
\end{equation}
for any positive time $ T \in (0,\infty) $. Here $ C_{r_1,\sigma_1,\beta_1,\beta_2}$ is some constant depending on $ r_1, \sigma_1,\beta_1,\beta_2 $. 
\end{prop}
We separate the proof of Proposition \ref{prop:energy-estimates} in the following two lemmas. We start with exploring the $ L^2 $ estimates of $ \eta_\tau $ and $ \zeta $. 

\begin{lm}\label{lm:L2}
Under the same assumptions as in Proposition \ref{prop:energy-estimates}, we have
\begin{equation}\label{lmest:L2}
	\begin{aligned}
		& \alpha^{r_1} \int x^4 \overline\rho \abs{\eta_\tau}{2} \,dx + \alpha^{-l_1} \int x^2 \abs{\zeta}{2} \,dx  + \int_0^\tau \alpha^{r_1-1}\alpha_\tau \int x^4 \overline\rho \abs{\eta_\tau}{2} \,dx \,d\tau' \\
		& ~~~~ + \int_0^\tau \alpha^{r_1+2} \int x^2 \bigl\lbrack (1+\eta) x\eta_{x\tau} - x\eta_x \eta_\tau \bigr\rbrack^2 \,dx\,d\tau' \\
		& ~~~~ + \int_0^\tau \alpha^{-l_1-1} \alpha_\tau \int x^2 \abs{\zeta}{2} \,dx\,d\tau'  \leq C_{r_1}  \omega \mathcal D_0(\tau) \\
		& ~~~~ +  C_{r_1} \bigl( \beta_1^{-\varsigma} + 1 \bigr) \mathcal E_{in},
	\end{aligned}
\end{equation}
for any given $ \tau \in (0,T) $, some constant $ C_{r_1} $ depends only on $ r_1 $, and some constant $ \varsigma > 0 $. Here $ l_1 $ is given in \eqref{indices}. 
\end{lm}

\begin{proof}
	Taking the $ L^2 $-inner product of \subeqref{eq:perturbation-0}{1} with $ \alpha^{r_1 - 1} x^3 (1+\eta)^2 \eta_\tau $ yields
	\begin{equation}\label{ene:001}
	\begin{aligned}
	& \dfrac{d}{d\tau} \biggl\lbrace \dfrac{\alpha^{r_1}}{2} \int x^4 \overline{\rho} \abs{\eta_\tau}{2} \,dx \biggr\rbrace + ( 1 - \dfrac{r_1}{2}) \alpha^{r_1-1} \alpha_\tau \int x^4 \overline{\rho} \abs{\eta_\tau}{2} \,dx \\
	& ~~~~ + \dfrac{4}{3} \mu \alpha^{r_1+2} \int \dfrac{x^2 \lbrack (1+\eta) x\eta_{x\tau} - x\eta_x\eta_\tau \rbrack^2}{1+\eta + x\eta_x} \,dx = L_1 + L_2,
	\end{aligned}
	\end{equation}
	where
	\begin{align*}
	& L_1 := \delta \alpha^{r_1+3-3\gamma} \int x^4 \overline{\rho} (1+\eta) \eta \eta_\tau \,dx, \\
	& L_2 := - \alpha^{r_1+3-3\gamma} \int x^3 \biggl( \dfrac{\zeta}{\bigl\lbrack (1+\eta+x\eta_x)(1+\eta)^2 \bigr\rbrack^\gamma} \biggr)_x (1+\eta)^2 \eta_\tau \,dx \\
	& ~~~~ = \alpha^{r_1+3-3\gamma} \int \dfrac{x^2 (1+\eta) \bigl( (1+\eta)(\eta_\tau + x\eta_{x\tau}) + 2(1+\eta + x\eta_x)\eta_\tau \bigr) \zeta}{\bigl\lbrack (1+\eta + x\eta_x)(1+\eta)^2 \bigr\rbrack^\gamma} \,dx.
	\end{align*}
	On the other hand, taking the $ L^2 $-inner product of \subeqref{eq:perturbation-0}{2} with $ \alpha^{-l_1} x^2 \zeta $ yields
	\begin{equation}\label{ene:002}
	\begin{aligned}
	\dfrac{d}{d\tau} \biggl\lbrace \dfrac{\alpha^{-l_1}}{2} \int x^2  \abs{\zeta}{2} \,dx \biggr\rbrace + \dfrac{l_1}{2} \alpha^{-l_1-1} \alpha_\tau \int x^2 \abs{\zeta}{2} \,dx = L_3 + L_4, 
	\end{aligned}
	\end{equation}
	where
	\begin{align*}
	& L_3 := - \gamma \alpha^{-l_1} \int x^2 \overline p \zeta \bigl\lbrack (1+\eta+x\eta_x)^{\gamma-1}(1+\eta)^{2\gamma} (\eta_\tau + x\eta_{x\tau}) \\
	&	~~~~ ~~~~ + 2 (1+\eta+x\eta_x)^{\gamma}(1+\eta)^{2\gamma-1} \eta_\tau \bigr\rbrack \,dx, \\
	& L_4 := \dfrac{4}{3} \mu (\gamma-1) \alpha^{3\gamma-l_1-1} \int x^2 \bigl\lbrack (1+\eta + x\eta_x) (1+\eta)^2 \bigr\rbrack^{\gamma} \zeta \mathfrak B^2 \,dx.
	\end{align*}
	
	We notice that by applying the fundamental theorem of calculus and H\"older's inequality to $ \int x^4 \overline{\rho} \abs{\eta}{2} \,dx $, one has,
	\begin{equation}\label{est:eta}
	\biggl( \int x^4 \overline{\rho} \abs{\eta(\tau)}{2} \,dx \biggr)^{1/2} \leq \biggl(\int x^4 \overline{\rho} \abs{\eta_0}{2} \,dx \biggr)^{1/2} + \int_0^\tau \biggl(\int x^4 \overline{\rho} \abs{\eta_\tau}{2} \,dx \biggr)^{1/2} \,d\tau.
	\end{equation}
	In addition, by applying Hardy's inequality, one can derive 
	\begin{equation*}
	\begin{aligned}
	& \int x^2 ( \eta_\tau^2 + x^2\eta_{x\tau}^2 )\,dx \lesssim (1+\omega) \biggl( \int x^2 \eta_\tau^2 \,dx \\
	& ~~~~ + \int x^2( (1+\eta)x\eta_{x\tau} - x\eta_x\eta_\tau )^2 \,dx \biggr) \\
	& \lesssim (1+\omega) \biggl( \int x^4 (1-x)^2 \eta_\tau^2 + x^4 \eta_{x\tau}^2 \,dx \\
	& ~~~~ + \int x^2( (1+\eta)x\eta_{x\tau} - x\eta_x\eta_\tau )^2 \,dx\biggr) \\
	& \lesssim (1+\omega) \int x^4 (1-x)^2 \eta_\tau^2 \,dx + \omega \int x^2 \eta_\tau^2 \,dx \\
	& ~~~~ + (1+\omega) \int x^2( (1+\eta)x\eta_{x\tau} - x\eta_x\eta_\tau )^2 \,dx,\\
	\end{aligned}
	\end{equation*}
	which implies, provided $ \omega $ is small enough,
	\begin{align*}
	& \int x^2 ( \eta_\tau^2 + x^2 \eta_{x\tau}^2 )\,dx \lesssim (1+\omega) \int x^4 (1-x)^2 \eta_\tau^2 \,dx \\ & ~~~~ + (1+\omega) \int x^2( (1+\eta)x\eta_{x\tau} - x\eta_x\eta_\tau )^2 \,dx.
	\end{align*}
	Repeating the this process, one can get
	\begin{equation}\label{est:hardy-01}
	\begin{aligned}
	& \int x^2 ( \eta_\tau^2 + x^2 \eta_{x\tau}^2 )\,dx \lesssim (1+\omega) \int x^4 \overline{\rho} \eta_\tau^2 \,dx \\
	& ~~~~ + (1+\omega) \int x^2( (1+\eta)x\eta_{x\tau} - x\eta_x\eta_\tau )^2 \,dx,
	\end{aligned}
	\end{equation}
	with property \eqref{density:boundary}. 
	
	With inequalities \eqref{est:eta} and \eqref{est:hardy-01}, we estimate $ L_i $'s in the following:
	\begin{align*}
	& \int L_1 \,d\tau \leq (1+\omega) \delta \int \alpha^{r_1 + 3 - 3\gamma} \biggl( \int x^4 \overline{\rho} \abs{\eta}{2} \,dx \biggr)^{1/2}\biggl( \int x^4 \overline{\rho} \abs{\eta_\tau}{2} \,dx \biggr)^{1/2} \,d\tau \\
	& ~~ \leq (1 + \omega) \delta \int \alpha^{r_1 + 3 - 3\gamma} \biggl( \int x^4 \overline\rho \abs{\eta_\tau}{2} \,dx \biggr)^{1/2} \,d\tau \times \biggl\lbrace \biggl( \int x^4 \overline\rho \abs{\eta_0}{2} \,dx \biggr)^{1/2}\\
	&	~~~~  + \int \biggl( \int x^4 \overline\rho \abs{\eta_\tau}{2} \,dx \biggr)^{1/2} \,d\tau \biggr\rbrace \leq \varepsilon \int \beta_1 \alpha^{r_1} \int x^4 \overline \rho \abs{\eta_\tau}{2}\,dx \,d\tau \\
	& ~~~~ + C_\varepsilon (1+\omega)\delta^2 \int x^4 \overline
	\rho \abs{\eta_0}{2} \,dx \times \int \beta_1^{-1} \alpha^{r_1 + 6 - 6\gamma} \,d\tau \\
	& ~~~~ + (1+\omega)\delta \beta_1^{-1} \biggl( \int \alpha^{r_1+6-6\gamma} \,d\tau\biggr)^{1/2} \biggl( \int \alpha^{-r_1} \,d\tau \biggr)^{1/2}\\
	& ~~~~ \times  \int \beta_1 \alpha^{r_1} \int x^4 \overline\rho \abs{\eta_\tau}{2} \,dx \,d\tau, \\
	& \int L_2 \,d\tau \lesssim (1+\omega) \int \alpha^{r_1+3-3\gamma}  \biggl( \int x^2 \abs{\zeta}{2} \,dx \biggr)^{1/2} \\
	& ~~~~ \times \biggl( \int x^2 ( \eta_\tau^2 + x^2 \eta_{x\tau}^2 ) \,dx \biggr)^{1/2}\,d\tau \lesssim (1+\omega) \int \alpha^{r_1 + 3 - 3\gamma} \biggl( \int x^2  \abs{\zeta}{2} \,dx \biggr)^{1/2} \\
	& ~~~~ \times \biggl( \int x^4 \overline \rho \eta_\tau^2 \,dx + \int x^2 ( ( 1+\eta ) x\eta_{x\tau} - x\eta_x \eta_\tau )^2 \,dx \biggr)^{1/2} \,d\tau \\
	& ~~ \lesssim \varepsilon \int \beta_1 \alpha^{r_1} \int x^4 \overline\rho \abs{\eta_\tau}{2} \,dx \,d\tau + \varepsilon \int \alpha^{r_1+2} \int x^2 ( (1+\eta) x\eta_{x\tau} - x\eta_x \eta_\tau )^2 \,dx\,d\tau \\
	& ~~~~ + C_\varepsilon (1+\omega)  \sup_\tau ( \beta_1^{-2} \alpha^{r_1 +l_1 + 6 - 6\gamma} + \beta_1^{-1} \alpha^{r_1+l_1+4-6\gamma} ) \int \beta_1 \alpha^{-l_1} \int x^2 \abs{\zeta}{2} \,dx \,d\tau,\\
	& \int L_3 \,d\tau \lesssim (1+\omega) \int \alpha^{-l_1} \biggl( \int x^2  \abs{\zeta}{2} \,dx \biggr)^{1/2} \biggl( \int x^2 ( \eta_{\tau}^2 + x^2\eta_{x\tau}^2 )\,dx \biggr)^{1/2}  \,d\tau \\
	& ~~ \lesssim \varepsilon \int \beta_1 \alpha^{-l_1} \int x^2 \abs{\zeta}{2} \,dx \,d\tau + C_\varepsilon \beta_1^{-2} \sup_\tau \alpha^{-r_1-l_1} \times \int \beta_1 \alpha^{r_1} \int x^4 \overline\rho \abs{\eta_\tau}{2} \,dx\,d\tau \\
	& ~~~~ + C_\varepsilon \beta_1^{-1} \sup_\tau \alpha^{-r_1-l_1-2} \times \int \alpha^{r_1+2} \int x^2 ((1+\eta)x\eta_{x\tau} - x\eta_x\eta_\tau)^2 \,dx\,d\tau,\\
	& \int L_4 \,d\tau \lesssim \int \alpha^{3\gamma-l_1-1} \int x^2 \abs{\zeta}{} \abs{\mathfrak B}{}  \abs{(1+\eta) x\eta_{x\tau} - x\eta_x \eta_\tau}{}  \,dx \,d\tau \\
	& ~~ \lesssim \omega \int \alpha^{3\gamma-l_1-\sigma-1} \biggl(\int x^2 \abs{\zeta}{2} \,dx \biggr)^{1/2} \biggl( \int x^2 \bigl\lbrack(1+\eta) x\eta_{x\tau} - x\eta_x \eta_\tau \bigr\rbrack^2 \,dx \biggr)^{1/2} \,d\tau \\
	& ~~ \lesssim \omega \beta_1^{-1/2} \sup_\tau \alpha^{3\gamma - l_1/2 - r_1/2 - \sigma - 2} \times   \biggl( \int \beta_1 \alpha^{-l_1} \int x^2 \zeta^2 \,dx \,d\tau \biggr)^{1/2} \\
	& ~~~~ \times \biggl(  \int \alpha^{r_1+2} \int x^2 ((1+\eta) x\eta_{x\tau} - x\eta_x \eta_\tau)^2 \,dx\,d\tau \biggr)^{1/2}.
	\end{align*}
Then with our choice of $ r_1, l_1 $ in \eqref{indices}, and the fact 
\begin{equation}\label{ineq:integral-time-weight}
	\int_0^\tau \alpha^{\iota} \,d\tau \leq \beta_1^{-1} \int_0^\tau \alpha^{\iota - 1} \alpha_\tau \,d\tau = \dfrac{1}{\beta_1 \iota} \bigl( \alpha^\iota - \alpha^{\iota}_0 \bigr), ~~ \iota \neq 0,
\end{equation}
the estimates above yield the lemma after integrating \eqref{ene:001} and \eqref{ene:002} with respect to $ \tau $, letting $ \varepsilon $ be small enough and $ \beta_1 $ be large enough.
%
\end{proof}

Next, we are going to show the $ L^2 $ estimates of $ \eta_{\tau\tau} $ and $ \zeta_\tau $. In order to do so, applying the operations 
$ \partial_\tau \lbrace (1+\eta)^2 \eqref{eq:perturbation-0}_{1} \rbrace $  and  $ \partial_\tau  \eqref{eq:perturbation-0}_{2} $ leads to the following system:
\begin{equation}\label{eq:perturbation-1}
	\begin{cases}
		\alpha x\overline \rho \eta_{\tau\tau\tau}  + 2 \alpha_\tau x \overline \rho \eta_{\tau\tau} +\alpha_{\tau\tau} x\overline \rho \eta_\tau \\
		~~~~ + \biggl\lbrace - \alpha^{4-3\gamma}   \delta x \overline \rho (1+\eta) \eta + \alpha^{4-3\gamma} (1+\eta)^2\biggl( \dfrac{\zeta}{\bigl \lbrack (1+\eta+x\eta_x)(1+\eta)^2\bigr\rbrack^\gamma} \biggr)_x  \biggr\rbrace_\tau \\
		~~~~ = \dfrac{4}{3} \mu \alpha^3 (1+\eta)^2 \bigl\lbrack \mathfrak B_{x\tau} + 3 \bigl( \dfrac{\eta_\tau}{1+\eta} \bigr)_{x\tau}  \bigr\rbrack + \dfrac{8}{3} \mu \alpha^3 (1+\eta) \eta_\tau \bigl\lbrack \mathfrak B_x + 3 \bigl( \dfrac{\eta_\tau}{1+\eta}\bigr)_x  \bigr\rbrack \\
		~~~~ + 4 \mu \alpha^2 \alpha_\tau (1+\eta)^2 \bigl\lbrack \mathfrak B_{x} + 3 \bigl( \dfrac{\eta_\tau}{1+\eta} \bigr)_{x}  \bigr\rbrack , \\
		\zeta_{\tau\tau} + \gamma \overline p \bigl\lbrace (1+\eta+x\eta_x)^{\gamma-1}(1+\eta)^{2\gamma} (\eta_\tau + x\eta_{x\tau}) \\
		~~~~ ~~~~ + 2 (1+\eta+x\eta_x)^{\gamma}(1+\eta)^{2\gamma-1} \eta_\tau \bigr\rbrace_\tau \\
		~~~~ = \dfrac{4}{3} \mu (\gamma -1) \bigl\lbrace \alpha^{3\gamma-1} \bigl \lbrack (1+\eta+x\eta_x)(1+\eta)^2\bigr\rbrack^\gamma \mathfrak B^2 \bigr\rbrace_\tau,
	\end{cases}
\end{equation}
with the boundary condition
\begin{equation*}
	\mathfrak B_\tau \big|_{x=1} = \bigl( \dfrac{\eta_\tau + x\eta_{x\tau}}{1+\eta+x\eta_x} - \dfrac{\eta_\tau}{1+\eta} \bigr)_\tau \big|_{x=1} = 0. 
\end{equation*}

\begin{lm}\label{lm:d-tau}
	Under the same assumptions as in Proposition \ref{prop:energy-estimates}, we have
	\begin{equation}\label{lmest:L2-d-tau}
	\begin{aligned}
		& \alpha^{r_2} \int x^4 \overline\rho \abs{\eta_{\tau\tau}}{2} \,dx + \alpha^{-l_2} \int x^2 \abs{\zeta_\tau}{2} \,dx + \int_0^\tau \alpha^{r_2-1}\alpha_\tau \int x^4 \overline\rho \abs{\eta_{\tau\tau}}{2} \,dx \,d\tau' \\
		& ~~~~ + \int_0^\tau \alpha^{r_2 +2} \int x^2 \bigl\lbrack (1+\eta) x\eta_{x\tau\tau} - x\eta_x \eta_{\tau\tau} \bigr\rbrack^2 \,dx\,d\tau' \\
		& ~~~~ + \int_0^\tau \alpha^{-l_2-1} \alpha_\tau \int x^2 \abs{\zeta_\tau}{2} \,dx\,d\tau'  \leq C_{r_1,\sigma_1,\beta_1,\beta_2} \omega \mathcal D_0(\tau) \\
		& ~~~~ +  C_{r_1,\sigma_1,\beta_1,\beta_2}  \mathcal E_{in},
	\end{aligned}
\end{equation}
for any given $ \tau \in (0,T) $, some constant $ C_{r_1,\sigma_1,\beta_1,\beta_2} $ depends only on $ r_1, \sigma_1,\beta_1,\beta_2 $, and some constant $ \varsigma > 0 $. Here, $ r_2, l_2 $ are given in \eqref{indices}. 
\end{lm}

\begin{proof}
	After taking the $ L^2 $-inner product of \subeqref{eq:perturbation-1}{1} with $ \alpha^{r_2 - 1} x^3 \eta_{\tau\tau} $ and \subeqref{eq:perturbation-1}{2} with $ \alpha^{- l_2} x^2 \zeta_\tau $, it follows
	\begin{align}
		& \dfrac{d}{d\tau} \biggl\lbrace \dfrac{\alpha^{r_2}}{2} \int x^4 \overline \rho \abs{\eta_{\tau\tau}}{2} \,dx  \biggr\rbrace + \bigl(2 -\dfrac{r_2}{2}  \bigr) \alpha^{r_2 - 1} \alpha_\tau \int x^4 \overline \rho \abs{\eta_{\tau\tau}}{2} \,dx {\nonumber} \\
		& ~~~~ +  \dfrac{4\mu}{3}\alpha^{r_2 + 2} \int \dfrac{x^2 \lbrack (1+\eta) x \eta_{x\tau\tau} - x\eta_x \eta_{\tau\tau} \rbrack^2}{1+\eta+x\eta_x} \,dx = \sum_{i=5}^{10} L_i, \label{ene:003}\\
		& \dfrac{d}{d\tau} \biggl\lbrace \dfrac{\alpha^{-l_2}}{2} \int x^2 \abs{\zeta_\tau}{2} \,dx 
		\biggr\rbrace +  \dfrac{l_2}{2} \alpha^{-l_2-1} \alpha_\tau \int x^2 \abs{\zeta_\tau}{2} \,dx
		{\nonumber}\\
		& ~~~~ 
		= \sum_{i=11}^{14} L_i,
		\label{ene:004}
	\end{align}
	where
	\begin{align*}
		& L_5: =  - \alpha^{r_2 - 1} \alpha_{\tau\tau} \int  x^4 \overline \rho \eta_\tau \eta_{\tau\tau} \,dx, \\
		& L_6 := \delta \alpha^{r_2 + 3 - 3 \gamma} \int \bigl\lbrace x^4 \overline \rho (1+\eta) \eta  \bigr\rbrace_\tau  \eta_{\tau\tau} \,dx, \\
		& L_7 := (4-3\gamma) \delta \alpha^{r_2 + 2-3\gamma}\alpha_\tau \int x^4 \overline\rho (1+\eta)\eta \eta_{\tau\tau} \,dx , \\
		& L_8 := \int \biggl\lbrace 2 \alpha^{4-3\gamma} x^2 (1+\eta)(1+\eta+x\eta_x) \biggl( \dfrac{\zeta}{\bigl\lbrack (1+\eta+x\eta_x)(1+\eta)^2 \bigr\rbrack^\gamma} \biggr) \biggr\rbrace_\tau \alpha^{r_2 - 1} \eta_{\tau\tau} \,dx \\
		& ~~~~ + \int \biggl\lbrace \alpha^{4-3\gamma} x^2 (1+\eta)^2 \biggl( \dfrac{\zeta}{\bigl\lbrack (1+\eta+x\eta_x)(1+\eta)^2 \bigr\rbrack^\gamma} \biggr) \biggr\rbrace_\tau \alpha^{r_2 - 1} (\eta_{\tau\tau} + x\eta_{x\tau\tau} ) \,dx,  \\
		& L_9 : =  \dfrac{4}{3} \mu \alpha^{r_2 + 2} \int \biggl( x^2 (1+\eta)^2 (\eta_{\tau\tau} + x \eta_{x\tau\tau}) + 2 x^2 (1+\eta)(1+\eta+x\eta_{x}) \eta_{\tau\tau} \biggr) \\
		& ~~~~ \times \biggl( \dfrac{(\eta_\tau+x\eta_{x\tau})^2}{(1+\eta+x\eta_x)^2} - \dfrac{\eta_\tau^2}{(1+\eta)^2} \biggr) + 6 x^2 (1+\eta)^2 \eta_{\tau\tau} \biggl( \dfrac{x\eta_{x} \eta_\tau^2}{(1+\eta)^3} - \dfrac{x\eta_{x\tau}\eta_\tau}{(1+\eta)^2} \biggr)  \,dx  , \\
		& L_{10} := \dfrac{8}{3} \mu \alpha^{r_2 +2} \int - ( x^3 (1+\eta) \eta_\tau \eta_{\tau\tau})_x \mathfrak B + 3 x^3 (1+\eta) \eta_\tau \eta_{\tau\tau} \biggl( \dfrac{\eta_\tau}{1+\eta} \biggr)_x \,dx\\
		& ~~~~ + 4 \mu \alpha^{r_2 + 1}\alpha_\tau \int  - (x^3 (1+\eta)^2 \eta_{\tau\tau})_x \mathfrak B + 3 x^3(1+\eta)^2 \eta_{\tau\tau} \biggl( \dfrac{\eta_\tau}{1+\eta} \biggr)_x \,dx  ,\\
		& L_{11} :=  - \gamma \alpha^{-l_2} \int x^2 \overline p (1+\eta + x\eta_x)^{\gamma} (1+\eta)^{2\gamma} \biggl( \dfrac{\eta_{\tau\tau} + x\eta_{x\tau\tau}}{1+\eta+x\eta_x} + 2 \dfrac{\eta_{\tau\tau}}{1+\eta} \biggr) \zeta_\tau\,dx,\\
		& L_{12} := - \gamma \alpha^{-l_2} \int x^2 \overline p \biggl\lbrace \bigl\lbrack ( 1+\eta + x\eta_x)^{\gamma-1} (1+\eta)^{2\gamma} \bigr\rbrack_\tau (\eta_\tau + x\eta_{x\tau}) \\
		& ~~~~ + 2 \bigl\lbrack ( 1+\eta + x\eta_x)^{\gamma} (1+\eta)^{2\gamma-1} \bigr\rbrack_\tau \eta_\tau \biggr\rbrace \zeta_\tau \,dx, \\
		& L_{13} : = \dfrac{4}{3} \mu (\gamma - 1)(3\gamma-1) \alpha^{3\gamma - l_2 - 2} \alpha_\tau \int x^2 \bigl\lbrack (1+\eta+x\eta_x)(1+\eta)^2 \bigr\rbrack^\gamma \mathfrak B^2 \zeta_\tau \,dx , \\
		& L_{14} := \dfrac{4}{3} \mu (\gamma-1) \alpha^{3\gamma-l_2-1} \int x^2 \bigl\lbrace \bigl\lbrack (1+\eta+x\eta_x)(1+\eta)^2 \bigr\rbrack^\gamma \mathfrak B^2 \bigr\rbrace_\tau \zeta_\tau \,dx  .
	\end{align*}
	Applying H\"older's and Young's inequalities to $ \int L_5 \,d\tau $ yields the following inequalities:
	\begin{align*}
		& \int L_5 \,d\tau \lesssim \int \beta_2^2 \alpha^{r_2} \int x^4 \overline \rho \abs{\eta_\tau}{} \abs{\eta_{\tau\tau}}{} \,dx \,d\tau \lesssim \varepsilon \int \beta_1 \alpha^{r_2} \int x^4 \overline\rho \abs{\eta_{\tau\tau}}{2} \,dx\,d\tau\\
		& ~~~~ + C_\varepsilon \sup_\tau \alpha^{r_2 - r_1} \beta_2^4 \beta_1^{-2} \int \beta_1 \alpha^{r_1} \int x^4 \overline \rho \abs{\eta_\tau}{2} \,dx\,d\tau,
	\end{align*}
	Similarly, with \eqref{est:eta},
	\begin{align*}
		& \int L_6 \,d\tau \lesssim \varepsilon \int \beta_1 \alpha^{r_2} \int x^4 \overline\rho \abs{\eta_{\tau\tau}}{2}\,dx\,d\tau \\
		& ~~~~ + C_\varepsilon \delta^2 \beta_{1}^{-2} \sup_\tau \alpha^{r_2 - r_1 - 6 \gamma + 6} \int \beta_1 \alpha^{r_1} \int x^4 \overline \rho \abs{\eta_\tau}{2} \,dx\,d\tau, \\
		& \int L_7 \,d\tau \lesssim \delta \int \alpha^{r_2 - 3 \gamma + 2} \alpha_\tau \biggl( \int x^4 \overline \rho \abs{\eta_{\tau\tau}}{2} \,dx \biggr)^{1/2} \biggl\lbrack \biggl( \int x^4 \overline\rho \abs{\eta_0}{2} \,dx \biggr)^{1/2} \\
		& ~~~~  + \int_0^\tau \biggl( \int x^4 \overline\rho \abs{\eta_\tau(\tau')}{2} \,dx \biggr)^{1/2} \,d\tau'  \biggr\rbrack \,d\tau \lesssim \varepsilon \int \alpha^{r_2-1} \alpha_\tau \int x^4 \overline\rho \abs{\eta_{\tau\tau}}{2} \,dx\,d\tau\\
		& ~~~~ + C_\varepsilon \delta^2 \int \alpha^{ r_2 - 6\gamma + 5}\alpha_\tau \,d\tau \times \int x^4 \overline \rho \abs{\eta_0}{2} \,dx + C_\varepsilon \delta^2 \beta_1^{-1}  \int \alpha^{-r_1} \,d\tau \\
		& ~~~~ ~~~~ \times  \int \alpha^{r_2 - 6 \gamma + 5} \alpha_\tau \,d\tau  \times \int \beta_1 \alpha^{r_1} \int x^4 \overline \rho \abs{\eta_\tau}{2} \,dx\,d\tau, 
	\end{align*}
	Moreover, notice $ L_8 $ -- $ L_{14} $,
	\begin{align*}
		& L_8 \lesssim (1+\omega) \alpha^{r_2 - 3 \gamma + 3} \int x^2 \abs{\zeta}{} ( \abs{\eta_\tau}{} + \abs{x\eta_{x\tau}}{} )(\abs{\eta_{\tau\tau}}{} + \abs{x\eta_{x\tau\tau}}{}) \,dx\\
		& ~~~~ + (1+\omega) \alpha^{r_2 - 3\gamma + 3} \int x^2 \abs{\zeta_\tau}{} (\abs{\eta_{\tau\tau}}{} + \abs{x\eta_{x\tau\tau}}{}) \,dx \\
		& ~~~~ + (1+\omega) \alpha^{r_2 - 3 \gamma + 2} \alpha_\tau \int x^2 \abs{\zeta}{} (\abs{\eta_{\tau\tau}}{} + \abs{x\eta_{x\tau\tau}}{}) \,dx, \\
		& L_9 \lesssim (1+\omega) \alpha^{r_2 + 2} \int x^2 (\abs{\eta_{\tau}}{} + \abs{x\eta_{x\tau}}{}) \abs{(1+\eta)x\eta_{x\tau} - x\eta_x \eta_\tau}{} ( \abs{\eta_{\tau\tau}}{} + \abs{x\eta_{x\tau\tau}}{}) \,dx, \\
		& L_{10} \lesssim (1+\omega)  \alpha^{r_2 + 2} \int x^2 (\abs{\eta_{\tau}}{} + \abs{x\eta_{x\tau}}{})\abs{(1+\eta)x\eta_{x\tau} - x\eta_x \eta_\tau}{}( \abs{\eta_{\tau\tau}}{} + \abs{x\eta_{x\tau\tau}}{}) \,dx \\
		& ~~~~ + (1+\omega) \alpha^{r_2+1} \alpha_\tau \int x^2 \abs{\mathfrak B}{}( \abs{(1+\eta)x\eta_{x\tau\tau} - x\eta_x \eta_{\tau\tau}}{}) \,dx, \\
		& L_{11} \lesssim (1+\omega) \alpha^{-l_2} \int x^2 \abs{\zeta_\tau}{} (\abs{\eta_{\tau\tau}}{} + \abs{x\eta_{x\tau\tau}}{} ) \,dx, \\
		& L_{12} \lesssim (1+\omega) \alpha^{-l_2}  \int x^2 \abs{\zeta_\tau}{} (\abs{\eta_{\tau}}{2} + \abs{x\eta_{x\tau}}{2} ) \,dx, \\
		& L_{13} \lesssim (1+\omega) \alpha^{3\gamma-l_2 - 2} \alpha_\tau \int x^2 \abs{\zeta_\tau}{} \abs{\mathfrak B}{} \abs{(1+\eta)x\eta_{x\tau} - x\eta_x \eta_\tau}{} \,dx,\\
		& L_{14} \lesssim (1+\omega) \alpha^{3\gamma - l_2 - 1} \int x^2 \abs{\zeta_\tau}{} \bigl( \abs{(1+\eta)x\eta_{x\tau} - x\eta_x \eta_\tau}{}(\abs{\eta_{\tau}}{2} +  \abs{x\eta_{x\tau}}{2}) \\
		& ~~~~ ~~~~ +  \abs{\mathfrak B}{} \abs{(1+\eta)x\eta_{x\tau\tau} - x\eta_x \eta_{\tau\tau}}{}  \bigr) \,dx.
	\end{align*}
	Consequently, applying \eqref{def:a-priori-assm}, H\"older's and Young's inequalities implies, as before, 
	\begin{align*}
		& \int L_8 \,d\tau \lesssim \varepsilon \int  \alpha^{r_2-1}\alpha_\tau \int x^4 \overline \rho \abs{\eta_{\tau\tau}}{2} \,dx\,d\tau\\
		& ~~~~  + \varepsilon \int \alpha^{r_2 + 2} \int x^2 \bigl\lbrack (1+\eta) x\eta_{x\tau\tau} - x\eta_x \eta_{\tau\tau}\bigr\rbrack^2 \,dx\,d\tau \\
		& ~~~~ + C_\varepsilon \omega^2 \bigl( \beta_1^{-2} \sup_\tau \alpha^{r_2 + l_1 - 6\gamma - 2 \sigma_1 + 6} + \beta_1^{-1} \sup_\tau \alpha^{r_2 + l_1 - 6\gamma - 2\sigma_1 + 4} \bigr)\\
		& ~~~~ ~~~~ \times \int \beta_1 \alpha^{-l_1} \int x^2 \abs{\zeta}{2} \,dx\,d\tau + C_\varepsilon \bigl( \beta_1^{-2} \sup_\tau \alpha^{r_2 + l_2 - 6\gamma  + 6} \\
		& ~~~~ ~~~~ + \beta_1^{-1} \sup_\tau \alpha^{r_2 + l_2 - 6\gamma + 4} \bigr) \times \int \beta_1 \alpha^{-l_2} \int x^2 \abs{\zeta_\tau}{2} \,dx\,d\tau \\
		& ~~~~ + C_\varepsilon \bigl(  \sup_\tau \alpha^{r_2 + l_1 - 6\gamma + 6} + \sup_\tau \alpha^{r_2 + l_1 - 6\gamma + 3} \alpha_\tau \bigr)\\
		& ~~~~ ~~~~ \times \int \alpha^{-l_1-1}\alpha_\tau \int x^2 \abs{\zeta}{2} \,dx\,d\tau, \\
		& \int L_9 + L_{10} \,d\tau \lesssim \varepsilon \int \alpha^{r_2-1}\alpha_\tau \int x^4 \overline\rho \abs{\eta_{\tau\tau}}{2} \,dx\,d\tau \\
		& ~~~~ + \varepsilon \int \alpha^{r_2 +2 }\int x^2 \bigl\lbrack (1+\eta) x\eta_{x\tau\tau} - x\eta_x \eta_{\tau\tau}\bigr\rbrack^2 \,dx\,d\tau \\
		& ~~~~ + C_\varepsilon \sup_\tau \bigl( \omega \beta_1^{-1} \alpha^{r_2 - r_1 - 2\sigma_1 + 2} + \omega \alpha^{r_2 - r_1 - 2\sigma_1} + \alpha^{r_2 - r_1 - 2 }\alpha_\tau^2 \bigr)\\
		& ~~~~ ~~~~ \times \int \alpha^{r_1 +2 }\int x^2 \bigl\lbrack (1+\eta) x\eta_{x\tau} - x\eta_x \eta_{\tau}\bigr\rbrack^2 \,dx\,d\tau, \\
		& \int L_{11} \,d\tau \lesssim \varepsilon \int \beta_1 \alpha^{-l_2} \int x^2 \abs{\zeta_\tau}{2} \,dx\,d\tau + C_\varepsilon \sup_\tau(\beta_1^{-2} \alpha^{-l_2 - r_2} + \beta_1^{-1} \alpha^{-l_2 - r_2 - 2} ) \\
		& ~~~~ \times \biggl( \int \beta_1 \alpha^{r_2} \int x^4 \overline \rho \abs{\eta_{\tau\tau}}{2} \,dx \,d\tau + \int \alpha^{r_2 + 2} \int x^2 \bigl\lbrack (1+\eta)x\eta_{x\tau\tau} - x\eta_{x}\eta_{\tau\tau}\bigr\rbrack^2 \,dx\,d\tau \biggr), \\
		& \int L_{12} \,d\tau \lesssim  \varepsilon \int  \alpha^{-l_2-1}\alpha_\tau \int x^2 \abs{\zeta_\tau}{2} \,dx\,d\tau \\
		& ~~~~ ~~~~ + C_\varepsilon \omega \sup_\tau(\beta_1^{-2} \alpha^{-l_2 - r_1 - 2\sigma_1} + \beta_1^{-1} \alpha^{-l_2 - r_1 - 2\sigma_1 - 2} )\\
		& ~~~~ ~~~~ \times \biggl( \int  \alpha^{r_1-1}\alpha_\tau \int x^4 \overline \rho \abs{\eta_{\tau}}{2} \,dx \,d\tau  + \int \alpha^{r_1 + 2} \int x^2 \bigl\lbrack (1+\eta)x\eta_{x\tau} - x\eta_{x}\eta_{\tau}\bigr\rbrack^2 \,dx\,d\tau \biggr), \\
		& \int L_{13} \,d\tau \lesssim \varepsilon \int \alpha^{-l_2- 1} \alpha_\tau \int x^2 \abs{\zeta_\tau}{2} \,dx\,d\tau + C_\varepsilon \omega
		 \sup_\tau (\alpha^{6\gamma-l_2 - r_1 - 2\sigma - 5} \alpha_\tau ) \\
		 & ~~~~ ~~~~ \times \int \alpha^{r_1 + 2} \int x^2 \bigl\lbrack(1+\eta)x\eta_{x\tau} - x\eta_x \eta_\tau\bigr\rbrack^2 \,dx\,d\tau, \\
		& \int L_{14} \,d\tau \lesssim \varepsilon \int \beta_1 \alpha^{-l_2} \int x^2 \abs{\zeta_\tau}{2} \,dx \,d\tau + C_\varepsilon \omega  \beta_1^{-1} \sup_\tau \alpha^{6\gamma - l_2 -r_1 - 4\sigma_1 - 4} \\
		& ~~~~\times  \int \alpha^{r_1+2} \int x^2 \bigl\lbrack (1+\eta) x\eta_{x\tau} - x\eta_x \eta_\tau \bigr\rbrack^2 \,dx \,d\tau \\
		& ~~~~ + C_\varepsilon \omega  \beta_1^{-1} \sup_\tau \alpha^{6\gamma - l_2 -r_2 - 2\sigma - 4} \times \int \alpha^{r_2+2} \int x^2 \bigl\lbrack (1+\eta) x\eta_{x\tau\tau} - x\eta_x \eta_{\tau\tau} \bigr\rbrack^2 \,dx \,d\tau ,
	\end{align*}
	where we have also applied inequality \eqref{est:hardy-01} with $ \eta_{\tau} $ replaced by $ \eta_{\tau\tau} $ and the fact that in $ L_{10} $,
	\begin{gather*}
	x \biggl(\dfrac{\eta_\tau}{1+\eta}\biggr)_x = \dfrac{(1+\eta+x\eta_x)}{(1+\eta)} \mathfrak B = \dfrac{(1+\eta)x\eta_{x\tau} - x\eta_x \eta_\tau }{(1+\eta)^2}, \\
	 - (x^3 (1+\eta)^2 \eta_{\tau\tau})_x \mathfrak B + 3 x^3(1+\eta)^2 \eta_{\tau\tau} \biggl( \dfrac{\eta_\tau}{1+\eta} \biggr)_x \\
	  = - x^2(1+\eta) ((1+\eta)x\eta_{x\tau\tau} - x\eta_x \eta_{\tau\tau})  \mathfrak B.
	\end{gather*}
	
	Consequently, after integrating \eqref{ene:003} and \eqref{ene:004} with respect to $ \tau $, choosing $ \varepsilon $ small enough, $ \beta_1 $ large enough, with $ \sigma = 1, \sigma_1 < \sigma $, and $ r_2, l_2 $ given in \eqref{indices}, 
	the estimates above and \eqref{lmest:L2} yield the lemma.
\end{proof}

\begin{proof}[Proof of Proposition \ref{prop:energy-estimates}]
	For fixed $ r_1 $ and $ \sigma_1 \in (0,\sigma) $, $ \sigma = 1 $, provided $ \omega $ small enough in \eqref{def:a-priori-assm} and $ \beta_1 $ large enough, summing up \eqref{lmest:L2}, \eqref{est:eta} and \eqref{lmest:L2-d-tau} implies \eqref{propest:energy} and finishes the proof of the proposition.  
\end{proof}

\section{Interior estimates}\label{sec:interior}
In this section, we are going to show the boundedness of the interior energy and dissipation functionals for a smooth solution to \eqref{eq:perturbation-0}. That is:
\begin{prop}[Interior estimates]\label{prop:interior-estimates}
Consider a smooth enough solution $ (\eta, \zeta) $ to system \eqref{eq:perturbation-0}, satisfying \eqref{bc:perturbation} and \eqref{initial:perturbation}. Suppose that \eqref{def:a-priori-assm} is satisfied with $ \omega \in (0,1) $ small enough, and that the expanding rate $ \beta_1 = \alpha_1 $ in \eqref{expandingrate} is large enough, we have
\begin{equation}\label{propest:interior}
	\begin{aligned}
	& \mathcal E_1 (T) + \mathcal D_1(T) \leq C_{r_1,\sigma_1,\beta_1,\beta_2} \mathcal E_{in} + C_{r_1,\sigma_1,\beta_1,\beta_2} \int \chi \bigl( \abs{\eta_0}{2} + \abs{x\eta_{0,x}}{2} \bigr) \,dx \\
	& ~~~~ + C_{r_1,\sigma_1,\beta_1,\beta_2} \bigl(\mathcal E_{in} + \mathcal E_0(T) +  \mathcal D_0(T) \bigr),
	\end{aligned}
\end{equation}
for any positive time $ T \in (0,\infty) $.  Here $ C_{r_1,\sigma_1,\beta_1,\beta_2}$ is some constant depending on $ r_1, \sigma_1,\beta_1,\beta_2$. 
\end{prop}

Without loss of generality, we always assume $ \beta_1 > 1 $ in the following . 
We establish Proposition \ref{prop:interior-estimates} in the following two lemmas. 
First, the following lemma investigates the Elliptic structure of \subeqref{eq:perturbation-0}{1}:
\begin{lm}\label{lm:elliptic-est-01}
Under the same assumptions as in Proposition \ref{prop:interior-estimates}, we have
\begin{gather}
	\alpha^{-l_3} \int \chi \abs{\zeta}{2} \,dx + \int_0^\tau \alpha^{-l_3 - 1} \alpha_\tau \int \chi \abs{\zeta}{2} \,dx \,d\tau' \leq C_{r_1, \sigma_1} \mathcal E_{in} \nonumber \\
	 ~~~~ + C_{r_1,\sigma_1} \beta_1^{-\varsigma} \int_0^\tau \alpha^{-l_3 - 1} \alpha_\tau \int \chi \abs{\zeta}{2} \,dx\,d\tau' \nonumber\\
	 ~~~~ + C_{r_1,\sigma_1} ( \beta_1^{-\varsigma} + \omega + 1 ) \bigl( \mathcal E_0 + \mathcal D_0 \bigr), 	\label{lmest:elliptic-01-01}\\
	 \int_0^\tau \alpha^{r_3} \int \chi (\abs{\eta_\tau}{2} + \abs{x\eta_{x\tau}}{2} )\,dx \,d\tau' \leq C_{r_1,\sigma_1,\beta_1,\beta_2}  \bigl( \mathcal E_0 + \mathcal D_0 \bigr) \nonumber \\
	 ~~~~ + C_{r_1,\sigma_1} \beta_1^{-\varsigma} \int_0^\tau \alpha^{-l_3 - 1} \alpha_\tau \int \chi \abs{\zeta}{2} \,dx\,d\tau' , \label{lmest:elliptic-01-02} \\
		\int x^2 \bigl\lbrack (1+\eta)x\eta_{x\tau} - x\eta_{x} \eta_\tau \bigr\rbrack^2\,dx \leq C_{r_1,\sigma_1,\beta_1,\beta_2} \alpha^{- \mathfrak a} \mathcal E_0 , \label{ene:008} \\
		 \int \chi ( \abs{\eta_\tau}{2} + \abs{x\eta_{x\tau}}{2} ) \,dx \leq C_{r_1,\sigma_1,\beta_1,\beta_2} \alpha^{-\mathfrak b} \bigl( \mathcal E_0 \nonumber \\
		  + \sup_{0 \leq \tau' \leq \tau}\alpha^{-l_3}(\tau') \int \chi \abs{\zeta(\tau')}{2} \,dx \bigr), \label{ene:009} \\
	 \int\chi (\abs{\eta}{2} + \abs{x\eta_{x}}{2} )\,dx \leq  \int\chi (\abs{\eta_0}{2} + \abs{x\eta_{0,x}}{2} )\,dx + C_{r_1,\sigma_1,\beta_1,\beta_2} \bigl(\mathcal E_0 \nonumber \\
	 + \sup_{0 \leq \tau' \leq \tau}\alpha^{-l_3}(\tau') \int \chi \abs{\zeta(\tau')}{2} \,dx \bigr),\label{ene:011}
\end{gather}
for any given $ \tau \in (0,T) $, some constant $ C_{r_1,\sigma_1} $ depends only on $ r_1, \sigma_1 $, some constant $ C_{r_1,\sigma_1,\beta_1,\beta_2} $ depends only on $ r_1, \sigma_1,\beta_1,\beta_2 $, and some constant $ \varsigma > 0 $.
Here $ r_3, l_3 $ are given in \eqref{indices}, $ \mathfrak a = r_1 + \sigma_1 + 1, \mathfrak b = r_1 $.

In partucular, for $ \beta_1 $ large enough, \eqref{lmest:elliptic-01-01} and \eqref{lmest:elliptic-01-02} implies
\begin{equation}\label{lmest:elliptic-01-03}
	\begin{aligned}
		& \alpha^{-l_3} \int \chi \abs{\zeta}{2} \,dx + \int_0^\tau \alpha^{-l_3 - 1} \alpha_\tau \int \chi \abs{\zeta}{2} \,dx \,d\tau' \\
		& ~~~~ + \int_0^\tau \alpha^{r_3} \int \chi (\abs{\eta_\tau}{2} + \abs{x\eta_{x\tau}}{2} )\,dx \,d\tau' \leq C_{r_1,\sigma_1} \mathcal E_{in} \\
		& ~~~~ + C_{r_1,\sigma_1,\beta_1, \beta_2} \bigl( \mathcal E_0 + \mathcal D_0 \bigr). 
	\end{aligned}
\end{equation}
	
\end{lm}

\begin{proof}
	After taking the $ L^2 $-inner product of \subeqref{eq:perturbation-0}{1} with $ \chi x \eta_\tau $ and applying integration by parts in the resultant, one has the following equation,
	\begin{equation}\label{elpt-est-01:001}
		\dfrac{4\mu}{3} \alpha^3 \int \chi \biggl\lbrace \dfrac{(\eta_\tau+x\eta_{x\tau})^2}{1+\eta+x\eta_x} + \dfrac{(1+\eta + x\eta_x)\eta_\tau^2}{(1+\eta)^2} \biggr\rbrace \,dx = I_{1} + I_2 + I_{3} + I_{4},
	\end{equation}
	where
	\begin{align*}
		& I_1 := - \dfrac{4\mu}{3} \alpha^3 \int \chi' x\eta_\tau \biggl( \dfrac{\eta_\tau+x\eta_{x\tau}}{1+\eta+x\eta} + \dfrac{\eta_\tau}{1+\eta} \biggr) \,dx \\
		& ~~~~ + \alpha^{4-3\gamma} \int \chi' \dfrac{x\zeta\eta_\tau}{ \bigl\lbrack (1+\eta+x\eta_x)(1+\eta)^2 \bigr\rbrack^\gamma} \,dx ,\\
		& I_2 := \alpha^{4-3\gamma} \int \chi \dfrac{\zeta(\eta_\tau +x\eta_{x\tau})}{\bigl\lbrack (1+\eta+x\eta_x)(1+\eta)^2 \bigr\rbrack^\gamma} \,dx,  \\
		& I_3 := - \alpha \int \chi \dfrac{x^2 \overline\rho \eta_{\tau\tau} \eta_\tau}{(1+\eta)^2} \,dx - \alpha_\tau 
		\int \chi \dfrac{x^2 \overline \rho \eta_\tau^2}{(1+\eta)^2} \,dx, \\
		& I_{4} := \alpha^{4-3\gamma} \delta \int \chi \dfrac{x^2 \overline\rho \eta \eta_\tau}{1+\eta} \,dx.
	\end{align*}
	Again, applying H\"older's and Young's inequalities to $ I_k $'s yields
	\begin{align*}
		& I_1 \lesssim \alpha^3 \bigg( \int x^4 \overline\rho \abs{\eta_\tau}{2} \,dx \biggr)^{1/2} \biggl( \int x^2 \bigl\lbrack (1+\eta) x\eta_{x\tau} - x\eta_x \eta_\tau\bigr\rbrack^2 \,dx \\
		& ~~~~ + \int x^4 \overline\rho \abs{\eta_\tau}{2} \,dx \biggr)^{1/2} + \alpha^{4-3\gamma} \biggl( \int x^2 \abs{\zeta}{2} \,dx \biggr)^{1/2} \bigg( \int x^4 \overline\rho \abs{\eta_\tau}{2} \,dx \biggr)^{1/2}, \\
		& I_2 \lesssim \varepsilon \alpha^{3} \int \chi ( \abs{\eta_\tau}{2} + \abs{x\eta_{x\tau}}{2} ) \,dx + C_\varepsilon \alpha^{5-6\gamma} \int \chi \abs{\zeta}{2} \,dx, \\
		& I_3 \lesssim \varepsilon \alpha^3 \int \chi \abs{\eta_\tau}{2} \,dx + C_\varepsilon \alpha^{-1} \int x^4 \overline\rho \abs{\eta_{\tau\tau}}{2} \,dx + C_\varepsilon \alpha^{-3} \alpha_\tau^2 \int x^4 \overline\rho \abs{\eta_\tau}{2} \,dx, \\
		& I_4 \lesssim \varepsilon \alpha^3 \int \chi \abs{\eta_\tau}{2} \,dx + C_\varepsilon \delta^2 \alpha^{5 - 6\gamma} \int x^4 \overline\rho \abs{\eta}{2} \,dx. 
	\end{align*}
	Consequently, after multiplying \eqref{elpt-est-01:001} with $ \alpha^{- 3} $
	, it follows
	\begin{equation}\label{ene:005}
		\begin{aligned}
		& \int \chi (\abs{\eta_\tau}{2} + \abs{x\eta_{x\tau}}{2} )\,dx \lesssim  \bigg( \int x^4 \overline\rho \abs{\eta_\tau}{2} \,dx \biggr)^{1/2} \\
		& ~~~~ \times \biggl( \int x^2 \bigl\lbrack (1+\eta) x\eta_{x\tau} - x\eta_x \eta_\tau\bigr\rbrack^2 \,dx + \int x^4 \overline\rho \abs{\eta_\tau}{2} \,dx \biggr)^{1/2} \\
		& ~~~~ + \alpha^{ - 3\gamma + 1} \biggl( \int x^2 \abs{\zeta}{2} \,dx \biggr)^{1/2} \bigg( \int x^4 \overline\rho \abs{\eta_\tau}{2} \,dx \biggr)^{1/2}\\
		& ~~~~ + \alpha^{ - 6\gamma + 2} \int\chi \abs{\zeta}{2} \,dx + \alpha^{-4} \int x^4 \overline\rho \abs{\eta_{\tau\tau}}{2} \,dx \\
		& ~~~~ + \alpha^{ - 6} \alpha_\tau^2 \int x^4 \overline\rho \abs{\eta_\tau}{2} \,dx + \delta^2 \alpha^{ - 6\gamma +2} \int x^4 \overline\rho \abs{\eta}{2} \,dx,
		\end{aligned}
	\end{equation}
	where we have chosen $ \varepsilon $ to be small enough. Meanwhile, taking the $ L^2 $-inner product of \subeqref{eq:perturbation-0}{2} with $ \alpha^{-l_3} \chi \zeta $ leads to
	\begin{equation}\label{ene:006}
		\dfrac{d}{dt} \biggl\lbrace \dfrac{\alpha^{-l_3}}{2}  \int \chi \abs{\zeta}{2} \,dx\biggr\rbrace + \dfrac{l_3}{2} \alpha^{-l_3-1} \alpha_\tau \int \chi \abs{\zeta}{2} \,dx
		 = L_{15} + L_{16},
	\end{equation}
	where
	\begin{align*}
		& L_{15} : = - \gamma \alpha^{-l_3} \int \chi \overline p \bigl\lbrack (1+\eta + x\eta_x )^{\gamma-1} (1+\eta)^{2\gamma} (\eta_\tau + x\eta_{x\tau}) \\
		& ~~~~ + 2 (1+\eta +x\eta_x)^\gamma (1+\eta)^{2\gamma-1} \eta_\tau 
		\bigr\rbrack \zeta \,dx \\
		& ~~~~ \lesssim \varepsilon  \alpha^{-l_3-1}\alpha_\tau \int \chi \abs{\zeta}{2}\,dx + C_\omega  \alpha^{-l_3+1}\alpha_\tau^{-1} \int \chi (\abs{\eta_\tau}{2} + \abs{x\eta_{x\tau}}{2})\,dx,\\
		& L_{16} := \dfrac{4\mu}{3}(\gamma-1) \alpha^{3\gamma - l_3 - 1} \int \chi \bigl\lbrack (1+\eta+x\eta_x) (1+\eta)^2 \bigr\rbrack^\gamma  \mathfrak B^2 \zeta \,dx \\
		& ~~~~ \lesssim \varepsilon \alpha^{-l_3-1}\alpha_\tau \int \chi \abs{\zeta}{2} \,dx + C_\varepsilon \omega \alpha^{6\gamma - l_3 - 2\sigma - 1}\alpha_\tau^{-1} \int\chi ( \abs{\eta_\tau}{2} + \abs{x\eta_{x\tau}}{2}  )\,dx. 
	\end{align*}
	Consequently, after integrating \eqref{ene:006} in $ \tau $ and substituting \eqref{ene:005}, with $ \varepsilon $ small enough, one has
		\begin{align*}
			& \sup_\tau \biggl\lbrace \dfrac{\alpha^{-l_3}}{2} \int \chi \abs{\zeta}{2}\,dx \biggr\rbrace + \int \alpha^{-l_3 - 1}\alpha_\tau \int \chi \abs{\zeta}{2} \,dx \,d\tau \\
			& ~~~~ \lesssim \bigl( \sup_\tau \alpha^{2\sigma - 6\gamma + 2 } + \omega \bigr) \biggl\lbrace  \sup_\tau( \alpha^{6\gamma - l_3 - 2\sigma -r_1-3/2 } \alpha_\tau^{-3/2} + \alpha^{6\gamma - l_3 - 2\sigma -r_1 } \alpha_\tau^{-2}) \\
			& ~~~~ \times \bigg( \int \alpha^{r_1-1}\alpha_\tau \int x^4 \overline\rho \abs{\eta_\tau}{2} \,dx \,d\tau \biggr)^{1/2}  \times \biggl( \int \alpha^{r_1 + 2}  \int x^2 \bigl\lbrack (1+\eta) x\eta_{x\tau} - x\eta_x \eta_\tau\bigr\rbrack^2 \,dx \,d\tau \\
			& ~~~~ ~~~~ + \int \alpha^{r_1-1} \alpha_\tau \int x^4 \overline\rho \abs{\eta_\tau}{2} \,dx \,d\tau  \biggr)^{1/2} \\
		& ~~~~ + \sup_\tau (\alpha^{6\gamma - l_3 - 2\sigma -r_1/2 + l_1/2 - 3\gamma + 1} \alpha_\tau^{-2}) \biggl( \int \alpha^{-l_1-1} \alpha_\tau \int x^2 \abs{\zeta}{2} \,dx \,d\tau \biggr)^{1/2} \\
		& ~~~~ ~~~~ \times  \bigg( \int \alpha^{r_1-1}\alpha_\tau \int x^4 \overline\rho \abs{\eta_\tau}{2} \,dx \,d\tau \biggr)^{1/2}\\
		& ~~~~ + \sup_\tau (\alpha^{-2\sigma + 2} \alpha_\tau^{-2}) \int \alpha^{-l_3 - 1} \alpha_\tau \int\chi \abs{\zeta}{2} \,dx\,d\tau \\
		& ~~~~ + \sup_\tau (\alpha^{6\gamma - l_3 -2\sigma -r_2-4} \alpha_\tau^{-2}) \int \alpha^{r_2 - 1} \alpha_\tau \int x^4 \overline\rho \abs{\eta_{\tau\tau}}{2} \,dx\,d\tau \\
		& ~~~~ + \sup_\tau \alpha^{6\gamma -l_3 - 2\sigma -r_1 - 6}  \int \alpha^{r_1-1}\alpha_\tau \int x^4 \overline\rho \abs{\eta_\tau}{2} \,dx \,d\tau \\
		& ~~~~ + \delta^2 \int  \alpha^{-l_3 - 2\sigma + 1}\alpha_\tau^{-1} \,d\tau \times \sup_\tau \int x^4 \overline\rho \abs{\eta}{2} \,dx \biggr\rbrace + \dfrac{\alpha^{-l_3}}{2} \int \chi \abs{\zeta}{2} \,dx \Big|_{\tau = 0} \\
		&  \lesssim \bigl( \sup_\tau \alpha^{2\sigma - 6\gamma + 2 } + \omega \bigr) \biggl\lbrace  \sup_\tau( \alpha^{6\gamma - l_3 - 2\sigma -r_1-3/2 } \alpha_\tau^{-3/2} + \alpha^{6\gamma - l_3 - 2\sigma -r_1 } \alpha_\tau^{-2}) \mathcal D_0 \\
		& ~~~~ + \sup_\tau (\alpha^{3\gamma - l_3 - 2\sigma -r_1/2 + l_1/2  + 1} \alpha_\tau^{-2}) \mathcal D_0  + \sup_\tau (\alpha^{-2\sigma + 2} \alpha_\tau^{-2}) \int \alpha^{-l_3 - 1} \alpha_\tau \int \chi \abs{\zeta}{2} \,dx\,d\tau \\
		& ~~~~ + \sup_\tau ( \alpha^{6\gamma - l_3 -2\sigma -r_2-4} \alpha_\tau^{-2}) \mathcal D_0 + \sup_\tau ( \alpha^{6\gamma -l_3 - 2\sigma -r_1 - 6}) \mathcal D_0 \\
		& ~~~~ + \delta^2  \mathcal E_0  \int \alpha^{-l_3 - 2\sigma + 1}\alpha_\tau^{-1} \,d\tau\biggr\rbrace + \dfrac{\alpha^{-l_3}}{2} \int \chi \abs{\zeta}{2} \,dx \Big|_{\tau = 0}.
		\end{align*}
	Then this finishes the proof of \eqref{lmest:elliptic-01-01} for $ \sigma = 1 $ and $ r_2, l_3 $ given in \eqref{indices}.
%
%
	
	 On the other hand, from \eqref{ene:005}, one has, 
	 \begin{align*}
	 	& \int \alpha^{r_3} \int \chi (\abs{\eta_\tau}{2} + \abs{x\eta_{x\tau}}{2} )\,dx \,d\tau \lesssim \sup_\tau ( \alpha^{r_3-r_1 - 1/2} \alpha_\tau^{-1/2} + \alpha^{r_3 - r_1 + 1} \alpha_\tau^{-1} ) \\
	 	& ~~~~ \times  \bigg( \int \alpha^{r_1-1}\alpha_\tau \int x^4 \overline\rho \abs{\eta_\tau}{2} \,dx\,d\tau  \biggr)^{1/2} \biggl(\int \alpha^{r_1+2} \int x^2 \bigl\lbrack (1+\eta) x\eta_{x\tau} - x\eta_x \eta_\tau\bigr\rbrack^2 \,dx \,d\tau \\
	 	& ~~~~ ~~~~ + \int  \alpha^{r_1-1} \alpha_\tau \int x^4 \overline\rho \abs{\eta_\tau}{2} \,dx \,d\tau \biggr)^{1/2} \\
		& ~~~~ + \sup_\tau (\alpha^{ r_3- 3\gamma + l_1/2 - r_1/2 + 2} \alpha_\tau^{-1}) \biggl(\int  \alpha^{-l_1-1} \alpha_\tau \int x^2 \abs{\zeta}{2} \,dx\,d\tau  \biggr)^{1/2} \\
		& ~~~~ ~~~~ \times \bigg( \int  \alpha^{r_1-1}\alpha_\tau \int x^4 \overline\rho \abs{\eta_\tau}{2} \,dx\,d\tau \biggr)^{1/2}\\
		& ~~~~ + \sup_\tau (\alpha^{r_3 - 6\gamma + l_3 + 3} \alpha_\tau^{-1}) \int \alpha^{-l_3-1}\alpha_\tau \int\chi \abs{\zeta}{2} \,dx\,d\tau \\
		& ~~~~  + \sup_\tau( \alpha^{r_3-r_2-3}\alpha_\tau^{-1})\int  \alpha^{r_2-1} \alpha_\tau \int x^4 \overline\rho \abs{\eta_{\tau\tau}}{2} \,dx \,d\tau \\
		& ~~~~ + \sup_\tau (\alpha^{r_3-r_1 - 5} \alpha_\tau )\int  \alpha^{r_1-1} \alpha_\tau \int x^4 \overline\rho \abs{\eta_\tau}{2} \,dx\,d\tau\\
		& ~~~~ + \delta^2 \int \alpha^{r_3 - 6\gamma +2}\,d\tau \times \sup_\tau \int x^4 \overline\rho \abs{\eta}{2} \,dx\\
		& ~~~~ \lesssim \sup_\tau ( \alpha^{r_3-r_1 - 1/2} \alpha_\tau^{-1/2} + \alpha^{r_3 - r_1 + 1} \alpha_\tau^{-1} ) \mathcal D_0 \\
		& ~~~~ + \sup_\tau (\alpha^{ r_3- 3\gamma + l_1/2 - r_1/2 + 2} \alpha_\tau^{-1}) \mathcal D_0 \\
		& ~~~~ + \sup_\tau (\alpha^{r_3 - 6\gamma + l_3 + 3} \alpha_\tau^{-1}) \int \alpha^{-l_3 - 1} \alpha_\tau \int \chi \abs{\zeta}{2} \,dx\,d\tau \\
		& ~~~~  + \sup_\tau( \alpha^{r_3-r_2-3}\alpha_\tau^{-1}) \mathcal D_0  + \sup_\tau (\alpha^{r_3-r_1 - 5} \alpha_\tau ) \mathcal D_0 \\
		& ~~~~ + \delta^2 \mathcal E_0 \int \alpha^{r_3 - 6\gamma +2}\,d\tau. 
	 \end{align*}
	 Then with $ l_1, r_2, l_3, r_3 $ given in \eqref{indices}, the above inequality yields \eqref{lmest:elliptic-01-02}.
	 
	 
%
	 
	 
	 In addition, after taking the $ L^2 $-inner product of \subeqref{eq:perturbation-0}{1} with $ x^3 (1+\eta)^2 \eta_\tau $, similar to \eqref{ene:001}, we have,
	 \begin{align*}
	 	& \dfrac{4 \mu }{3} \alpha^{3} \int \dfrac{x^2 \bigl\lbrack (1+\eta)x\eta_{x\tau} - x\eta_x \eta_\tau \bigr\rbrack^2}{1+\eta+x\eta_x} \,dx = \delta \alpha^{4-3\gamma} \int x^4 \overline\rho (1+\eta)\eta \eta_\tau\,dx \\
	 	& ~~~~ - \alpha^{4-3\gamma} \int \dfrac{x^2 (1+\eta) \bigl( (1+\eta) (\eta_\tau +x\eta_{x\tau}) + 2 (1+\eta+x\eta_x) \eta_\tau \bigr) \zeta}{\bigl\lbrack (1+\eta+x\eta_x) (1+\eta)^2 \bigr\rbrack^\gamma} \,dx\\
	 	& ~~~~ -  \alpha \int x^4 \overline\rho \eta_{\tau\tau} \eta_\tau \,dx - \alpha_\tau \int x^4 \overline\rho \abs{\eta_{\tau}}{2} \,dx \lesssim \delta \alpha^{-r_1/2 - 3\gamma + 4} \\
	 	& ~~~~ \times \biggl( \alpha^{r_1} \int x^4 \overline\rho \abs{\eta_\tau}{2}\,dx \biggr)^{1/2} \biggl( \int x^4 \overline\rho \abs{\eta}{2} \,dx \biggr)^{1/2} + \alpha^{l_1/2 +4-3\gamma} \\
	 	& ~~~~ \times \biggl( \alpha^{-l_1} \int x^2 \abs{\zeta}{2} \,dx \biggr)^{1/2} \biggl( \int x^2 \bigl\lbrack (1+\eta)x\eta_{x\tau} - x\eta_x \eta_\tau \bigr\rbrack^2 \,dx + \int x^4 \overline\rho \abs{\eta_\tau}{2} \,dx \biggr)^{1/2}\\
	 	& ~~~~ + \alpha^{1 - r_1/2 - r_2/2} \biggl( \alpha^{r_1} \int x^4 \overline\rho \abs{\eta_\tau}{2} \,dx \biggr)^{1/2} \biggl( \alpha^{r_2} \int x^4 \overline\rho \abs{\eta_{\tau\tau}}{2} \,dx\biggr)^{1/2} \\
	 	& ~~~~ + \alpha^{-r_1} \alpha_\tau  \times \alpha^{r_1} \int x^4 \overline\rho\abs{\eta_\tau}{2} \,dx 
	 	\lesssim  \varepsilon \alpha^3 \int x^2 \bigl\lbrack (1+\eta) x\eta_{x\tau} - x\eta_x \eta_\tau \bigr\rbrack^2 \, dx \\
	 	&~~~~ + \delta \alpha^{- {r_1} / 2 - 3\gamma + 4} \times \biggl( \alpha^{r_1} \int x^4 \overline\rho \abs{\eta_\tau}{2}\,dx \biggr)^{1/2} \biggl( \int x^4 \overline\rho \abs{\eta}{2} \,dx \biggr)^{1/2}\\
	 	& ~~~~ +C_\varepsilon \alpha^{l_1 - 6\gamma + 5} \times \alpha^{-l_1} \int x^2 \abs{\zeta}{2} \,dx \\
	 	& ~~~~ +  \alpha^{l_1/2-r_1/2- 3\gamma + 4} \biggl( \alpha^{r_1} \int x^4 \overline\rho \abs{\eta_\tau}{2} \, dx \biggr)^{1/2} \biggl( \alpha^{-l_1} \int x^2 \abs{\zeta}{2} \,dx \biggr)^{1/2}\\
	 	& ~~~~  + \alpha^{1 - {r_1}/2 - {r_2}/2}  \biggl( \alpha^{r_1} \int x^4 \overline\rho \abs{\eta_\tau}{2} \,dx \biggr)^{1/2} \biggl( \alpha^{r_2} \int x^4 \overline\rho \abs{\eta_{\tau\tau}}{2} \,dx \biggr)^{1/2} \\
	 	& ~~~~ + \alpha^{-r_1} \alpha_\tau  \times \alpha^{r_1} \int x^4 \overline\rho\abs{\eta_\tau}{2} \,dx,
	 \end{align*}
	 or, equivalently, after choosing $ \varepsilon $ small enough,
	 \begin{equation*}\tag{\ref{ene:008}}
	 	\begin{aligned}
	 	& \int x^2 \bigl\lbrack (1+\eta)x\eta_{x\tau} - x\eta_{x} \eta_\tau \bigr\rbrack^2\,dx \lesssim C_{\beta_1,\beta_2}\bigl( \delta \alpha^{-r_1/2 - 3\gamma + 1} + \alpha^{l_1 - 6\gamma + 2} \\
	 	& ~~~~ + \alpha^{l_1/2 - r_1/2 - 3\gamma + 1} + \alpha^{-r_1/2 - r_2/2 - 2} + \alpha^{-r_1 - 2}  \bigr) \mathcal E_0 \lesssim C_{\beta_1,\beta_2} \alpha^{-\mathfrak{a}} \mathcal E_0.
	 	\end{aligned}
	 \end{equation*}
	 With \eqref{ene:008} and \eqref{ene:005}, one has
	\begin{equation*}\tag{\ref{ene:009}}
	\begin{aligned}
	 	& \int \chi ( \abs{\eta_\tau}{2} + \abs{x\eta_{x\tau}}{2} ) \,dx \lesssim C_{\beta_1,\beta_2} \bigl( \alpha^{-r_1} + \alpha^{-r_1/2 - \mathfrak a/2 } \\
	 	& ~~~~ + \alpha^{-r_1/2 + l_1/2 -3\gamma+1} + \alpha^{l_3-6\gamma + 2} + \alpha^{-r_2 - 4} + \alpha^{-r_1 - 4} + \alpha^{-6\gamma + 2} \bigr) \\
	 	& ~~~~ \times \bigl( \mathcal E_0 + \sup_{\tau} \alpha^{-l_3} \int \chi \abs{\zeta}{2} \,dx \bigr) \lesssim C_{\beta_1,\beta_2} \alpha^{-\mathfrak{b}} \bigl( \mathcal E_0 + \sup_{\tau}\alpha^{-l_3} \int \chi \abs{\zeta}{2} \,dx \bigr) . 
	 \end{aligned}
	 \end{equation*}
	 It is easy to check
	 \begin{equation}\label{ene:010}
	 	\begin{aligned}
	 		& \mathfrak a = \min \bigl\lbrace 3\gamma + r_1/2 - 1, 6\gamma - l_1 - 2, 3\gamma +r_1 /2 - l_1/2 - 1, \\
	 		& ~~~~ ~~~~ r_1/2 + r_2/2 + 2, r_1 + 2  \bigr\rbrace = r_1 + \sigma_1 + 1  > 0 , \\
	 		& \mathfrak b = \min \bigl\lbrace r_1, r_1/2 + \mathfrak a / 2, 3\gamma + r_1/2 - l_1/2 - 1, 6\gamma - l_3 - 2\\
	 		& ~~~~ ~~~~ r_2 + 4, r_1 + 4, 6\gamma - 2  \bigr\rbrace = r_1 > 0. 
	 	\end{aligned}
	 \end{equation}
	 Therefore, after using the fundamental theorem of calculus, \eqref{ene:009} implies \eqref{ene:011}. 
%
%
%
\end{proof}


The next lemma is concerning the $ L^2 $ estimate of $ \eta_{\tau\tau} $ and $ \zeta_\tau $ in the interior domain. 
\begin{lm}\label{lm:interior-energy}
	Under the same assumptions as in Proposition \ref{prop:interior-estimates}, we have
	\begin{equation}\label{lmest:interior-energy}
		\begin{aligned}
			& \alpha^{-r_4} \int \chi x^2 \overline\rho \abs{\eta_{\tau\tau}}{2} \,dx + \alpha^{-l_4} \int \chi \abs{\zeta_\tau}{2} \,dx + \int_0^\tau \alpha^{-r_4 - 1} \alpha_\tau \int \chi x^2 \overline\rho \abs{\eta_{\tau\tau}}{2} \,dx\,d\tau' \\
			& ~~~~ + \int_0^\tau \alpha^{2-r_4} \int \chi \bigl( \abs{\eta_{\tau\tau}}{2} + \abs{x\eta_{x\tau\tau}}{2} \bigr) \,dx\,d\tau'  + \int_0^\tau \alpha^{-l_4 - 1} \alpha_\tau \int \chi \abs{\zeta_\tau}{2} \,dx\,d\tau' \\
			& ~~ \leq C_{r_1,\sigma_1} \mathcal E_{in} + C_{r_1,\sigma_1} \bigl( \beta_1^{-\varsigma} + \omega \bigr) \biggl( \int_0^\tau \alpha^{-l_4 - 1} \alpha_\tau \int \chi \abs{\zeta_\tau}{2} \,dx\,d\tau' \\
			& ~~~~ ~~~~ +  \int_0^\tau \alpha^{2-r_4} \int \chi \bigl( \abs{\eta_{\tau\tau}}{2} 
			 + \abs{x\eta_{x\tau\tau}}{2} \bigr) \,dx\,d\tau' \biggr) \\
			& ~~~~ + C_{r_1,\sigma_1, \beta_1, \beta_2}\biggl( \sup_\tau \int \chi \abs{\eta}{2} \,dx + \int_0^\tau \alpha^{r_3} \int \chi \bigl( \abs{\eta_\tau}{2} + \abs{x\eta_{x\tau}}{2} \bigr) \,dx \,d\tau' \\
			& ~~~~ ~~~~ + \int_0^\tau \alpha^{-l_3-1} \alpha_\tau \int \chi \abs{\zeta}{2} \,dx\,d\tau' + \mathcal D_0 \biggr),
		\end{aligned}
	\end{equation}
	for any given $ \tau \in (0,T) $, some constant $ C_{r_1,\sigma_1} $ depends only on $ r_1, \sigma_1 $,  some constant $ C_{r_1,\sigma_1,\beta_1,\beta_2} $ depends only on $ r_1, \sigma_1,\beta_1,\beta_2 $, and some constant $ \varsigma > 0 $.
	Here $ r_4, l_4 $ are given in \eqref{indices}. 
\end{lm}

\begin{proof}
	We are going to estimate the $ L^2 $ norm of $ \int \chi x^2 \overline\rho \abs{\eta_{\tau\tau}}{2} \,dx $ and $ \int \chi \abs{\zeta_\tau}{2} \,dx $ with appropriate time weights. After taking the $ L^2 $-inner product of \subeqref{eq:perturbation-1}{1} with $ \alpha^{-r_4-1} \chi x \eta_{\tau\tau} $ and \subeqref{eq:perturbation-1}{2}  with $ \alpha^{-l_4} \chi \zeta_\tau $, the resultant equations are
	\begin{equation}\label{ene:012}
	\begin{aligned}
		& \dfrac{d}{d\tau} \biggl\lbrace \dfrac{\alpha^{-r_4}}{2} \int \chi x^2 \overline\rho \abs{\eta_{\tau\tau}}{2} \,dx \biggr\rbrace + \biggl( \dfrac{r_4}{2} + 2\biggr) \alpha^{-r_4 - 1} \alpha_\tau  \int \chi x^2 \overline\rho \abs{\eta_{\tau\tau}}{2} \,dx \\
		& ~~~~ + \dfrac{4\mu}{3} \alpha^{2-r_4} \int \chi \biggl\lbrack \dfrac{(1+\eta)^2 x^2 \eta_{x\tau\tau}^2}{1+\eta+x\eta_x} + \dfrac{ \bigl\lbrack 2(1+\eta +x\eta_x)^2 - (1+\eta)^2 \bigr\rbrack \eta_{\tau\tau}^2}{1+\eta + x\eta_{x}} \biggr\rbrack\,dx\\
		& ~~ = \sum_{i=17}^{26} L_{i}, 
	\end{aligned}
	\end{equation}
	and
	\begin{equation}\label{ene:013}
		\begin{aligned}
			& \dfrac{d}{d\tau} \biggl\lbrace \dfrac{\alpha^{-l_4}}{2} \int \chi \abs{\zeta_\tau}{2}\,dx \biggr\rbrace + \dfrac{l_4}{2} \alpha^{-l_4-1} \alpha_\tau \int \chi \abs{\zeta_\tau}{2} \,dx = \sum_{i=27}^{30} L_i,
		\end{aligned}
	\end{equation}
	where
	\begin{align*}
		& L_{17} := - \alpha^{-r_4 - 1} \alpha_{\tau\tau} \int \chi x^2 \overline\rho\eta_\tau \eta_{\tau\tau} \,dx, \\
		& L_{18} := \delta \alpha^{3-3\gamma -r_4} \int \chi x^2 \overline\rho \bigl\lbrace (1+\eta)\eta \bigr\rbrace_\tau \eta_{\tau\tau} \,dx, \\
		& L_{19} := (4-3\gamma) \delta \alpha^{2 - 3\gamma - r_4} \alpha_\tau \int \chi x^2 \overline\rho (1+\eta) \eta \eta_{\tau\tau} \,dx, \\
		& L_{20} := \int \chi \biggl\lbrace 2 \alpha^{4-3\gamma} (1+\eta) x\eta_x \biggl( \dfrac{\zeta}{\bigl\lbrack (1+\eta+x\eta_x) (1+\eta)^2 \bigr\rbrack^\gamma} \biggr)  \biggr\rbrace_\tau \alpha^{-r_4 - 1} \eta_{\tau\tau} \,dx \\
		& ~~~~ + \int \chi \biggl\lbrace \alpha^{4-3\gamma} (1+\eta)^2 \biggl( \dfrac{\zeta}{\bigl\lbrack (1+\eta+x\eta_x) (1+\eta)^2 \bigr\rbrack^\gamma} \biggr)  \biggr\rbrace_\tau \alpha^{-r_4 - 1} (\eta_{\tau\tau} + x\eta_{x\tau\tau} )\,dx, \\
		& L_{21} := \int \chi' x \biggl\lbrace \alpha^{4-3\gamma} (1+\eta)^2 \biggl( \dfrac{\zeta}{\bigl\lbrack (1+\eta+x\eta_x) (1+\eta)^2 \bigr\rbrack^\gamma} \biggr) \biggr\rbrace_\tau \alpha^{-r_4 - 1} \eta_{\tau\tau} \,dx, \\
		& L_{22} := \dfrac{4}{3} \mu \alpha^{2-r_4} \int \chi \biggl\lbrack \dfrac{(\eta_\tau + x\eta_{x\tau})^2}{(1+\eta + x\eta_x)^2} + 2 \dfrac{\eta_\tau^2}{(1+\eta)^2} \biggr\rbrack \\
		& ~~~~ ~~~~ \times \bigl\lbrack (1+\eta)(1+\eta+ 2x\eta_x) \eta_{\tau\tau} + (1+\eta)^2 x\eta_{x\tau\tau} \bigr\rbrack \,dx \\
		& ~~~~ - \dfrac{8}{3} \mu \alpha^{2-r_4} \int \chi \biggl\lbrack \mathfrak B + 3 \biggl( \dfrac{\eta_\tau}{1+\eta} \biggr) \biggr\rbrack \times \biggl\lbrack (1+\eta + x\eta_x )\eta_\tau \eta_{\tau\tau} + (1+\eta) x\eta_{x\tau} \eta_{\tau\tau} \\
		& ~~~~ ~~~~ + (1+\eta) \eta_\tau x\eta_{x\tau\tau}  \biggr\rbrack  \,dx  , \\
		& L_{23} :=  - 4\mu \alpha^{1-r_4} \alpha_\tau \int \chi \biggl\lbrack \mathfrak B + 3 \biggl( \dfrac{\eta_\tau}{1+\eta} \biggr) \biggr\rbrack \times \biggl\lbrack (1+\eta)^2 \eta_{\tau\tau} + 2(1+\eta) x\eta_x \eta_{\tau\tau} + (1+\eta)^2 x\eta_{x\tau\tau}  \biggr\rbrack \,dx ,\\
		& L_{24} := - \dfrac{4}{3} \mu \alpha^{2-r_4} \int \chi' x \biggl\lbrace (1+\eta)^2 \biggl\lbrack \mathfrak B_\tau + 3 \biggl( \dfrac{\eta_\tau}{1+\eta} \biggr)_\tau  \biggr\rbrack - 2 (1+\eta) \eta_{\tau\tau} \biggr\rbrace \eta_{\tau\tau}  \,dx,  \\
		& L_{25}:= - \dfrac{8}{3} \mu \alpha^{2-r_4} \int \chi' x \biggl\lbrack \mathfrak B + 3 \biggl( \dfrac{\eta_\tau}{1+\eta} \biggr) \biggr\rbrack \times  (1+\eta) \eta_\tau\eta_{\tau\tau} \,dx , \\
		& L_{26} := - 4\mu \alpha^{1-r_4} \alpha_\tau \int \chi' x \biggl\lbrack \mathfrak B + 3 \biggl( \dfrac{\eta_\tau}{1+\eta} \biggr) \biggr\rbrack \times (1+\eta)^2 \eta_{\tau\tau} \,dx , \\
		& L_{27} := - \gamma \alpha^{-l_4} \int \chi \overline p (1+\eta + x\eta_x)^\gamma (1+\eta)^{2\gamma} \biggl( \dfrac{\eta_{\tau\tau} + x\eta_{x\tau\tau}}{1+\eta+x\eta_x} + 2 \dfrac{\eta_{\tau\tau}}{1+\eta}\biggr) \zeta_\tau \,dx ,\\
		& L_{28} := - \gamma \alpha^{-l_4} \int \chi \overline p \biggl\lbrace \bigl\lbrack (1+\eta_x + x\eta_x)^{\gamma-1} (1+\eta)^{2\gamma} \bigr\rbrack_\tau (\eta_\tau + x\eta_{x\tau}) \zeta_\tau \\
		& ~~~~ ~~~~ + \bigl\lbrack (1+\eta_x + x\eta_{x})^\gamma (1+\eta)^{2\gamma-1} \bigr\rbrack_\tau \eta_\tau \zeta_\tau \,dx ,\\
		& L_{29} := \dfrac{4}{3} \mu (\gamma-1) (3\gamma-1) \alpha^{3\gamma - l_4 - 2} \alpha_\tau \int \chi \bigl\lbrack (1+\eta_x +x\eta_x)(1+\eta)^2 \bigr\rbrack^\gamma \mathfrak B^2 \zeta_\tau \,dx, \\
		& L_{30} := \dfrac{4}{3} \mu (\gamma-1) \alpha^{3\gamma-l_4 - 1} \int \chi  \bigl\lbrace \bigl\lbrack (1+\eta_x +x\eta_x)(1+\eta)^2 \bigr\rbrack^\gamma  \mathfrak B^2 \bigr\rbrace_\tau \zeta_\tau \,dx.
	\end{align*}
	Next, we apply the H\"older's and Young's inequalities to estimate the right hand side of \eqref{ene:012} and \eqref{ene:013}. Indeed, we have, as before,
	\begin{align*}
		& \int L_{17} \,d\tau \lesssim \beta_2^2 \int \alpha^{-r_4} \int \chi x^2 \overline\rho \abs{\eta_\tau}{} \abs{\eta_{\tau\tau}}{} \,dx \,d\tau \\
		& ~~~~ \lesssim \varepsilon \int \alpha^{-r_4 - 1} \alpha_\tau \int \chi x^2 \overline \rho \abs{\eta_{\tau\tau}}{2} \,dx\,d\tau \\
		& ~~~~ ~~~~ + C_\varepsilon \beta_2^4 \beta_1^{-1} \sup_\tau \alpha^{-r_3 - r_4} \int \alpha^{r_3} \int \chi \abs{\eta_\tau}{2} \,dx\,d\tau, \\
		& \int L_{18} \,d\tau \lesssim \delta (1+\omega) \int \alpha^{3-3\gamma - r_4} \int \chi x^2 \overline \rho \abs{\eta_\tau}{} \abs{\eta_{\tau\tau}}{} \,dx\,d\tau \\
		& ~~~~ \lesssim \varepsilon \int \alpha^{-r_4 - 1}\alpha_\tau \int \chi x^2 \overline\rho \abs{\eta_{\tau\tau}}{2} \,dx \,d\tau \\
		& ~~~~ ~~~~ + C_\varepsilon \delta^2 \beta_1^{-1} \sup_\tau \alpha^{6 - 6\gamma - r_3 - r_4} \int \alpha^{r_3} \int \chi \abs{\eta_{\tau}}{2} \,dx\,d\tau , \\
		& \int L_{19} \,d\tau \lesssim \delta \int \alpha^{2-3\gamma-r_4}\alpha_\tau \int \chi x^2 \overline\rho \abs{\eta}{} \abs{\eta_{\tau\tau}}{} \,dx\,d\tau \\
		& ~~~~ \lesssim \varepsilon \int \alpha^{-r_4 - 1}\alpha_\tau \int \chi x^2 \overline \rho \abs{\eta_{\tau\tau}}{2} \,dx\,d\tau \\
		& ~~~~ ~~~~ + C_\varepsilon \delta^2 \beta_2 \int \alpha^{6-6\gamma - r_4} \,d\tau \times \sup_\tau \int \chi \abs{\eta}{2} \,dx , \\
		& \int L_{20} \,d\tau \lesssim  (1+\omega) \int \alpha^{2-3\gamma-r_4} \alpha_\tau \int \chi \abs{\zeta}{} ( \abs{\eta_{\tau\tau}}{} + \abs{x\eta_{x\tau\tau}}{} ) \,dx \,d\tau \\
		& ~~~~ ~~~~ + \omega \int \alpha^{3-3\gamma-r_4 - \sigma_1} \int \chi \abs{\zeta}{} ( \abs{\eta_{\tau\tau}}{} + \abs{x\eta_{x\tau\tau}}{} ) \,dx \,d\tau  \\
		& ~~~~ ~~~~ + (1+\omega) \int \alpha^{3-3\gamma - r_4} \int \chi \abs{\zeta_\tau}{} ( \abs{\eta_{\tau\tau}}{} + \abs{x\eta_{x\tau\tau}}{} ) \,dx \,d\tau \\
		& ~~~~ \lesssim \varepsilon \int \alpha^{2-r_4} \int \chi (\abs{\eta_{\tau\tau}}{2} + \abs{x\eta_{x\tau\tau}}{2} )\,dx \,d\tau \\
		& ~~~~ ~~~~ + C_\varepsilon \sup_\tau \bigl( \alpha^{3 - 6 \gamma +l_3 - r_4} \alpha_\tau + \omega \alpha^{5 - 6\gamma - 2 \sigma_1 + l_3 - r_4} \alpha_\tau^{-1}   \bigr) \int \alpha^{-l_3 - 1}\alpha_\tau \int \chi \abs{\zeta}{2} \,dx\,d\tau \\
		& ~~~~ ~~~~ + C_\varepsilon \sup_\tau ( \alpha^{5-6\gamma + l_4 - r_4} \alpha_\tau^{-1} ) \int \alpha^{-l_4-1} \alpha_\tau \int \chi \abs{\zeta_\tau}{2} \,dx\,d\tau , \\
		& \int L_{21} \,d\tau \lesssim (1+\omega) \int \alpha^{2-3\gamma - r_4} \alpha_\tau \int \abs{\chi'}{} x \abs{\zeta}{} \abs{\eta_{\tau\tau}}{} \,dx\,d\tau \\
		& ~~~~ ~~~~ + (1+\omega) \int \alpha^{3-3\gamma - r_4 } \int \abs{\chi'}{} x \abs{\zeta_\tau}{} \abs{\eta_{\tau\tau}}{} \,dx \,d\tau \\
		& ~~~~ ~~~~ +  \omega \int \alpha^{3-3\gamma-r_4 - \sigma_1} \int \abs{\chi'}{} x \abs{\zeta}{} \abs{\eta_{\tau\tau}}{} \,dx\,d\tau \\
		& ~~~~ \lesssim \sup_\tau \alpha^{3 - 3\gamma - r_4 - r_2/2 + l_1/2} \biggl( \int \alpha^{-l_1-1} \alpha_\tau \int x^2 \abs{\zeta}{2} \,dx\,d\tau \biggr)^{1/2} \\ & ~~~~ ~~~~ \times \biggl( \int \alpha^{r_2 - 1} \alpha_\tau \int x^4 \overline\rho \abs{\eta_{\tau\tau}}{2} \,dx\,d\tau \biggr)^{1/2} 
		+ \sup_\tau ( \alpha^{4 - 3\gamma - r_4 - r_2/2 + l_2/2} \alpha_\tau^{-1} ) \\
		& ~~~~ ~~~~ \times \biggl( \int \alpha^{-l_2- 1}\alpha_\tau \int x^2 \abs{\zeta_\tau}{2} \,dx\,d\tau \biggr)^{1/2} \biggl(  \int \alpha^{r_2 - 1} \alpha_\tau \int x^4 \overline\rho \abs{\eta_{\tau\tau}}{2} \,dx \,d\tau\biggr)^{1/2} \\
		& ~~~~ +  \omega \sup_\tau ( \alpha^{4 - 3\gamma - r_4 -r_2/2 + l_1/2 - \sigma_1} \alpha_\tau^{-1} )\biggl( \int \alpha^{-l_1-1} \alpha_\tau \int x^2 \abs{\zeta}{2} \,dx\,d\tau\biggr)^{1/2} \\
		& ~~~~ ~~~~ \times \biggl( \int \alpha^{r_2 -1} \alpha_\tau \int x^4 \overline \rho \abs{\eta_{\tau\tau}}{2} \,dx\,d\tau \biggr)^{1/2}  , \\
		& \int L_{22} \,d\tau \lesssim \omega \int \alpha^{2-r_4-\sigma_1} \int \chi ( \abs{\eta_\tau}{} + \abs{x\eta_{x\tau}}{}) ( \abs{\eta_{\tau\tau}}{} + \abs{x\eta_{x\tau\tau}}{} ) \,dx\,d\tau \\
		& ~~~~ \lesssim \varepsilon \int \alpha^{2-r_4} \int \chi  ( \abs{\eta_{\tau\tau}}{2} + \abs{x\eta_{x\tau\tau}}{2} ) \,dx\,d\tau \\
		& ~~~~ ~~~~ + C_\varepsilon \omega \sup_\tau \alpha^{2-r_4 - r_3 - 2\sigma_1} \int \alpha^{r_3} \int \chi ( \abs{\eta_\tau}{2} + \abs{x\eta_{x\tau}}{2} ) \,dx\,d\tau  , \\
		& \int L_{23} \,d\tau \lesssim (1+\omega) \int \alpha^{1-r_4}\alpha_\tau \int \chi (\abs{\eta_\tau}{} + \abs{x\eta_{x\tau}}{} )( \abs{\eta_{\tau\tau}}{} + \abs{x\eta_{x\tau\tau}}{} ) \,dx\,d\tau \\
		& ~~~~ \lesssim \varepsilon \int \alpha^{2-r_4} \int \chi ( \abs{\eta_{\tau\tau}}{2} + \abs{x\eta_{x\tau\tau}}{2}) \,dx\,d\tau \\
		& ~~~~ ~~~~ + C_\varepsilon \sup_\tau (\alpha^{-r_3-r_4} \alpha_\tau^2 ) \int \alpha^{r_3} \int \chi (\abs{\eta_\tau}{2} + \abs{x\eta_{x\tau}}{2}) \,dx\,d\tau , \\
		& \int L_{24} + L_{25} \,d\tau \lesssim (1+\omega) \int \alpha^{2-r_4} \int \abs{\chi'}{} x  ( \abs{\eta_{\tau\tau}}{} + \abs{x\eta_{x\tau\tau}}{} ) \abs{\eta_{\tau\tau}}{} \,dx\,d\tau\\
		& ~~~~ + (1+\omega) \int \alpha^{2-r_4} \int \abs{\chi'}{} x ( \abs{\eta_\tau}{2} + \abs{x\eta_{x\tau}}{2} ) \abs{\eta_{\tau\tau}}{} \,dx\,d\tau \\
		& ~~~~ \lesssim \sup_\tau ( \alpha^{3-r_2-r_4} \alpha_\tau^{-1} + \alpha^{3/2 - r_2 - r_4} \alpha_\tau^{-1/2} ) \biggl( \int \alpha^{r_2 - 1} \alpha_\tau \int x^4 \overline \rho \abs{\eta_{\tau\tau}}{2} \,dx \,d\tau \biggr)^{1/2} \\
		& ~~~~ ~~~~ \times \biggl( \int \alpha^{r_2 - 1} \alpha_\tau \int x^4 \overline\rho \abs{\eta_{\tau\tau}}{2} \,dx \,d\tau + \int \alpha^{r_2 +2} \int x^2 \bigl\lbrack (1+\eta)x\eta_{x\tau\tau} - x\eta_x \eta_{\tau\tau} \bigr\rbrack^2 \,dx \,d\tau \biggr)^{1/2} \\
		& ~~~~ + \omega \sup_{\tau} ( \alpha^{3-r_1/2 -r_2/2 - r_4 - \sigma_1} \alpha_\tau^{-1} + \alpha^{3/2 - r_1/2 - r_2/2 - r_4 - \sigma_1 } \alpha_\tau^{-1/2} ) \biggl( \int \alpha^{r_2 - 1} \alpha_\tau \int x^4 \overline\rho \abs{\eta_{\tau\tau}}{2} \,dx \,d\tau \biggr)^{1/2} \\
		& ~~~~ ~~~~ \times  \biggl( \int \alpha^{r_1 - 1} \alpha_\tau \int x^4 \overline \rho \abs{\eta_\tau}{2} \,dx\,d\tau + \int \alpha^{r_1 + 2} \int x^2 \bigl\lbrack (1+\eta) x\eta_{x\tau} - x\eta_x \eta_\tau \bigr\rbrack^2 \,dx\,d\tau \biggr)^{1/2}   , \\
		& \int L_{26} \,d\tau \lesssim    \sup_{\tau} ( \alpha^{2-r_1/2 -r_2/2 - r_4 } + \alpha^{1/2 - r_1/2 - r_2/2 - r_4  } \alpha_\tau^{1/2} ) \biggl( \int \alpha^{r_2 - 1} \alpha_\tau \int x^4 \overline\rho \abs{\eta_{\tau\tau}}{2} \,dx \,d\tau \biggr)^{1/2} \\
		& ~~~~ ~~~~ \times  \biggl( \int \alpha^{r_1 - 1} \alpha_\tau \int x^4 \overline \rho \abs{\eta_\tau}{2} \,dx\,d\tau + \int \alpha^{r_1 + 2} \int x^2 \bigl\lbrack (1+\eta) x\eta_{x\tau} - x\eta_x \eta_\tau \bigr\rbrack^2 \,dx\,d\tau \biggr)^{1/2}   , \\
		& \int L_{27} \,d\tau \lesssim  (1+\omega) \int \alpha^{-l_4} \int \chi \abs{\zeta_\tau}{} ( \abs{\eta_{\tau\tau}}{} + \abs{x\eta_{x\tau\tau}}{} ) \,dx\,d\tau \\
		& ~~~~ \lesssim \varepsilon \int \alpha^{-l_4 - 1} \alpha_\tau \int \chi \abs{\zeta_{\tau}}{2} \,dx\,d\tau \\
		& ~~~~ ~~~~ + C_\varepsilon \sup_\tau ( \alpha^{r_4 - l_4 - 1} \alpha_\tau^{-1} ) \times \int \alpha^{2-r_4} \int \chi (\abs{\eta_{\tau\tau}}{2} + \abs{x\eta_{x\tau\tau}}{2} )\,dx\,d\tau , \\
		& \int L_{28} \,d\tau  \lesssim \varepsilon \int \alpha^{-l_4 - 1} \alpha_\tau \int \chi \abs{\zeta_{\tau}}{2} \,dx\,d\tau \\
		& ~~~~ ~~~~ + C_\varepsilon \omega \sup_\tau ( \alpha^{-r_3 - l_4 - 2\sigma_1 + 1} \alpha_\tau^{-1} ) \times \int \alpha^{r_3} \int \chi (\abs{\eta_{\tau}}{2} + \abs{x\eta_{x\tau}}{2} )\,dx\,d\tau  , \\
		& \int L_{29} \,d\tau \lesssim  (1+\omega) \int \alpha^{3\gamma - l_4 - 2 } \alpha_\tau \int \chi \abs{\mathfrak B}{2}  \abs{\zeta_\tau}{} \,dx \,d\tau \\
		& ~~~~ \lesssim \varepsilon \int \alpha^{-l_4 - 1} \alpha_\tau \int \chi \abs{\zeta_\tau}{2} \,dx\,d\tau\\
		& ~~~~ ~~~~ + C_\varepsilon \omega \sup_\tau ( \alpha^{6\gamma - 3 - l_4 - r_3 - 2\sigma} \alpha_\tau ) \int \alpha^{r_3} \int \chi (\abs{\eta_\tau}{2} + \abs{x\eta_{x\tau}}{2} ) \,dx\,d\tau  , \\
		& \int L_{30} \,d\tau \lesssim  (1+\omega) \int \alpha^{3\gamma - l_4 - 1} \int \chi ( \abs{\eta_\tau}{2} + \abs{x\eta_{x\tau}}{2} ) \abs{\mathfrak B}{} \abs{\zeta_\tau}{} \,dx \,d\tau \\
		& ~~~~ + (1+\omega) \int \alpha^{3\gamma - l_4 - 1} \int \chi \abs{\mathfrak B}{}  ( \abs{\eta_{\tau\tau}}{} + \abs{x\eta_{x\tau\tau}}{} ) \abs{\zeta_\tau}{} \,dx\,d\tau \\
		& ~~~~ \lesssim \varepsilon \int \alpha^{-l_4 - 1} \alpha_\tau \int \chi \abs{\zeta_\tau}{2} \,dx\,d\tau + C_\varepsilon \omega \sup_\tau ( \alpha^{6\gamma - l_4 -r_3 - 2\sigma - 2\sigma_1 - 1} \alpha_\tau^{-1}) \\
		& ~~~~ \times \int \alpha^{r_3} \int \chi (\abs{\eta_\tau}{2} + \abs{x\eta_{x\tau}}{2} ) \,dx\,d\tau + C_\varepsilon \omega \sup_\tau ( \alpha^{6\gamma - l_4 +r_4 - 2\sigma -3} \alpha_\tau^{-1} ) \\
		& ~~~~ \times \int \alpha^{2-r_4} \int \chi ( \abs{\eta_{\tau\tau}}{2} + \abs{x\eta_{x\tau\tau}}{2} ) \,dx\,d\tau .
	\end{align*}
Therefore, integrating \eqref{ene:012} and \eqref{ene:013} with respect to $ \tau $ yields the lemma, with small enough $ \varepsilon $ and the indices given in \eqref{indices}. 
%
\end{proof}

\begin{proof}[Proof of Proposition \ref{prop:interior-estimates}]
	For $ \beta_1 $ large enough, $ \omega $ small enough, \eqref{propest:interior} is the consequence of the estimates in \eqref{ene:011}, \eqref{lmest:elliptic-01-03} and \eqref{lmest:interior-energy}.  This finishes the proof of the proposition. 
\end{proof}

\section{Point-wise estimates}\label{sec:point-wise}

In this section, we aim at closing the a priori assumption in \eqref{def:a-priori-assm}. In fact, from sections \ref{sec:energy-estimate} and \ref{sec:interior}, we already have the boundedness of the total energy and dissipation functionals. In fact, summarizing Proposition \ref{prop:energy-estimates} and Propostion \ref{prop:interior-estimates} implies:
\begin{prop}[Total energy estimates]\label{prop:total-functional}
	Consider a smooth enough solution $ (\eta, \zeta) $ to system \eqref{eq:perturbation-0}, satisfying \eqref{bc:perturbation} and \eqref{initial:perturbation}. Suppose that \eqref{def:a-priori-assm} is satisfied with $ \omega \in (0,1) $ small enough, and that the expanding rate $ \beta_1 = \alpha_1 $ in \eqref{expandingrate} is large enough, we have
\begin{equation}\label{propest:total}
	\begin{aligned}
	& \mathcal E_0 (T) + \mathcal E_1(T) + \mathcal D_0(T) + \mathcal D_1(T) \leq C_{r_1,\sigma_1,\beta_1,\beta_2} \mathcal E_{in} \\
	& ~~~~ + C_{r_1,\sigma_1,\beta_1,\beta_2} \int \chi \bigl( \abs{\eta_0}{2} + \abs{x\eta_{0,x}}{2} \bigr) \,dx ,
	\end{aligned}
\end{equation}
for any positive time $ T \in (0,\infty) $. Here $ C_{r_1,\sigma_1, \beta_1,\beta_2 }$ is some constant depending on $ r_1, \sigma_1, \beta_1, \beta_2 $. 
\end{prop}

In the following, $ \beta_1, \beta_2 $ are fixed according to Proposition \ref{prop:total-functional}. 
In order to present the point-wise estimates properly, we split the estimates into four steps. Also, to shorten the notation, in the following, we use the constant $ C $ to denote a generic constant depending on $ r_1, \sigma_1, \beta_1, \beta_2 $, which may be different from line to line. Also, the symbol $ A \lesssim B $ is used to denote $ A \leq C B $ for the constant $ C $ as described. 

{\bf\par\noindent Step 1:} Point-wise estimate of $ \zeta $ in term of $ \mathfrak B $. 
After integrating \subeqref{eq:perturbation-0}{2} in $ \tau $-variable, one has
\begin{equation}\label{ene:101}
\begin{aligned}
	& \abs{\zeta}{} = \biggl| \zeta_0 - \overline p \bigl\lbrace  (1+\eta + x\eta_{x})^\gamma (1+\eta)^{2\gamma} -   (1+\eta_0 + x\eta_{0,x})^\gamma (1+\eta_0)^{2\gamma} \bigr\rbrace\\
	& ~~~~ + \dfrac{4}{3} \mu(\gamma-1) \int_0^\tau \alpha^{3\gamma - 1}\bigl\lbrack (1+\eta + x\eta_{x})(1+\eta)^2 \bigr\rbrack^\gamma \mathfrak B^2 \,d\tau \biggr|\\
	& ~~ \lesssim \abs{\zeta_0}{} +\overline p \bigl(  \abs{ \eta_0}{} + \abs{x\eta_{0,x}}{} + \abs{\eta}{} + \abs{x\eta_{x}}{} \bigr) + \int_0^\tau \alpha^{3\gamma-1} \abs{\mathfrak B}{2} \,d\tau ,
\end{aligned}
\end{equation}
which implies
\begin{equation}\label{ene:102}
\begin{aligned}
	& \norm{\zeta}{\Lnorm{\infty}} \lesssim \norm{\zeta_0}{\Lnorm{\infty}} + \norm{\eta}{\Lnorm{\infty}} + \norm{x\eta_x}{\Lnorm{\infty}} + \int_0^\tau \alpha^{3\gamma-1} \norm{\mathfrak B}{\Lnorm{\infty}}^2 \,d\tau
	\\
	& ~~~~ \lesssim \norm{\zeta_0}{\Lnorm{\infty}} + \norm{\eta}{\Lnorm{\infty}} + \norm{x\eta_x}{\Lnorm{\infty}} + \omega \int_0^\tau \alpha^{3\gamma-1- \sigma} \norm{\mathfrak B}{\Lnorm{\infty}} \,d\tau. 
\end{aligned}
\end{equation}
{\bf\par\noindent Step 2:} Point-wise estimate of $ \mathfrak B $. Notice, the right-hand side of \subeqref{eq:perturbation-0}{1} can be written as
\begin{equation*}
	\dfrac{4}{3} \mu \alpha^3 \biggl( \dfrac{\eta_\tau + x\eta_{x\tau}}{1+\eta + x\eta_x} + 2 \dfrac{\eta_\tau}{1+\eta} \biggr)_x = \dfrac{4}{3} \mu \alpha^3 \biggl( \dfrac{\bigl( x^3 (1+\eta)^3 \bigr)_{x\tau}}{\bigl( x^3 (1+\eta)^3 \bigr)_x} \biggr)_x. 
\end{equation*}
After multiplying \subeqref{eq:perturbation-0}{1} with $ x^3 (1+\eta)^3 $ and integrating the resultant in the interval $ (0,x) $, for $ x \in (0,1] $, it follows
\begin{equation}\label{ene:103}
	\begin{aligned}
		& \dfrac{4}{3} \mu \alpha^3 \underbrace{\biggl\lbrack \dfrac{\bigl( x^3 (1+\eta)^3 \bigr)_{x\tau}}{\bigl( x^3 (1+\eta)^3 \bigr)_x} \times x^3 (1+\eta)^3 - \int_0^x \dfrac{\bigl( x^3 (1+\eta)^3 \bigr)_{x\tau}}{\bigl( x^3 (1+\eta)^3 \bigr)_x} \times \bigl( x^3 (1+\eta)^3 \bigr)_x \,dx \biggr\rbrack}_{(i)} \\
		& ~~ = \alpha^{4-3\gamma} \underbrace{\biggl\lbrack \dfrac{\zeta \times x^3 (1+\eta)^3  }{\bigl\lbrack (1+\eta + x\eta_x) (1+\eta)^2 \bigr\rbrack^\gamma} - \int_0^x \dfrac{\zeta \times \bigl(x^3 (1+\eta)^3 \bigr)_x }{\bigl\lbrack (1+\eta + x\eta_x) (1+\eta)^2 \bigr\rbrack^\gamma} \,dx \biggr\rbrack}_{(ii)} \\
		& ~~~~ +\alpha  \underbrace{\int_0^x x^4\overline\rho (1+\eta) \partial_\tau^2 \eta \,dx}_{(iii)} + \alpha_\tau \underbrace{\int_0^x x^4\overline\rho (1+\eta) \partial_\tau \eta \,dx}_{(iv)} - \alpha^{4-3\gamma} \underbrace {\delta \int_0^x x^4 \overline\rho \eta (1+\eta)^2 \,dx}_{(v)}.
	\end{aligned}
\end{equation}
Then direct calculation shows
\begin{align*}
	& (i) = x^3 (1+\eta)^3 \mathfrak B,\\
	& \abs{(ii)}{}\lesssim x^3 (1+\eta)^3 \norm{\zeta}{\Lnorm{\infty}}, \\
	& \abs{(iii)}{} \lesssim x^3 (1+\eta)^3 \biggl( \int_0^x x^2 \overline\rho \abs{\partial_{\tau}^2 \eta}{2} \,dx \biggr)^{1/2} \lesssim x^3 (1+\eta)^3 \biggl( \int \chi x^2 \overline\rho \abs{\partial_{\tau}^2 \eta}{2} \,dx \\
	& ~~~~ + \int x^4 \overline\rho \abs{\partial_{\tau}^2 \eta}{2} \,dx \biggr)^{1/2} \lesssim x^3 (1+\eta)^3 \times ( \alpha^{-r_2/2} + \alpha^{r_4/2} ) ( \mathcal E_0 + \mathcal E_1)^{1/2} ,\\
	& \abs{(iv)}{} \lesssim x^3 (1+\eta)^3 \times ( \alpha^{-r_1/2} + \alpha^{-\mathfrak b/2} ) ( \mathcal E_0 + \mathcal E_1)^{1/2}, \\
	& \abs{(v)}{} \lesssim x^3(1+\eta)^3 \norm{\eta}{\Lnorm{\infty}}
\end{align*}
where we have applied \eqref{ene:009}. Consequently, \eqref{ene:103} yields
\begin{equation}\label{ene:104}
	\begin{aligned}
		& \norm{\mathfrak B}{\Lnorm{\infty}} \lesssim \alpha^{1-3\gamma} \bigl( \norm{\zeta}{\Lnorm{\infty}} + \norm{\eta}{\Lnorm{\infty}} \bigr) + \alpha^{r_4/2-2} (1+\beta_2) \bigl( \mathcal E_0 + \mathcal E_1 \bigr)^{1/2}.
	\end{aligned}
\end{equation}

{\bf\par\noindent Step 3:} Point-wise estimate of $ \norm{\alpha^{\sigma_1}\eta_\tau}{\Lnorm{\infty}} $ and $ \norm{\alpha^{\sigma_1}x\eta_{x\tau}}{\Lnorm{\infty}} $. After integrating \subeqref{eq:perturbation-0}{1} in the interval $ (x,1) $, for $ x \in (0,1) $, it follows
\begin{equation}\label{ene:105}
	\begin{aligned}
		& \dfrac{4}{3} \mu \alpha^3 \biggl\lbrack \dfrac{\eta_\tau + x\eta_{x\tau}}{1+\eta + x\eta_x} + 2 \dfrac{\eta_\tau}{1+\eta} \biggr\rbrack = 4\mu \alpha^3 \underbrace{\dfrac{\eta_\tau}{1+\eta}\Big|_{x=1}}_{(vi)} + \alpha^{4-3\gamma} \underbrace{\dfrac{\zeta}{\bigl\lbrack(1+\eta +x\eta_x)(1+\eta)^2 \bigr\rbrack^\gamma}}_{(vii)} \\
		& ~~~~ + \alpha^{4-3\gamma} \delta \underbrace{\int_x^1 \dfrac{x\overline\rho \eta}{1+\eta}\,dx }_{(viii)} - \alpha \underbrace{\int_x^1 \dfrac{x\overline\rho \partial_\tau^2 \eta}{(1+\eta)^2}\,dx}_{(ix)}- \alpha_\tau \underbrace{\int_x^1 \dfrac{x\overline\rho \partial_\tau \eta}{(1+\eta)^2}\,dx}_{(x)}.
	\end{aligned}
\end{equation}
Then, after applying the Sobolev imbedding inequality, H\"older's inequality, \eqref{ene:008} and \eqref{ene:009}, directly we have
\begin{align*}
	& \abs{(vi)}{} \lesssim \bigl( \int_{1/2}^1 ( \abs{\eta_\tau}{2} + \abs{\eta_{x\tau}}{2} ) \,dx \bigr)^{1/2} \lesssim \biggl( \int x^4 \overline\rho \abs{\eta_\tau}{2} \,dx + \int x^2 \bigl\lbrack (1+\eta) x\eta_{x\tau} - x\eta_x \eta_\tau\bigr\rbrack^2 \,dx \biggr)^{1/2} \\
	& ~~~~ \lesssim ( \alpha^{-r_1/2} + \alpha^{-\mathfrak a/2} ) \mathcal E_0^{1/2}, \\
	& \abs{(vii)}{} \lesssim \norm{\zeta}{\Lnorm{\infty}}, \\
	& \abs{(viii)}{} \lesssim  \norm{\eta}{\Lnorm{\infty}}, \\
	& \abs{(ix)}{} \lesssim \bigl( \int \chi x^2 \overline\rho \abs{\partial_\tau^2 \eta}{2} \,dx + \int x^4 \overline\rho \abs{\partial_\tau^2 \eta}{2} \,dx \bigr)^{1/2} \lesssim \alpha^{r_4/2} \bigl(\mathcal E_0 + \mathcal E_1\bigr)^{1/2} ,\\
	& \abs{(x)}{} \lesssim \bigl( \int \chi \abs{\partial_\tau \eta}{2} \,dx + \int x^4 \overline\rho \abs{\partial_\tau \eta}{2} \,dx \bigr)^{1/2} \lesssim ( \alpha^{-\mathfrak b/2} + \alpha^{-r_1/2} )\bigl(\mathcal E_0 + \mathcal E_1\bigr)^{1/2} .
\end{align*}
Consequently, \eqref{ene:105} yields
\begin{equation}\label{ene:106}
	\begin{aligned}
		& \norm{\dfrac{\eta_\tau + x\eta_{x\tau}}{1+\eta + x\eta_x}+ 2 \dfrac{\eta_\tau}{1+\eta}}{\Lnorm{\infty}} \lesssim (1+\beta_2)(\alpha^{-r_1/2} + \alpha^{-\mathfrak a /2} + \alpha^{r_4/2 - 2} 
		) \\ & ~~~~ \times \bigl( \mathcal E_0 + \mathcal E_1 \bigr)^{1/2} 
		 + \alpha^{1-3\gamma} \bigl( \norm{\zeta}{\Lnorm{\infty}} + \norm{\eta}{\Lnorm{\infty}} \bigr).
	\end{aligned}
\end{equation}
On the other hand, applying the fundamental theorem of calculus gives us the estimates
\begin{equation}\label{ene:107}
\begin{gathered}
	\norm{\eta}{\Lnorm{\infty}} \leq \norm{\eta_0}{\Lnorm{\infty}} + \int_0^\tau \norm{\eta_\tau}{\Lnorm{\infty}} \,d\tau, \\
	\norm{x\eta_x}{\Lnorm{\infty}} \leq \norm{x\eta_{0,x}}{\Lnorm{\infty}} + \int_0^\tau \norm{x\eta_{x\tau}}{\Lnorm{\infty}} \,d\tau.
\end{gathered}
\end{equation}
Therefore, combining \eqref{ene:102}, \eqref{ene:104}, \eqref{ene:106} and \eqref{ene:107} yields the following estimate
\begin{align*}
	& \norm{\eta_\tau}{\Lnorm{\infty}} + \norm{x\eta_{x\tau}}{\Lnorm{\infty}} \lesssim (1+\beta_2) ( \alpha^{r_4/2-2} + \alpha^{-r_1/2} + \alpha^{-\mathfrak a/2} ) \bigl( \mathcal E_0 + \mathcal E_1 \bigr)^{1/2}\\
	& ~~~~ + \alpha^{1-3\gamma}  \biggl( \norm{\zeta_0}{\Lnorm{\infty}} + \norm{\eta_0}{\Lnorm{\infty}} + \norm{x\eta_{0,x}}{\Lnorm{\infty}} \\
	& ~~~~ + \int_0^\tau (1+ \omega \alpha^{3\gamma - 1 - \sigma})( \norm{\eta_\tau}{\Lnorm{\infty}} + \norm{x\eta_x}{\Lnorm{\infty}} ) \,d\tau \biggr),
\end{align*}
or equivalently,
\begin{equation}\label{ene:108}
\begin{aligned}
	& \norm{\alpha^{\sigma_1} \eta_\tau}{\Lnorm{\infty}} + \norm{ \alpha^{\sigma_1} x\eta_{x\tau}}{\Lnorm{\infty}} \lesssim (1+\beta_2) ( \alpha^{\sigma_1 + r_4/2-2} + \alpha^{\sigma_1 -r_1/2} \\
	& ~~~~ + \alpha^{\sigma_1 -\mathfrak a/2} ) \bigl( \mathcal E_0 + \mathcal E_1 \bigr)^{1/2}
	 + \alpha^{1-3\gamma + \sigma_1} \bigl(  \norm{\zeta_0}{\Lnorm{\infty}} + \norm{\eta_0}{\Lnorm{\infty}} + \norm{x\eta_{0,x}}{\Lnorm{\infty}}\bigr) \\
	& ~~~~ + \sup\lbrace \alpha^{1-3\gamma + \sigma_1}, \alpha^{1-3\gamma}, \omega \alpha^{ - \sigma} \rbrace \int_0^\tau ( \norm{ \alpha^{\sigma_1} \eta_\tau}{\Lnorm{\infty}} + \norm{ \alpha^{\sigma_1} x\eta_x}{\Lnorm{\infty}} ) \,d\tau.
\end{aligned}
\end{equation}
Consequently, by applying Gr\"onwall's inequaltiy, we concludes from \eqref{ene:108} that
\begin{equation}\label{ene:109}
	\begin{aligned}
	& \norm{\alpha^{\sigma_1} \eta_\tau}{\Lnorm{\infty}} + \norm{\alpha^{\sigma_1}x\eta_{x\tau}}{\Lnorm{\infty}} \lesssim \norm{\zeta_0}{\Lnorm{\infty}} + \norm{\eta_0}{\Lnorm{\infty}} + \norm{x\eta_{0,x}}{\Lnorm{\infty}}\\
	& ~~~~ + \bigl(\mathcal E_0 + \mathcal E_1\bigr)^{1/2},
	\end{aligned}
\end{equation}
provided
\begin{gather*}
	\sigma_1 + r_4/2 - 2 < 0, \sigma_1 - r_1/2 < 0, \sigma_1 - \mathfrak a/2 < 0,\\
	1 - 3\gamma + \sigma_1 < 0,
\end{gather*}
or, after substituting the indices in \eqref{indices}
\begin{gather}
	2\sigma_1 < r_1. \label{constraint-90}
\end{gather}
%

{\bf\par\noindent Step 4:} Point-wise estimate of $ \norm{\alpha^{\sigma}\mathfrak B}{\Lnorm{\infty}} $. From \eqref{ene:102}, \eqref{ene:104} and \eqref{ene:107}, one can derive
\begin{align*}
	& \norm{\mathfrak B}{\Lnorm{\infty}} \lesssim (1+\beta_2) \alpha^{r_4/2-2} \bigl( \mathcal E_0 + \mathcal E_1 \bigr)^{1/2} + \alpha^{1-3\gamma} \biggl( \norm{\zeta_0}{\Lnorm{\infty}} + \norm{\eta_0}{\Lnorm{\infty}} \\
	& ~~~~ + \norm{x\eta_{0,x}}{\Lnorm{\infty}} 
	 + \int_0^\tau ( \norm{\eta_\tau}{\Lnorm{\infty}} + \norm{x\eta_{x\tau}}{\Lnorm{\infty}} )\,d\tau + \omega \int_0^\tau \alpha^{3\gamma - 1 -\sigma} \norm{\mathfrak B}{\Lnorm{\infty}} \,d\tau   \biggr),
\end{align*}
or equivalently, after substituting \eqref{ene:109},
\begin{equation}\label{ene:110}
	\begin{aligned}
		& \norm{\alpha^\sigma \mathfrak B}{\Lnorm{\infty}} \lesssim (1+\beta_2) \alpha^{\sigma + r_4/2 - 2}\bigl( \mathcal E_0 + \mathcal E_1\bigr)^{1/2} \\
		& ~~~~ + \alpha^{\sigma + 1 - 3\gamma} \biggl( \norm{\zeta_0}{\Lnorm{\infty}} + \norm{\eta_0}{\Lnorm{\infty}} + \norm{x\eta_{0,x}}{\Lnorm{\infty}} + \bigl( \mathcal E_0 + \mathcal E_1\bigr)^{1/2} \biggr) \\
		& ~~~~ + \max\lbrace \omega \alpha^{\sigma + 1 - 3\gamma}, \omega\alpha^{-\sigma} \rbrace \int_0^\tau \norm{\alpha^\sigma \mathfrak B}{\Lnorm{\infty}} \,d\tau .
	\end{aligned}
\end{equation}
From \eqref{ene:110}, as before, we conclude
\begin{equation}\label{ene:111}
	\norm{\alpha^\sigma \mathfrak B}{\Lnorm{\infty}} \lesssim \norm{\zeta_0}{\Lnorm{\infty}} + \norm{\eta_0}{\Lnorm{\infty}} + \norm{x\eta_{0,x}}{\Lnorm{\infty}} + \bigl( \mathcal E_0 + \mathcal E_1 \bigr)^{1/2},
\end{equation}
provided
\begin{gather*}
	\sigma + r_4/2 - 2 < 0, \sigma + 1 - 3\gamma < 0, \sigma > 0,
\end{gather*}
or, after substituting the indices in \eqref{indices}
\begin{equation}\label{constraint-91}
	2 - r_1 < 2\sigma_1.
\end{equation}
Notice, \eqref{constraint-90} and \eqref{constraint-91} yield the constraint in \eqref{constraint-000}.

We summarize the estimate above in the following proposition:
\begin{prop}[Closing the a priori estimate]\label{prop:point-wise-0}
With the indices given as in \eqref{indices}, as long as the constraint \eqref{constraint-000} is satisfied, we have
\begin{equation}\label{propest:point-wise-1}
	\begin{aligned}
		& \max \lbrace \norm{\eta}{\Lnorm{\infty}}, \norm{x\eta_x}{\Lnorm{\infty}}, \norm{\alpha^{\sigma_1} \eta_\tau}{\Lnorm{\infty}}, \norm{\alpha^{\sigma_1} x\eta_{x\tau}}{\Lnorm{\infty}}, \norm{\alpha^\sigma \mathfrak B}{\Lnorm{\infty}} \rbrace \\
		& ~~~~ \leq C_{r_1,\sigma_1,\beta_1,\beta_2} \bigl(\norm{\zeta_0}{\Lnorm{\infty}} + \norm{\eta_0}{\Lnorm{\infty}} + \norm{x\eta_{0,x}}{\Lnorm{\infty}} + \bigl( \mathcal E_0 + \mathcal E_1 \bigr)^{1/2} \bigr),
	\end{aligned}
\end{equation}
under the same assumptions as in Proposition \ref{prop:total-functional}. Here $ C_{r_1,\sigma_1,\beta_1,\beta_2} $ is a constant depending on $ r_1, \sigma_1, \beta_1, \beta_2 $. 
In particular, the constraint \eqref{constraint-000} requires 
\begin{equation}\label{constraint-000-gamma}
	\gamma > \dfrac{7}{6}, ~ \text{i.e.} ~ \gamma \in I_0,
\end{equation}
where $ I_0 $ is given in \eqref{thmfull:gamma-intervals}.
\end{prop}

Next, we are going to show the boundedness of the perturbation variable $ q $ for the pressure. In particular, this shows that if the self-similar solution has bounded entropy, so will the perturbed solution do. From the definition of $ \zeta $ in \eqref{def:weighted-q}, this is equivalent to show the boundedness of $ \zeta / \overline p $. 
\begin{prop}[Boundedness of  $ \zeta / \overline p $]\label{prop:boundedness-entropy} With the indices given as in \eqref{indices}, as long as the constraint
\begin{equation}\label{constraint-001}
	\begin{gathered}
	\max\lbrace 3\gamma - 3 - r_1, 2 - r_1 \rbrace < 2 \sigma_1 < r_1 < \min\lbrace 6 \gamma - 6, 2\rbrace, \\
	\dfrac{7}{6} < \gamma < \dfrac{7}{3}, ~ \text{i.e.} ~\gamma \in I_1
	\end{gathered}
\end{equation}
is satisfied, we have, $ \exists \varepsilon_0 \in (0,1) $, 
\begin{equation}\label{propest:point-wise-2}
	\begin{aligned}
	& \norm{\dfrac{\zeta}{\overline p}}{\Lnorm{\infty}(\varepsilon_0,1)} \leq C_{r_1,\sigma_1,\beta_1,\beta_2} \biggl\lbrace  \norm{\dfrac{\zeta_0}{\overline p}}{\Lnorm{\infty}} + \norm{\eta_0}{\Lnorm{\infty}} + \norm{x\eta_{0,x}}{\Lnorm{\infty}} + \norm{\dfrac{\zeta_0}{\overline p}}{\Lnorm{\infty}}^2 \\
	& ~~~~ + \norm{\eta_0}{\Lnorm{\infty}}^2 
	 + \norm{x\eta_{0,x}}{\Lnorm{\infty}}^2 
	+  \bigl( \mathcal E_0 + \mathcal E_1\bigr)^{1/2}+\bigl( \mathcal E_0 + \mathcal E_1\bigr)\biggr\rbrace,
	\end{aligned}
\end{equation}
and
\begin{equation}\label{propest:point-wise-21}
	\begin{aligned}
		& \norm{\dfrac{\zeta}{\overline p}}{\Lnorm{\infty}} \leq C_{r_1,\sigma_1,\beta_1,\beta_2} \alpha^{3\gamma - 3} \biggl\lbrace  \norm{\dfrac{\zeta_0}{\overline p}}{\Lnorm{\infty}} + \norm{\eta_0}{\Lnorm{\infty}} + \norm{x\eta_{0,x}}{\Lnorm{\infty}} + \norm{\dfrac{\zeta_0}{\overline p}}{\Lnorm{\infty}}^2 \\
	& ~~~~ + \norm{\eta_0}{\Lnorm{\infty}}^2 
	 + \norm{x\eta_{0,x}}{\Lnorm{\infty}}^2 
	+  \bigl( \mathcal E_0 + \mathcal E_1\bigr)^{1/2} + \bigl( \mathcal E_0 + \mathcal E_1 \bigr) \biggr\rbrace,
	\end{aligned}
\end{equation}
provided that $ \norm{q_0}{\Lnorm{\infty}} $ is small enough. 
If in addition,
\begin{equation}\label{constraint-002}
	\begin{gathered}
	\max \lbrace 3\gamma - 1 - r_1, 2 - r_1 \rbrace < 2\sigma_1 < r_1 < \min\lbrace 6 \gamma - 6, 2\rbrace, \\
	\dfrac{11}{9} < \gamma < \dfrac{5}{3}, ~ \text{i.e.} ~ \gamma \in I_2,
	\end{gathered}
\end{equation}
we have
\begin{equation}\label{propest:point-wise-3}
	\begin{aligned}
	& \norm{\dfrac{\zeta}{\overline p}}{\Lnorm{\infty}} \leq C_{r_1,\sigma_1,\beta_1,\beta_2}\biggl\lbrace  \norm{\dfrac{\zeta_0}{\overline p}}{\Lnorm{\infty}} + \norm{\eta_0}{\Lnorm{\infty}} + \norm{x\eta_{0,x}}{\Lnorm{\infty}} + \norm{\dfrac{\zeta_0}{\overline p}}{\Lnorm{\infty}}^2 \\
	& ~~~~ + \norm{\eta_0}{\Lnorm{\infty}}^2 
	 + \norm{x\eta_{0,x}}{\Lnorm{\infty}}^2 
	+  \bigl( \mathcal E_0 + \mathcal E_1\bigr)^{1/2} + \bigl( \mathcal E_0 + \mathcal E_1\bigr) \biggr\rbrace,
	\end{aligned}
\end{equation}
provided that $ \norm{q_0}{\Lnorm{\infty}} $ is small enough.
Here $ C_{r_1,\sigma_1,\beta_1,\beta_2} $ is a constant depending on $ r_1, \sigma_1, \beta_1, \beta_2 $, and $ C_{r_1,\sigma_1,\beta_1,\beta_2,\varpi} $ is a constant depending on $ r_1, \sigma_1, \beta_1, \beta_2, \varpi $.

\end{prop}
\begin{proof} Again, we separate the proof into three steps. 
	{\par\noindent\bf Step 1:} Estimate of $ \norm{\zeta}{\Lnorm{\infty}} $. After substituting \eqref{ene:104} and \eqref{propest:point-wise-1} in \eqref{ene:102}, one has
	\begin{equation}\label{ene:112}
		\begin{aligned}
			& \norm{\zeta}{\Lnorm{\infty}} \lesssim \norm{\zeta_0}{\Lnorm{\infty}} + \norm{\eta_0}{\Lnorm{\infty}} + \norm{x\eta_{0,x}}{\Lnorm{\infty}} + \bigl(\mathcal E_0 + \mathcal E_1 \bigr)^{1/2} \\
			& ~~ + \int_0^\tau \biggl( \alpha^{1-3\gamma}  \norm{\zeta}{\Lnorm{\infty}}^2 + \alpha^{1-3\gamma} (\norm{\zeta_0}{\Lnorm{\infty}}^2 + \norm{\eta_0}{\Lnorm{\infty}}^2 + \norm{x\eta_{0,x}}{\Lnorm{\infty}}^2 + \mathcal E_0 + \mathcal E_1) \\
			& ~~ + \alpha^{3\gamma + r_4 - 5}(\mathcal E_0 + \mathcal E_1)  \biggr) \,d\tau.
		\end{aligned}
	\end{equation}
	Then, for $ \norm{\zeta_0}{\Lnorm{\infty}} \leq 2 \norm{q_0}{\Lnorm{\infty}} $ small enough, applying continuity arguments with \eqref{ene:112} yields
	\begin{equation}\label{ene:1161}
		\begin{aligned}
		& \norm{\zeta}{\Lnorm{\infty}} \lesssim \norm{\zeta_0}{\Lnorm{\infty}} + \norm{\eta_0}{\Lnorm{\infty}} + \norm{x\eta_{0,x}}{\Lnorm{\infty}} + \norm{\zeta_0}{\Lnorm{\infty}}^2 + \norm{\eta_0}{\Lnorm{\infty}}^2 + \norm{x\eta_{0,x}}{\Lnorm{\infty}}^2\\
		& ~~~~ + \bigl( \mathcal E_0 + \mathcal E_1\bigr)^{1/2} + \bigl( \mathcal E_0 + \mathcal E_1 \bigr),
		\end{aligned}
	\end{equation}
	provided
	\begin{gather*}
		3\gamma + r_4 - 5 < 0,
	\end{gather*}
	or, after substituting \eqref{indices},
	\begin{equation}\label{constraint-010}
		3\gamma - 1 - r_1 < 2\sigma_1.
	\end{equation}

	{\par\noindent\bf Step 2:} Estimate of $ \overline p^{-1/2} \mathfrak B $ near the boudnary $ \lbrace x = 1\rbrace $. Consider $ x_\varepsilon \in (1/2,1) $ such that $ \max\lbrace \overline p(x_\varepsilon)^{1/2}, 1-x_{\varepsilon} \rbrace \leq \varepsilon $, $ \varepsilon \in (0,1) $. We denote $ \Lnorm{\infty}_\varepsilon : = \Lnorm{\infty}(x_\varepsilon, 1) $. We integrate \subeqref{eq:perturbation-0}{1} in the interval $ (x,1) $, for $ x \in(0,1) $ to obtain, similar to \eqref{ene:105},
	\begin{equation*}\tag{\ref{ene:105}'} \label{ene:113}
	\begin{aligned}
		& \dfrac{4}{3} \mu \alpha^3 \mathfrak B = 4 \mu \alpha^3 \underbrace{\int_x^1 \biggl( \dfrac{\eta_\tau}{1+\eta} \biggr)_x \,dx}_{(vi')} + \alpha^{4-3\gamma} \underbrace{\dfrac{\zeta}{\bigl\lbrack(1+\eta +x\eta_x)(1+\eta)^2 \bigr\rbrack^\gamma}}_{(vii)} \\
		& ~~~~ + \alpha^{4-3\gamma} \delta \underbrace{\int_x^1 \dfrac{x\overline\rho \eta}{1+\eta}\,dx }_{(viii)} - \alpha \underbrace{\int_x^1 \dfrac{x\overline\rho \partial_\tau^2 \eta}{(1+\eta)^2}\,dx}_{(ix)}- \alpha_\tau \underbrace{\int_x^1 \dfrac{x\overline\rho \partial_\tau \eta}{(1+\eta)^2}\,dx}_{(x)}. 
	\end{aligned}
	\end{equation*}
	Notice, $ \overline p $ is non-increasing and 
	\begin{equation*}
		\biggl( \dfrac{\eta_\tau}{1+\eta} \biggr)_x = \dfrac{1+\eta + x\eta_x}{x(1+\eta)} \mathfrak B, ~~ \delta \int_x^1 x \overline \rho \,dx = \overline p, 
	\end{equation*}
	from \subeqref{eq:self-similar-sol}{2} and \subeqref{eq:self-similar-sol}{3}. Therefore, we estimate the right-hand side of \eqref{ene:113} as follows: for $ x \in (x_\varepsilon, 1) \subset (1/2,1) $,
	\begin{align*}
		& \abs{(vi')}{} \lesssim (1-x) \overline p^{1/2} \norm{\overline p^{-1/2} \mathfrak B}{\Lnorm{\infty}_\varepsilon} \lesssim \varepsilon \overline p^{1/2} \norm{\overline p^{-1/2} \mathfrak B}{\Lnorm{\infty}_\varepsilon}, \\
		& \abs{(vii)}{} \lesssim \overline p \norm{\dfrac{\zeta}{\overline p}}{\Lnorm{\infty}_\varepsilon}, \\
		& \abs{(viii)}{} \lesssim  \int_x^1 x \overline\rho \,dx \times \norm{\eta}{\Lnorm{\infty}} \lesssim \overline p \times \norm{\eta}{\Lnorm{\infty}} \\
		& \abs{(ix)}{} \lesssim \bigl( \int_x^1 x \overline\rho \,dx \bigr)^{1/2} \bigl( \int_0^1 x^4 \overline\rho \abs{\partial_\tau^2 \eta}{2} \,dx \bigr)^{1/2} \lesssim \overline p^{1/2} \alpha^{-r_2/2} \mathcal E_0^{1/2}, \\
		& \abs{(x)}{} \lesssim \bigl( \int_x^1 x \overline\rho \,dx \bigr)^{1/2} \bigl( \int_0^1 x^4 \overline\rho \abs{\partial_\tau \eta}{2} \,dx \bigr)^{1/2} \lesssim \overline p^{1/2} \alpha^{-r_1/2} \mathcal E_0^{1/2}.
	\end{align*}
	Consequently, \eqref{ene:113} yields, for $ \varepsilon \in (0,\varepsilon_0] $ with $ \varepsilon_0 \in (0,1) $ small enough, 
	\begin{equation}\label{ene:114}
			\norm{\overline p^{-1/2} \mathfrak B}{\Lnorm{\infty}_\varepsilon} \lesssim  \alpha^{1-3\gamma} \varepsilon \bigl( \norm{\dfrac{\zeta}{\overline p}}{\Lnorm{\infty}_\varepsilon} + \norm{\eta}{\Lnorm{\infty}} \bigr) + ( \beta_2 \alpha^{-2-r_1/2} + \alpha^{-2-r_2/2} ) \mathcal E_0^{1/2}.
	\end{equation}
	
{\par\noindent\bf Step 3:} Estimate of $ \dfrac{\zeta}{\overline p} $ near the boundary $ \lbrace x =1 \rbrace $.  We fix $ \varepsilon = \varepsilon_0 \in (0,1) $, where $ \varepsilon_0 $ is given before so that \eqref{ene:114} holds. In particular, $ \Lnorm{\infty}_{\varepsilon_0} = \Lnorm{\infty}(x_{\varepsilon_0},1) $.  From \eqref{ene:101}, we have
\begin{equation*}
	\abs{\dfrac{\zeta}{\overline p}}{} \lesssim \abs{\dfrac{\zeta_0}{\overline p}}{} +  \abs{ \eta_0}{} + \abs{x\eta_{0,x}}{} + \abs{\eta}{} + \abs{x\eta_{x}}{}+ \int_0^\tau \alpha^{3\gamma-1} \abs{ \overline p^{-1/2} \mathfrak B}{2} \,d\tau.
\end{equation*}
Consequently, after substituting \eqref{propest:point-wise-1} and \eqref{ene:114} into the above expression, we obtain
\begin{equation}\label{ene:115}
\begin{aligned}
	& \norm{\dfrac{\zeta}{\overline p}}{\Lnorm{\infty}_{\varepsilon_0}} \lesssim \norm{\dfrac{\zeta_0}{\overline p}}{\Lnorm{\infty}} +  \norm{\zeta_0}{\Lnorm{\infty}} + \norm{\eta_0}{\Lnorm{\infty}} + \norm{x\eta_{0,x}}{\Lnorm{\infty}} + (\mathcal E_0 + \mathcal E_1)^{1/2} \\
	& ~~~~ + \int_0^\tau \alpha^{1-3\gamma}  \norm{\dfrac{\zeta}{\overline p}}{\Lnorm{\infty}_{\varepsilon_0}}^2\,d\tau 
	+ \int_0^\tau \alpha^{1-3\gamma} \bigl( \norm{\zeta_0}{\Lnorm{\infty}}^2 + \norm{\eta_0}{\Lnorm{\infty}}^2 + \norm{x\eta_{0,x}}{\Lnorm{\infty}}^2 \\
	& ~~~~ ~~~~ + \mathcal E_0 + \mathcal E_1 \bigr) \,d\tau 
	 + \int_0^\tau (\beta_2^2 + 1) ( \alpha^{3\gamma -r_1 - 5} + \alpha^{3\gamma - r_2 - 5})  \mathcal E_0 \,d\tau.
\end{aligned}
\end{equation}
Then, for $ \norm{\dfrac{\zeta_0}{\overline p}}{\Lnorm{\infty}} \leq 2 \norm{q_0}{\Lnorm{\infty}} $ small enough, applying continuity arguments with \eqref{ene:115} yields
\begin{equation}\label{ene:116}
	\begin{aligned}
	& \norm{\dfrac{\zeta}{\overline p}}{\Lnorm{\infty}_{\varepsilon_0}} \lesssim \norm{\dfrac{\zeta_0}{\overline p}}{\Lnorm{\infty}} + \norm{\eta_0}{\Lnorm{\infty}} + \norm{x\eta_{0,x}}{\Lnorm{\infty}} + \norm{\dfrac{\zeta_0}{\overline p}}{\Lnorm{\infty}}^2 + \norm{\eta_0}{\Lnorm{\infty}}^2 \\
	& ~~~~ + \norm{x\eta_{0,x}}{\Lnorm{\infty}}^2 
	+  \bigl( \mathcal E_0 + \mathcal E_1\bigr)^{1/2} + \bigl( \mathcal E_0 + \mathcal E_1\bigr),
	\end{aligned}
\end{equation}
provided
\begin{gather*}
	3\gamma - r_1 - 5 <0,  ~ 3\gamma - r_2 - 5 <0,
\end{gather*}
or, after substituting \eqref{indices}
	\begin{equation}\label{constraint-011}
		3\gamma - 3 - r_1 < 2\sigma_1.
	\end{equation}
	In this case, we can only derive, from \eqref{ene:102} and \eqref{propest:point-wise-1}
	\begin{equation}\label{ene:1163}
		\norm{\zeta}{\Lnorm{\infty}} \lesssim \alpha^{3\gamma - 1 - 2\sigma} \biggl\lbrace \norm{\zeta_0}{\Lnorm{\infty}} + \norm{\eta_0}{\Lnorm{\infty}} + \norm{x\eta_{0,x}}{\Lnorm{\infty}} + \bigl(\mathcal E_0 + \mathcal E_1 \bigr)^{1/2} \biggr\rbrace.
	\end{equation}

Consequently, from \eqref{ene:116}, \eqref{constraint-011} and \eqref{ene:1163}, we have shown \eqref{propest:point-wise-2} and \eqref{propest:point-wise-21}; from \eqref{ene:1161}, \eqref{constraint-010}, \eqref{ene:116}, we have shown \eqref{propest:point-wise-3}. 
%
\end{proof}

\section{Regularity}\label{sec:regularity}
In this section, we focus on the regularity estimates. In particular, we aim at obtaining the $ H^1 $ norm of $ \zeta $, and the $ H^2(0,1) $ norm of $ x\eta_\tau $, which yields the $ H^2(0,1) $ norm of $ x\eta $ via the fundamental theorem of calculus. With such regularity in hand, system \eqref{eq:perturbation-0} holds almost everywhere for $ (x,\tau) \in (0,1) \times (0,T) $ for any given $ T \in (0,\infty) $,  and therefore the solution we obtained is a strong solution. To shorten the notation in the following, we use
\begin{equation}\label{initial-regularity}
	\begin{aligned}
		& \mathfrak G_0: = \norm{\zeta_0}{\Lnorm{\infty}} + \norm{\eta_0}{\Lnorm{\infty}} + \norm{x\eta_{0,x}}{\Lnorm{\infty}} \\
		& ~~~~ + \norm{\zeta_{0,x}}{\Lnorm{2}} + \norm{\eta_{0,x}}{\Lnorm{2}} + \norm{x\eta_{0,x}}{\Lnorm{2}}, 
	\end{aligned}
\end{equation}
to denote the initial date for the regularity estimates, below.

\begin{prop}\label{prop:regularity}
Under the same assumptions as in Proposition \ref{prop:energy-estimates}, we have
\begin{equation}\label{ene:205}
	\begin{aligned}
		& \sup_{0 \leq  \tau \leq  T} \lbrace \norm{\eta_{x}(\tau)}{\Lnorm{2}},  \norm{x\eta_{xx}(\tau)}{\Lnorm{2}},  \norm{\eta_{x\tau}(\tau)}{\Lnorm{2}},  \norm{x\eta_{xx\tau}(\tau)}{\Lnorm{2}}\rbrace \\
		& ~~~~ \leq C_{r_1, r_2, \beta_1, \beta_2} \bigl( \mathfrak G_0 + \mathcal E_0^{1/2}(T) + \mathcal E_1^{1/2}(T) \bigr),
	\end{aligned}
\end{equation}
and
\begin{equation}\label{ene:206}
	\sup_{0\leq \tau \leq T}\norm{\zeta_x(\tau)}{\Lnorm{2}} \leq C_{r_1, r_2, \beta_1, \beta_2} \bigl\lbrace \mathfrak G_0 + \alpha^{3\gamma-1} \bigl( \mathcal E_0^{1/2}(T) + \mathcal E_1^{1/2}(T)\bigr) \bigr\rbrace. 
\end{equation}
Here $ C_{r_1, r_2, \beta_1, \beta_2} $ is a constant depending on $r_1, r_2, \beta_1, \beta_2 $.
\end{prop}
\begin{proof}
We are going to apply Lemma \ref{lm:estimates-of-relative-entropy}. In the following, we assume $ \omega $ is small enough so that  estimates \eqref{lm:s001} and \eqref{lm:s002} are valid for $ h = \eta $. Notice, the right-hand side of \subeqref{eq:perturbation-0}{1} can be written as,
\begin{align*}
	& \dfrac{4}{3} \mu \alpha^3 \biggl( \dfrac{\eta_\tau + x\eta_{x\tau}}{1+\eta+x\eta_x} + 2 \dfrac{\eta_\tau}{1+\eta} \biggr)_x = \dfrac{4}{3} \mu \alpha^3 \bigl( \ln (x^3 (1+\eta)^3)_x \bigr)_{x\tau} \\
	& ~~~~ = \dfrac{4}{3} \mu \alpha^3 \bigl( \log \bigl\lbrack (1+\eta)^2(1+\eta+x\eta_x) \bigr\rbrack  \bigr)_{x\tau} = \dfrac{4}{3} \mu \alpha^3  \bigl( \mathfrak H(\eta) \bigr)_{x\tau},
\end{align*}
where $ \mathfrak H (\cdot) $ is defined in \eqref{def:relative-entropy-functional}. Then directly from \subeqref{eq:perturbation-0}{1}, one can derive
\begin{equation}\label{ene:200}
\begin{aligned}
	& \alpha^3 \norm{\bigl( \mathfrak H (\eta) \bigr)_{x\tau}}{\Lnorm{2}} \lesssim \alpha \norm{x \overline \rho \partial_{\tau}^2 \eta}{\Lnorm{2}} + \alpha_\tau \norm{x \overline \rho \partial_{\tau} \eta}{\Lnorm{2}} \\
	& ~~~~ + \alpha^{4-3\gamma} \norm{x\overline\rho \eta}{\Lnorm{2}} + \alpha^{4-3\gamma} \norm{\zeta_x}{\Lnorm{2}} +\alpha^{4-3\gamma}  \norm{\zeta}{\Lnorm{\infty}} \bigl( \norm{\eta_x}{\Lnorm{2}} + \norm{x\eta_{xx}}{\Lnorm{2}}  \bigr)\\
	& \lesssim \alpha \biggl\lbrace \bigl( \int x^4 \overline \rho \abs{\partial_{\tau\tau} \eta}{2} \,dx \bigr)^{1/2}  + \bigl(  \int \chi x^2 \overline \rho \abs{\partial_{\tau\tau} \eta}{2} \,dx \bigr)^{1/2} + \beta_2  \bigl( \int x^4 \overline \rho \abs{\partial_{\tau} \eta}{2} \,dx \bigr)^{1/2} \\
	& ~~~~ + \beta_2 \bigl( \int \chi \abs{\partial_{\tau} \eta}{2} \,dx \bigr)^{1/2}  \biggr\rbrace + \alpha^{4-3\gamma} \norm{\eta}{\Lnorm{\infty}} + \alpha^{4-3\gamma} \norm{\zeta_x}{\Lnorm{2}} \\
	& ~~~~ + \alpha^{4-3\gamma} \norm{\zeta}{\Lnorm{\infty}} \norm{\bigl( \mathfrak H(\eta) \bigr)_x}{\Lnorm{2}},
\end{aligned}
\end{equation}
where we have applied \eqref{lm:s001}. 
Simultaneously, notice
\begin{gather*}
	\mathfrak B = \bigl( \mathfrak H (\eta) \bigr)_\tau - 3 \dfrac{\eta_\tau}{1+\eta}, \\
	\bigl( \mathfrak B \bigr)_x = \bigl( \mathfrak H (\eta) \bigr)_{x\tau} - 3 \dfrac{\eta_{x\tau}}{1+\eta} + 3 \dfrac{\eta_{\tau} \eta_x}{(1+\eta)^2}, \\
	(1+\eta)^2 ( 1+ \eta +x\eta_x ) = e^{\mathfrak H(\eta)},
\end{gather*}
and \subeqref{eq:perturbation-0}{2} can be written as
\begin{align*}
	\partial_\tau \zeta + \overline p \bigl(e^{\gamma \mathfrak H(\eta)} \bigr)_\tau = \dfrac{4}{3} \mu (\gamma-1) \alpha^{3\gamma-1} e^{ \gamma \mathfrak H(\eta)} \mathfrak B^2.
\end{align*}
Therefore, after taking spatial derivative to \subeqref{eq:perturbation-0}{2}, we have
\begin{equation}\label{ene:201}
	\begin{aligned}
		& \partial_\tau \zeta_{x} =  \delta x \overline\rho   \bigl( e^{\gamma \mathfrak H(\eta)}\bigr)_\tau  - \overline p  \bigl( e^{\gamma \mathfrak H(\eta)} \bigr)_{x\tau} 
		 + \dfrac{4}{3} \mu \gamma (\gamma-1) \alpha^{3\gamma-1}e^{\gamma \mathfrak H(\eta) }  \mathfrak B^2  \bigl( \mathfrak H(\eta) \bigr)_x \\
		& ~~~~ + \dfrac{8}{3} \mu (\gamma - 1) \alpha^{3\gamma-1} e^{\gamma \mathfrak H(\eta)} \mathfrak B \biggl( \bigl( \mathfrak H(\eta) \bigr)_{x\tau} - 3  \dfrac{\eta_{x\tau}}{1+\eta} + 3 \dfrac{\eta_{\tau} \eta_x}{(1+\eta)^2} \biggr),
	\end{aligned}
\end{equation}
and consequently, after integrating \eqref{ene:201} in the temporal variable and taking $ L^2 (0,1) $ norm to the resultant, it follows, with the help of Minkowski's inequality, 
\begin{equation}\label{ene:202}
\begin{aligned}
	& \norm{\zeta_x}{\Lnorm{2}} \lesssim \norm{\zeta_{0,x}}{\Lnorm{2}} + \norm{\delta x\overline \rho e^{\gamma \mathfrak H(\eta_0)}}{\Lnorm{2}} + \norm{ \overline p \bigl( e^{\gamma \mathfrak H(\eta_0)} \bigr)_x }{\Lnorm{2}} \\
	& ~~~~ +  \norm{\delta x\overline \rho e^{\gamma \mathfrak H(\eta)}}{\Lnorm{2}} + \norm{ \overline p \bigl( e^{\gamma \mathfrak H(\eta)} \bigr)_x }{\Lnorm{2}} + \int_0^\tau \alpha^{3\gamma-1} \norm{e^{\gamma \mathfrak H(\eta) }  \mathfrak B^2  \bigl( \mathfrak H(\eta) \bigr)_x}{\Lnorm{2}} \,d\tau \\
	& ~~~~ + \int_0^\tau \alpha^{3\gamma-1} \norm{e^{\gamma \mathfrak H(\eta)} \mathfrak B \biggl( \bigl( \mathfrak H(\eta) \bigr)_{x\tau} - 3  \dfrac{\eta_{x\tau}}{1+\eta} + 3 \dfrac{\eta_{\tau} \eta_x}{(1+\eta)^2} \biggr)}{\Lnorm{2}} \,d\tau \\
	& \lesssim \mathfrak G_0 + \norm{\eta}{\Lnorm{\infty}} + \norm{x\eta_x}{\Lnorm{\infty}} + \norm{\bigl( \mathfrak H(\eta) \bigr)_x}{\Lnorm{2}} \\
	& ~~~~ + \int_0^\tau \alpha^{3\gamma-1} \bigl( \omega \alpha^{-2\sigma}\norm{\bigl( \mathfrak H(\eta) \bigr)_{x}}{\Lnorm{2}} + \omega \alpha^{-\sigma}  \norm{\bigl( \mathfrak H(\eta) \bigr)_{x\tau}}{\Lnorm{2}} + \omega \alpha^{-\sigma} \norm{\eta_{x\tau}}{\Lnorm{2}} \\
	& ~~~~ + \omega \alpha^{-\sigma_1 - \sigma} \norm{\eta_x}{\Lnorm{2}}  \bigr)\,d\tau. 
\end{aligned}	
\end{equation}
Then, after plugging \eqref{ene:202} in \eqref{ene:200}, and applying \eqref{lm:s001}, \eqref{lm:s002}, \eqref{ene:102}, \eqref{propest:point-wise-1}, we have 
\begin{equation}\label{ene:203}
\begin{aligned}
	& \norm{\bigl( \mathfrak H(\eta) \bigr)_{x\tau}}{\Lnorm{2}} \lesssim (1+\beta_2)\alpha^{-2} \bigl( \alpha^{-r_1/2} + \alpha^{-r_2/2} + \alpha^{ r_4/2} + \alpha^{- \mathfrak b/2} \bigr) \bigl(\mathcal E_0^{1/2} + \mathcal E_1^{1/2}\bigr) \\
	& ~~~~ + \alpha^{1-3\gamma} \bigl( \mathfrak G_0 + \mathcal E_0^{1/2} + \mathcal E_1^{1/2} \bigr)  \\
	& ~~~~ + \alpha^{1-3\gamma} \bigl( \norm{\bigl( \mathfrak H(\eta) \bigr)_x}{\Lnorm{2}} + \omega \int_0^\tau \alpha^{3\gamma-1-\sigma} ( \norm{\bigl( \mathfrak H(\eta)\bigr)_x}{\Lnorm{2}} +  \norm{\bigl( \mathfrak H(\eta) \bigr)_{x\tau}}{\Lnorm{2}} ) \,d\tau \bigr)  \\
	& ~~~~ + \omega \alpha^{1-3\gamma}    \max\lbrace \beta_1^{-1} \alpha^{3\gamma-1-2\sigma}, \tau, 1\rbrace \norm{\bigl( \mathfrak H (\eta) \bigr)_x}{\Lnorm{2}} \lesssim (1+\beta_2) \alpha^{r_4/2 - 2}\bigl( \mathcal E_0^{1/2} + \mathcal E_1^{1/2} \bigr)\\
	& ~~~~ + \alpha^{1-3\gamma} \bigl( \mathfrak G_0 + \mathcal E_0^{1/2} + \mathcal E_1^{1/2} \bigr) + \max\lbrace \alpha^{1-3\gamma}, \omega \beta_1^{-1} \alpha^{- \sigma}, \omega \alpha^{1-3\gamma} \tau \rbrace \sup_{\tau} \norm{\bigl( \mathfrak H(\eta) \bigr)_x}{\Lnorm{2}} \\
	& ~~~~ + \omega \max\lbrace \beta_1^{-1} \alpha^{-\sigma} , \alpha^{1-3\gamma} \tau \rbrace \sup_\tau \norm{\bigl( \mathfrak H(\eta) \bigr)_{x\tau}}{\Lnorm{2}} ,
\end{aligned}
\end{equation}
where we have applied inequality \eqref{ineq:integral-time-weight}.
Then, by noticing
\begin{gather*}
	\alpha^{1-3\gamma} \tau \lesssim \alpha^{2-3\gamma},\\
	\sup_\tau \norm{\bigl( \mathfrak H(\eta) \bigr)_x}{\Lnorm{2}} \lesssim \mathfrak G_0 + \int_0^\tau \norm{\bigl( \mathfrak H(\eta) \bigr)_{x\tau}}{\Lnorm{2}}\,d\tau,
\end{gather*}
applying Gr\"onwall's inequality to \eqref{ene:203} yields
\begin{equation}\label{ene:204}
	\sup_\tau \lbrace \norm{\bigl( \mathfrak H(\eta) \bigr)_x}{\Lnorm{2}} , \norm{\bigl( \mathfrak H(\eta) \bigr)_{x\tau}}{\Lnorm{2}} \rbrace \lesssim \mathfrak G_0 + \mathcal E_0^{1/2} + \mathcal E_1^{1/2},
\end{equation}
for $ \omega $ small enough. 
Together with the relative entropy estimate in Lemma \ref{lm:estimates-of-relative-entropy}, \eqref{ene:204} yields \eqref{ene:205}. 
Meanwhile, directly from \eqref{ene:202}, \eqref{ene:206} follows.
\end{proof}

\paragraph{Acknowledgement} 
The work of X. L. is supported by the ONR grant N00014-15-1-2333.


\end{document}